\documentclass{amsart}
\usepackage{amsfonts,amscd,amsthm,amsgen,amsmath,amssymb}
\usepackage[all]{xy}
\usepackage{epsfig,color}
\usepackage[vcentermath]{youngtab}
\newtheorem{theorem}{Theorem}[section]
\newtheorem{lemma}[theorem]{Lemma}
\newtheorem{corollary}[theorem]{Corollary}
\newtheorem{proposition}[theorem]{Proposition}

\theoremstyle{definition}
\newtheorem{definition}[theorem]{Definition}
\newtheorem{example}[theorem]{Example}

\theoremstyle{remark}
\newtheorem{remark}[theorem]{Remark}

\numberwithin{equation}{section}

\begin{document}

\title[A gluing formula for families Seiberg-Witten invariants]{A gluing formula for families Seiberg-Witten invariants}
\author{David Baraglia, Hokuto Konno}

\address{School of Mathematical Sciences, The University of Adelaide, Adelaide SA 5005, Australia}
\address{Graduate School of Mathematical Sciences, the University of Tokyo, 3-8-1 Komaba, Meguro, Tokyo 153-8914, Japan}

\email{david.baraglia@adelaide.edu.au}
\email{hkonno@ms.u-tokyo.ac.jp}



\date{\today}


\begin{abstract}
We prove a gluing formula for the families Seiberg-Witten invariants of families of $4$-manifolds obtained by fibrewise connected sum. Our formula expresses the families Seiberg-Witten invariants of such a connected sum family in terms of the ordinary Seiberg-Witten invariants of one of the summands, under certain assumptions on the families. We construct some variants of the families Seiberg-Witten invariants and prove the gluing formula also for these variants. One variant incorporates a twist of the families moduli space using the charge conjugation symmetry of the Seiberg-Witten equations. The other variant is an equivariant Seiberg-Witten invariant of smooth group actions. We consider several applications of the gluing formula including: obstructions to smooth isotopy of diffeomorpihsms, computation of the mod $2$ Seiberg-Witten invariants of spin structures, relations between mod $2$ Seiberg-Witten invariants of $4$-manifolds and obstructions to the existence of invariant metrics of positive scalar curvature for smooth group actions on $4$-manifolds.
\end{abstract}

\maketitle


\section{Introduction}

The Seiberg-Witten invariants of smooth $4$-manifolds admit an extension to the parametrised setting. This means that instead of a single $4$-manifold, we consider the moduli space of solutions to the Seiberg-Witten equations for a {\em family} of smooth $4$-manifolds. Here a family of smooth $4$-manfolds means a locally trivial fibre bundle over a base space whose fibres are a fixed compact smooth $4$-manifold $X$ equipped with a Spin$^c$ structure, and the transition functions for the bundle are valued in the group of diffeomorphisms of $X$ preserving the orientation and Spin$^c$-structure. The idea of studying the Seiberg-Witten equations for such a family was proposed by Donaldson in \cite{don2} and later pursued by various authors including Ruberman \cite{rub1, rub3}, Liu \cite{liu}, Li-Liu \cite{liliu} and Nakamura \cite{na1}.

One expects the families Seiberg-Witten invariants to have many deep applications to the study of families of smooth $4$-manifolds, just as the ordinary Seiberg-Witten invariants have produced many profound results in the study of individual smooth $4$-manifolds. For example, Ruberman used a certain $1$-parameter families Seiberg-Witten invariant to show that the space of positive scalar metric on certain simply-connected $4$-manifolds has infinitely many connected components \cite{rub3}. In order to realise such applications one needs practical tools for computing the families invariants, just as one has for the ordinary Seiberg-Witten invariants. To date only a limited set of tools are available such as the families wall-crossing formula \cite{liliu} and families blowup formula \cite{liu}. Both of these express one set of families invariants in terms of another set, so it is difficult to get off the ground with just these. In \cite{rub3} Ruberman devised a technique for computing the families Seiberg-Witten invariant for a special class of $1$-dimensional families, where $X$ is a connected sum of the form $X = M \# \mathbb{CP}^2 \#^2 \overline{\mathbb{CP}^2}$ and the family is the mapping cylinder of a certain diffeomorphism of $\mathbb{CP}^2 \#^2 \overline{\mathbb{CP}^2}$.

In this paper, we will prove a far-reaching generalisation of Ruberman's $1$-parameter formula. Our formula will relate the parametrised Seiberg-Witten invariant for a connected sum of families with fibres $X = M \# N$ to the ordinary Seiberg-Witten invariants on $M$. In this way, tools for computing the Seiberg-Witten invariants of $M$ become tools for computing the families Seiberg-Witten invariants of $X$. Quite surprisingly, our formula for families Seiberg-Witten invariants has non-trivial consequences for the ordinary Seiberg-Witten invariants as well.


Let us briefly outline the definition of the families Seiberg-Witten invariants. Their construction is described in more detail in Section \ref{sec:familiesSW}. Our approach follows \cite{liliu} but is more general, in particular \cite{liliu} assumes the existence of a Spin$^c$-structure on the vertical tangent bundle while our definition does not always assume this. In addition to this, we define some further variants of the families Seiberg-Witten invariants. One is a variant which incorporates twisting the families moduli space by the so-called charge conjugation symmetry, as dscribed in Section \ref{sec:familiesSW}. The other is a variant of the families Seiberg-Witten invariant for actions of a group $G$ on a $4$-manifold by diffeomorphisms. This invariant takes values in the group cohomology of $G$ and is described in Section \ref{sec:SWgroupaction}.

Let $X$ be a compact smooth oriented $4$-manifold and let $B$ be a compact smooth manifold. Suppose we have a smooth fibrewise oriented fibre bundle $\pi_X : E_X \to B$ whose fibres are diffeomorphic to $X$. Let $\mathfrak{s}_X$ be a Spin$^c$-structure on $X$. We will assume that $\mathfrak{s}_X$ can be extended to a continuously varying family of Spin$^c$-structures on the fibres of $E_X$. Equivalently, the transition functions for $E_X$ can be chosen to be diffeomorphisms of $X$ preserving the isomorphism classes of $\mathfrak{s}_X$. In general this is a weaker condition than requiring that the vertical tangent bundle of $E_X$ admits a Spin$^c$-structures extending $\mathfrak{s}_X$.

For simplicity, we concern ourselves in this introduction with just the $\mathbb{Z}_2$-valued families Seiberg-Witten invariants (see Sections \ref{sec:familiesSW}-\ref{sec:setup} for details of the $\mathbb{Z}$-valued invariants). Let
\[
d(X,\mathfrak{s}_X) = \frac{ c_1(\mathfrak{s}_X)^2 - \sigma(X)}{4} -1 - b^+(X) + b_1(X)
\]
be the virtual dimension of the unparametrised Seiberg-Witten moduli space of $X$. Assume that
\[
b^+(X) > dim(B) + 1.
\]
This assumption is necessary to have a well-defined families Seiberg-Witten invariant. Under this assumption, for any family of fibrewise metrics $g_X$ on $E_X$, there exists a families perturbation $\eta_X$ on $E_X$ for which the families moduli space $\mathcal{M} = \mathcal{M}( E_X , g_X , \eta_X , \mathfrak{s}_X)$ is either empty or a smooth manifold of dimension 
\[
dim(\mathcal{M}) = d(X,\mathfrak{s}_X) + dim(B).
\]
If $\mathfrak{s}_{X}$ extends to a Spin$^c$-structure on the vertical tangent bundle of $E_X$, the definition of the families moduli space $\mathcal{M}$ is the standard one, as discussed in \cite{liliu}. However, one can still define the families moduli space $\mathcal{M}$ under the weaker condition that we can equip $E_X$ with a continuously varying family of Spin$^c$-structures, in the sense described above.

The simplest version of the families Seiberg-Witten invariant is defined to be the homomorphism 
\[
FSW^{\mathbb{Z}_2}(E_X , \mathfrak{s}_X , \; \; ) : H^{dim(\mathcal{M})}(B , \mathbb{Z}_2) \to \mathbb{Z}_2,
\]
given by
\[
FSW^{\mathbb{Z}_2}(E_X , \mathfrak{s}_X , \theta) = \langle \pi^* \theta, [\mathcal{M}] \rangle,
\]
where denotes $[\mathcal{M}]$ is the fundamental class of $\mathcal{M}$ with $\mathbb{Z}_2$-coefficients. 

A more general families Seiberg-Witten invariant is obtained by incorporating a certain line bundle $\mathcal{L} \to \mathcal{M}$ on the families moduli space. Recall that for the unparametrised Seiberg-Witten moduli space, one obtains $\mathcal{L}$ as the line bundle associated to the principal circle bundle $\widetilde{\mathcal{M}} \to \mathcal{M}$, where $\widetilde{\mathcal{M}}$ is the Seiberg-Witten moduli space modulo the reduced gauge group of gauge transformations which are the identity at some fixed basepoint of $X$. The circle bundle $\widetilde{\mathcal{M}} \to \mathcal{M}$ corresponds to the residual action of constant gauge transformations. In the families setting, the definition of the line bundle $\mathcal{L}$ is similar but requires additional care. We construct $\mathcal{L}$ in Section \ref{sec:familiesSW} under the assumptions that $b_1(X) = 0$ and that $\mathfrak{s}_X$ extends to a Spin$^c$-structure on the vertical tangent bundle, which we denote by $\tilde{\mathfrak{s}}_X$. Then for any $m \ge 0$ we define a families Seiberg-Witten invariant
\begin{equation*}
\begin{aligned}
& FSW^{\mathbb{Z}_2}_m(E_X , \mathfrak{s}_X) : H^{dim(\mathcal{M})-2m}(B , \mathbb{Z}_2) \to \mathbb{Z}_2 \\
& FSW^{\mathbb{Z}_2}_m(E_X , \tilde{\mathfrak{s}}_X , \theta) = \langle c_1(\mathcal{L})^m \cup \pi^* \theta, [\mathcal{M}] \rangle \in \mathbb{Z}_2.
\end{aligned}
\end{equation*}
We sometimes abuse notation and write $FSW^{\mathbb{Z}_2}_m(E_X , \mathfrak{s}_X , \theta)$ in place of \linebreak $FSW^{\mathbb{Z}_2}_m(E_X , \tilde{\mathfrak{s}}_X , \theta)$, noting however that $FSW^{\mathbb{Z}_2}_m(E_X , \tilde{\mathfrak{s}}_X , \theta)$ in general depends on the choice of extension $\tilde{\mathfrak{s}}_X$.


We now describe the setup for the gluing formula. More details can be found in Section \ref{sec:setup}. Our gluing result is for families obtained by a families connected sum. Let $M,N$ be compact, smooth, oriented $4$-manifolds and set $X = M \# N$. In this construction $M$ and $N$ will play specific roles and the assumptions on $M$ and $N$ will not be symmetric. In particular we assume that $b_1(N) = 0$ but make no assumption about $b_1(M)$ in general.

Let $B$ be a compact smooth manifold and suppose we have smooth fibrewise oriented families $\pi_M : E_M \to B$, $\pi_N : E_N \to B$ whose fibres are $M,N$ respectively. Suppose we are given sections $\iota_M : B \to E_M$ and $\iota_N : B \to E_N$ whose normal bundles are orientation reversing isomorphic. Then by removing tubular neighbourhoods of the sections $\iota_M, \iota_N$ from $E_M, E_N$ and identifying their boundaries using the orientation reversing isomorphism, we obtain a family $E_X$ over $B$ whose fibres are diffeomorphic to the connected sum $X = M \# N$. In general $E_X$ will depend on the isotopy classes of the sections $\iota_M, \iota_N$ and the isomorphism of their normal bundles.

Let $\mathfrak{s}_M, \mathfrak{s}_N$ be Spin$^c$-structures on $M,N$ and let $\mathfrak{s}_X$ be the Spin$^c$-structure on $X$ obtained as the connected sum of $\mathfrak{s}_M$ and $\mathfrak{s}_N$. We will assume that $\mathfrak{s}_M,\mathfrak{s}_N$ can be extended to continuously varying families of Spin$^c$-structures on the fibres of $E_M, E_N$. It follows that $\mathfrak{s}_X$ similarly extends to a continuous family of Spin$^c$-structures on the fibres of $E_X$. We will assume that $d(M , \mathfrak{s}_M) = 2m \ge 0$ is non-negative and even.

We make the following assumptions about $N$. Assume that:
\[
c_1(\mathfrak{s}_N)^2 = \sigma(N)
\]
so that the index of the Spin$^c$-Dirac operator on $N$ is zero. Assume also that:
\[
0 < b^+(N) \le dim(B).
\]
Let $\mathcal{M} = \mathcal{M}(E_X , g_X , \eta_X , \mathfrak{s}_X )$ be the Seiberg-Witten moduli space for the family $E_X$ with respect to a choice of fibrewise metric $g_X$ and fibrewise perturbation $\eta_X$. If $b^+(M) + b^+(N) > dim(B) + 1$, then $\mathcal{M}$ is a compact smooth manifold for a generic choice of families perturbation. Assume this is the case and that $\eta_X$ has been chosen generically. Then $\mathcal{M}$ is either empty or a compact smooth manifold of dimension
\begin{equation*}
\begin{aligned}
dim(\mathcal{M}) &= d(M , \mathfrak{s}_M) + d(N , \mathfrak{s}_N) + 1 + dim(B) \\
&= 2m + dim(B) - b^+(N).
\end{aligned}
\end{equation*}
The main gluing theorem of this paper is as follows:

\begin{theorem}[$\mathbb{Z}_{2}$-valued gluing formula]
Suppose that either
\begin{itemize}
\item[(i)]{$d(M , \mathfrak{s}_M) = 2m = 0$, or}
\item[(ii)]{$b_1(M) = b_1(X) = 0$ and $\mathfrak{s}_X$ extends to a Spin$^c$-structure on $T(E_X/B)$.}
\end{itemize}
Then:
\begin{itemize}
\item[(1)]{For any $\theta \in H^{dim(B)-b^+(N)}(B , \mathbb{Z}_2)$, we have
\[
FSW^{\mathbb{Z}_2}_m(E_X , \mathfrak{s}_X , \theta ) = SW( M , \mathfrak{s}_M) \cdot \langle \theta \cup w_{b^+(N)}(H^+(N)) , [B] \rangle.
\]
}
\item[(2)]{For any integer $k > 0$ and any $\theta \in H^{dim(B)-b^+(N)+2k}(B , \mathbb{Z}_2)$, we have
\[
FSW^{\mathbb{Z}_2}_{m-k}( E_X , \mathfrak{s}_X , \theta) = 0.
\]
}
\end{itemize}
\end{theorem}

We also prove such a gluing formula for the families Seiberg-Witten invariants incorporating twists by charge conjugation (Theorem \ref{thm:rubtwisted}), for the $\mathbb{Z}$-valued invariants (Theorem \ref{thm:rubz}) and for the Seiberg-Witten invariants of group actions (Theorem \ref{thm:rubequivariant}).

In Section \ref{sec:applications} we consider various applications of the families gluing formula. Below we summarise some of the main results obtained.

\begin{theorem}
Let $M, M'$ be compact simply-connected smooth $4$-manifolds with indefinite intersection forms. Suppose that $M,M'$ are homeomorphic and fix a homeomorphism $\phi : M \to M'$. Suppose that $b^+(M) > 1$ and that $\mathfrak{s}_M$ is a Spin$^c$-structure on $M$ with $d(M,\mathfrak{s}_M) = 0$ and that $SW(M,\mathfrak{s}_M) \neq SW(M' , \phi(\mathfrak{s}_M) ) \; ({\rm mod} \; 2)$. Lastly, suppose that $X = M \# (S^2 \times S^2)$ is diffeomorphic to $M' \# (S^2 \times S^2)$. Then there exists a diffeomorphism on $X$ which is continuously isotopic to the identity but not smoothly isotopic.
\end{theorem}

\begin{corollary}
The following $4$-manifolds admit diffeomorphisms which are continuously isotopic to the identity, but are not smoothly isotopic to the identity:
\begin{itemize}
\item[(i)]{$X = \#^n(S^2 \times S^2) \#^n (K3)$, for any $n \ge 2$.}
\item[(ii)]{$X = \#^{2n} \mathbb{CP}^2 \#^m \overline{\mathbb{CP}}^2$, for any $n \ge 2$ and any $m \ge 10n+1$.}
\end{itemize}
\end{corollary}

In \textsection \ref{sec:mod2spin} we use the gluing formula to give a simple new proof of a theorem of Morgan and Szab\'o:

\begin{theorem}[\cite{mosa}]
Let $M$ be smooth compact spin $4$-manifold with $b_1(M) = 0$, $b^+(M) = 4n-1$, $b^-(M) = 20n-1$, where $n \ge 1$. Let $\mathfrak{s}_M$ be a Spin$^c$-structure on $M$ which comes from a spin structure. Then $SW(M , \mathfrak{s}_M)$ is odd if $n=1$ and is even otherwise.
\end{theorem}

In \textsection \ref{sec:relations} we show how the gluing formula can be used to compute the (unparametrised) mod $2$ Seiberg-Witten invariants of $4$-manifolds under certain conditions. Amongst such results we obtain the following: 

\begin{theorem}
Let $M$ be a compact smooth $4$-manifold with $b_1(M) = 0$ and $b^+(M) = 3 \; ({\rm mod} \; 4)$. Suppose that $f : M \to M$ is an involutive diffeomorphism of $M$ whose fixed point set contains an isolated point. Suppose suppose $\mathfrak{s}_M$ is a Spin$^c$-structure such that $f(\mathfrak{s}_M) = -\mathfrak{s}_M$ and $c_1(\mathfrak{s}_M)^2 - \sigma(M) = 0 \; ({\rm mod} \; 16)$.  
\begin{itemize}
\item[(i)]{If $b^+(M) = 3$ and $c_1(\mathfrak{s}_M)^2 - \sigma(M) \ge 16$, then $SW(M , \mathfrak{s}_M) = 1 \; ({\rm mod} \; 2)$ if $f$ acts trivially on $H^+(M)$ and $SW(M , \mathfrak{s}_M) = 0 \; ({\rm mod} \; 2)$ otherwise.}
\item[(ii)]{If $b^+(M) > 3$, then $SW(M , \mathfrak{s}_M) = 0 \; ({\rm mod} \; 2)$.}
\end{itemize}
\end{theorem}

In \textsection \ref{sec:applgroupaction}, we prove the gluing formula for the Seiberg-Witten invariants of group actions as a consequence of the gluing formula for families. The non-vanishing of this invariant gives an obstruction to the existence of an invariant metric with positive scalar curvature. In \textsection \ref{sec:involutions} we apply this formula to the simplest case where the group is $\mathbb{Z}_2$. As a corollary, we find:
\begin{corollary}
For every $n \ge 4$ and $m \ge n+18$, the $4$-manifold $X = \#^{n} \mathbb{CP}^2 \#^{m} \overline{\mathbb{CP}}^2$ admits a metric of positive scalar curvature and a smooth involution $f : X \to X$ such that there does not exist an $f$-invariant metric of positive scalar curvature on $X$.
\end{corollary}


The following is a brief outline of the sections of the paper. In Section \ref{sec:familiesSW} we give the definition of the families Seiberg-Witten invariants including several generalisations. In Section \ref{sec:SWgroupaction} we define an equivariant Seiberg-Witten invariant for smooth group actions and relate it to the families invariants. In Section \ref{sec:setup} we describe in detail the setup required to formulate the families gluing formula and state the main result. Sections \ref{sec:step1}-\ref{sec:part2} give the proof of the gluing formula in several steps. Section \ref{sec:gluing} contains a review of the relevant aspects of gluing theory for unparametrised Seiberg-Witten theory, mostly following the textbook \cite{nic}. In particular \textsection \ref{sec:glueloc} is concerned with the local Kuranishi model for glued configurations in Seiberg-Witten theory. Section \ref{sec:gluefamily} is concerned with adapting the arguments of the previous sections to the families setting. This turns out to be not as straightforward as one might expect, in part because of the fact that certain obstruction spaces in the local Kuranishi model are non-zero. Additionally one has to choose the families perturbation very carefully in order for the gluing arguments to work well. In Section \ref{sec:applications} we consider various applications of the gluing formula: a vanishing theorem (\textsection \ref{sec:vanishing}), obstructions to smooth isotopy (\textsection \ref{sec:isotopy}), the computation of the mod $2$ Seiberg-Witten invariants of spin structures (\textsection \ref{sec:mod2spin}), relations between mod $2$ Seiberg-Witten invariants of different $4$-manifolds, the gluing formula for the Seiberg-Witten invariants of group actions (\textsection \ref{sec:applgroupaction}) and actions on $4$-manifolds by involutions (\textsection \ref{sec:involutions}).\\

\noindent{\bf Acknowledgments}.
The authors acknowledge useful initial discussions with Varghese Mathai. Varghese Mathai acknowledges partial support by the Australian Research Council via DP 170101054. The authors would also like to express their appreciation to Nobuhiro Nakamura for informing them of LeBrun's paper. D. Baraglia was financially supported by the Australian Research Council Discovery Early Career Researcher Award DE160100024 and Australian Research Council Discovery Project DP170101054. H. Konno was supported by JSPS KAKENHI Grant Number
16J05569 and the Program for Leading Graduate Schools, MEXT, Japan.

\section{The families Seiberg-Witten invariant}\label{sec:familiesSW}

In this section we describe the formulation of the families invariant which we consider in this paper.
Although this type of invariant has been defined by Li-Liu~\cite{liliu}, we exhibit here some of its generalisations. In Section \ref{sec:SWgroupaction} we formulate a version of this invariant for group actions.

Let $X$ be a compact, smooth oriented $4$-manifold.
Let $B$ be a compact smooth manifold of dimension $d$ and suppose we have a smooth fibrewise oriented fibre bundle $\pi_X : E_X \to B$ whose fibres are diffeomorphic to $X$.
Let $T(E_X/B) = Ker( {\pi_X}_* )$ be the vertical tangent bundle.

Let $\mathfrak{s}_X$ be a Spin$^c$-structure on $X$. We will assume that $\mathfrak{s}_X$ is monodromy invariant under the monodromy action of $\pi_1(B)$ on the set of Spin$^c$-structures on $X$ induced by the family $E_X$. That is, we assume $\mathfrak{s}_X$ extends to give a continuously varying family of Spin$^c$-structures on the fibres of $E_X$. Note that this is in general a weaker condition than requiring the existence of a Spin$^c$-structure on $T(E_X/B)$ which restricts to $\mathfrak{s}_X$ on each fibre.
See Proposition~2.1 in \cite{bar} on this point.

Let us first consider the $\mathbb{Z}_2$-valued families Seiberg-Witten invariant, so we do not have to consider orientations on the moduli space at this stage. The necessary modifications for the $\mathbb{Z}$-valued families Seiberg-Witten invariant are straightforward and will be dealt with later.
Let
\[
d(X,\mathfrak{s}_X) = \frac{ c_1(\mathfrak{s}_X)^2 - \sigma(X)}{4} -1 - b^+(X) + b_1(X)
\]
be the virtual dimension of the ordinary Seiberg-Witten moduli space of $X$. Assume that
\[
b^+(X) > d + 1.
\]
This assumption is necessary to have a well-defined families Seiberg-Witten invariant. Under this assumption, for any family of fibrewise metrics $g_X$ on $E_X$, there exists a families perturbation $\eta_X$ on $E_X$ for which the families moduli space $\mathcal{M} = \mathcal{M}( E_X , g_X , \eta_X , \mathfrak{s}_X)$ is a smooth manifold of dimension 
\[
dim(\mathcal{M}) = d(X,\mathfrak{s}_X) + d
\]
($\mathcal{M}$ is empty if this number is negative).
Here, if $\mathfrak{s}_{X}$ extends to a Spin$^c$-structure on $T(E_X/B)$, the definition of the families moduli space $\mathcal{M}$ is the standard one, such as discussed in \cite{liliu}.
However, in fact, even when $\mathfrak{s}_{X}$ does not extend, one can define the families moduli space.
See Subsection~2.2 in \cite{bar} on this point.

Suppose that $m \ge 0$ is a non-negative integer and set $s = dim(\mathcal{M}) - 2m$. Now we define the families invariant of interest. The invariant will be defined in the following two cases. Assume that either:
\begin{itemize}
\item[(i)]{$m = 0$, or}
\item[(ii)]{$b_1(X) = 0$ and $\mathfrak{s}_X$ extends to a Spin$^c$-structure on $T(E_X/B)$, which we denote by $\tilde{\mathfrak{s}}_X$.}
\end{itemize}
In case (ii), we obtain a line bundle $\mathcal{L} \to \mathcal{M}$ over $\mathcal{M}$ as follows. Let $\mathcal{G} = Map( X  , S^1)$ be the gauge group on $X$. Consider the subgroup $\mathcal{G}_0$ of $\mathcal{G}$ given by
\[
\mathcal{G}_0 = \{ g \in \mathcal{G} \; | \; g = e^{if}, \text{ for some $f$ with } \int_{X} f dvol_X = 0 \}.
\]
Since $b_1(X)  = 0$, we see that there is a short exact sequence $1 \to \mathcal{G}_0 \to \mathcal{G} \to S^1 \to 1$. If we consider the families moduli space of solutions to the Seiberg-Witten equations taken modulo $\mathcal{G}_0$, we obtain a principal circle bundle over $\mathcal{M}$, the families moduli space of solutions modulo $\mathcal{G}$, hence an associated line bundle $\mathcal{L} \to \mathcal{M}$. Note that this line bundle depends on the choice of extension of $\tilde{\mathfrak{s}}_X$ of $\mathfrak{s}_X$
(see again Subsection~2.2 in \cite{bar} on this point).
In fact, since we assume $b_1(X) = 0$, any two extensions differ by the action of the pullback of a line bundle $L_B \to B$. Then it is easy to see that $\mathcal{L}$ changes to $\mathcal{L} \otimes \pi^*(L_B)$, where $\pi : \mathcal{M} \to B$ is the natural map to $B$.

\begin{definition}[Families $\mathbb{Z}_2$-invariant]\label{def:fsw}
Suppose that $\theta \in H^s(B , \mathbb{Z}_2)$. Now we define the families invariant as follows:
\begin{itemize}
\item[(i)]{If $m = 0$, then $dim(\mathcal{M}) = s$ and we set
\[
FSW^{\mathbb{Z}_2}(E_X , \mathfrak{s}_X , \theta) = \langle \pi^* \theta, [\mathcal{M}] \rangle \in \mathbb{Z}_2,
\]
where $[\mathcal{M}]$ is the fundamental class of $\mathcal{M}$ with $\mathbb{Z}_2$-coefficients.}
\item[(ii)]{If $m > 0$, then choose an extension $\tilde{\mathfrak{s}}_X$ of $\mathfrak{s}_X$ to a Spin$^c$-structure on $T(E_X/B)$. In turn this determines a line bundle $\mathcal{L} \to \mathcal{M}$. We then set
\[
FSW^{\mathbb{Z}_2}_m(E_X , \tilde{\mathfrak{s}}_X , \theta) = \langle c_1(\mathcal{L})^m \cup \pi^* \theta, [\mathcal{M}] \rangle \in \mathbb{Z}_2.
\]
}
\end{itemize}
\end{definition}

The condition $b^+(X) > d + 1$ ensures that $FSW^{\mathbb{Z}_2}(E_X , \mathfrak{s}_X , \theta)$ does not depend on the choice of metric or perturbation and so gives an invariant of the family. In case (ii), the invariant will in general depend on the choice of lift $\tilde{\mathfrak{s}}_X$ of $\mathfrak{s}_X$. However, to unify notation we will often write the invariant in case (ii) simply as $FSW^{\mathbb{Z}_2}_m(E_X , \mathfrak{s}_X , \theta)$, suppressing the dependence on the choice of lift $\tilde{\mathfrak{s}}_X$. This will be justified later when we see that under certain conditions the invariant is actually independent of the choice of lift.

\begin{remark}
Variant (ii) of Definition \ref{def:fsw} can be defined also for $m=0$ in which case it clearly coincides with variant (i): $FSW^{\mathbb{Z}_2}_0(E_X , \tilde{\mathfrak{s}}_X , \theta) = FSW^{\mathbb{Z}_2}(E_X , \mathfrak{s}_X , \theta)$.
\end{remark}

Next, we generalise the setup slightly to incorporate the so-called charge conjugation symmetry $j$. This construction was considered by Ruberman in the case of $1$-dimensional families in \cite{rub1}. Recall that for any Spin$^c$-structure $\mathfrak{s}_X$ on $X$ there is a ``dual" or ``charge conjugate" Spin$^c$-structure, which we will denote by $-\mathfrak{s}_X$. We have that $c_1( -\mathfrak{s}_X ) = -c_1(\mathfrak{s}_X)$ and that $\mathfrak{s}_X$ comes from a spin structure precisely when $\mathfrak{s}_X = -\mathfrak{s}_X$. Let $S^\pm( \mathfrak{s}_X)$ denote the spinor bundles associated to $\mathfrak{s}_X$. Then charge conjugation determines an antilinear isomorphism $j : S^\pm(\mathfrak{s}_X) \to S^\pm(-\mathfrak{s}_X)$ with the property that $j^2 = -1$. Note that if $\mathfrak{s}_X$ comes from a spin structure, then $j$ gives a quaternionic structure on $S^\pm(\mathfrak{s}_X)$.

Let $\pi_X : E_X \to B$ be a smooth family of $4$-manifolds with fibre $X$. Recall that $\pi_1(B)$ acts on the set of Spin$^c$-structures by monodromy. We will consider families where $\mathfrak{s}_X$ is preserved up to charge conjugation. Thus we assume there exists a homomorphism $\rho : \pi_1(B) \to \mathbb{Z}_2$ such that the monodromy action of $\pi_1(B)$ preserves $\mathfrak{s}_X$ up to a sign factor given by $\rho$. Note that if $\mathfrak{s}_X$ does not come from a spin structure, then $\rho$ is uniquely determined by the monodromy. On the other hand if $\mathfrak{s}_X$ comes from a spin structure, then $\rho$ may be chosen arbitrarily, and this choice gives rise to different families Seiberg-Witten invariants.

Define $d(X,\mathfrak{s}_X)$ as before and assume that $b^+(X) > dim(B) + 1$. View $\rho$ as a class in $H^1( B , \mathbb{Z}_2)$. Represent $\rho$ as a $\mathbb{Z}_2$-valued cocycle $\rho_{\alpha \beta}$ with respect to some open cover $U_\alpha$ of $B$. We construct local families moduli spaces $\mathcal{M}_\alpha$ over each $U_\alpha$ exactly as before. On the overlap $U_\alpha \cap U_\beta$, we patch together $\mathcal{M}_\alpha$ and $\mathcal{M}_\beta$ using the charge conjugation symmetry $(a,\psi) \mapsto (-a , j(\psi))$. Note that this gives us a globally well-defined moduli space $\mathcal{M}^\rho$ because $j$ squares to a gauge transformation. Note that if $\eta_{\alpha}$ are the locally defined perturbation $2$-forms, then for these to patch together we require $\eta_{\alpha}|_{U_{\alpha \beta}} = (-1)^{\rho_{\alpha \beta}} \eta_{\beta}|_{U_{\alpha \beta}}$. Let $\mathbb{R}_\rho \to B$ be the real line bundle on $B$ determined by $\rho \in H^1(B , \mathbb{Z}_2)$. Then the above patching condition on the perturbations $\{ \eta_{\alpha} \}$ means that they patch together to give a global section $\eta$ of $H^+(X) \otimes \mathbb{R}_\rho$.

\begin{definition}[$\rho$-twisted families $\mathbb{Z}_2$-invariant]\label{def:fsw}
Suppose that $\theta \in H^s(B , \mathbb{Z}_2)$, where $s = d(X,\mathfrak{s}_X) + dim(B)$. Now we define the $\rho$-twisted families invariant as follows:
\[
FSW^{\mathbb{Z}_2}(E_X , \mathfrak{s}_X , \theta , \rho) = \langle \pi^* \theta, [\mathcal{M}^\rho] \rangle \in \mathbb{Z}_2,
\]
where $[\mathcal{M}^\rho]$ is the fundamental class of $\mathcal{M}^\rho$ with $\mathbb{Z}_2$-coefficients.
\end{definition}

Next we formulate a $\mathbb{Z}$-valued version of the families Seiberg--Witten invariant.
In addition to all assumptions supposed when we defined the $\mathbb{Z}_{2}$-valued invariant (Definition~\ref{def:fsw}), let us assume that $B$ is oriented and fix a homology orientation of $X$. Further, let $det^+(E_X)$ denote the real line bundle on $B$ whose fibre over $t \in B$ is $det( H^1(X_t , \mathbb{R}) \oplus H^+(X_t) )$, where $X_t$ denote the fibre of $E_X$ over $t$ and $H^+(X_t)$ denotes the space of harmonic self-dual $2$-forms on $X_t$ (with respect to some smoothly varying family of fibrewise metrics). Suppose that $det^+(E_X)$ is orientable and fix an orientation on $det^+(E_X)$. As in \cite{liliu}, this gives rise to an orientation of the families moduli space $\mathcal{M}$ for a generic families perturbation, and thus we get a fundamental class $[\mathcal{M}]$ in $\mathbb{Z}$-valued cohomology.
By repeating Definition~\ref{def:fsw}, we have: 

\begin{definition}[Families $\mathbb{Z}$-invariant]\label{def:fswZ}
Suppose that $\theta \in H^s(B , \mathbb{Z})$. Now we define the families invariant as follows:
\begin{itemize}
\item[(i)]{If $m = 0$, then $dim(\mathcal{M}) = s$ and we set
\[
FSW^{\mathbb{Z}}(E_X , \mathfrak{s}_X , \theta) = \langle \pi^* \theta, [\mathcal{M}] \rangle \in \mathbb{Z},
\]
where $[\mathcal{M}]$ is the fundamental class of $\mathcal{M}$ with $\mathbb{Z}$-coefficients.}
\item[(ii)]{If $m > 0$, then choose an extension $\tilde{\mathfrak{s}}_X$ of $\mathfrak{s}_X$ to a Spin$^c$-structure on $T(E_X/B)$. In turn this determines a line bundle $\mathcal{L} \to \mathcal{M}$. We then set
\[
FSW^{\mathbb{Z}}_m(E_X , \tilde{\mathfrak{s}}_X , \theta) = \langle c_1(\mathcal{L})^m \cup \pi^* \theta, [\mathcal{M}] \rangle \in \mathbb{Z}.
\]
}
\end{itemize}
\end{definition}

\begin{remark}
If $B$ or $det^+(E_X)$ are non-orientable, we can still define $\mathbb{Z}$-valued families Seiberg-Witten invariants provided we work with local coefficients. We can also consider a $\mathbb{Z}$-valued $\rho$-twisted families Seiberg-Witten invariant. For simplicity we omit the details of these generalisations.
\end{remark}

We here give a remark on a cohomological description of the families Seiberg--Witten invariants defined above.
The argument here is parallel for both $\mathbb{Z}_{2}$-case and $\mathbb{Z}$-case, and therefore we drop the coefficient from our notation of (co)homology and from the families Seiberg-Witten invariant.
First, recall that we can define the wrong way map
\[
\pi_{\ast} : H^{\ast}(\mathcal{M}) \to H^{\ast+dim(B)-dim(\mathcal{M})}(B)
\]
by composing the usual pull-back $\pi^{\ast}$ on cohomology and Poincar\'e duality for both $\mathcal{M}$ and $B$.
Next, in the setting of Definition~\ref{def:fsw} or Definition~\ref{def:fswZ}, let us write again the dimension of the families moduli space as $dim(\mathcal{M}) = 2m + s$ for some non-negative integer $m$ and $s$, and assume again that either:
\begin{itemize}
\item[(i)]{$m = 0$, or}
\item[(ii)]{$b_1(X) = 0$ and $\mathfrak{s}_X$ extends to a Spin$^c$-structure on $T(E_X/B)$, which we denote by $\tilde{\mathfrak{s}}_X$.}
\end{itemize}
In case (i), we set
\[
FSW(E_X , \mathfrak{s}_X) = \pi_{\ast}(1) \in H^{-d(X , \mathfrak{s}_X)}(B) = H^{d-s}(B),
\]
and in case (ii), we set
\[
FSW_m(E_X , \mathfrak{s}_X) = \pi_{\ast}(c_1(\mathcal{L})^m) \in H^{2m-d(X , \mathfrak{s}_X)}(B) = H^{d-s}(B).
\]
Then we get the equality
\begin{align}
\label{eq cohomological and numerical}
FSW_m(E_X , \mathfrak{s}_X , \theta) = \langle FSW_m(E_X , \mathfrak{s}_X) \cup \theta , [B] \rangle
\end{align}
for any $\theta \in H^s(B,\mathbb{Z})$ because of the so-called projection formula.
In particular, if we consider the $\mathbb{Z}_{2}$-coefficient case,
then equation \eqref{eq cohomological and numerical} can be regarded as a defining property of the cohomology class $FSW(E_X , \mathfrak{s}_X)$.
Thus we will sometime view the families Seiberg-Witten invariant as the cohomology class $FSW(E_X , \mathfrak{s}_X)$.
As well as in the untwisted case, note that we can regard the $\rho$-twisted families invariant as a cohomology class
\[
FSW^{\mathbb{Z}_2}(E_X, \mathfrak{s}_X , \rho) \in H^{-d(X , \mathfrak{s}_X)}(B , \mathbb{Z}_2).
\]

\section{The Seiberg-Witten invariants of group actions}\label{sec:SWgroupaction}

In this section we define a variant of the families Seiberg-Witten invariant which is defined for smooth group actions.

Let $X$ be a smooth compact oriented $4$-manifold and let $G$ be a group acting on $X$ by orientation preserving diffeomorphisms. Suppose that $G$ preserves the isomorphism class of a Spin$^c$-structure $\mathfrak{s}_X$. Suppose that $d(X , \mathfrak{s}_X) \le 0$ and that $b^+(X) > -d(X , \mathfrak{s}_X) + 1$. Let $B$ be any smooth compact manifold with $dim(B) < b^+(X) - 1$. Let $E \to B$ be a principal $G$-bundle over $B$ (where $G$ is given the discrete topology). We obtain an associated family $E_X = E \times_G X \to B$. Associated to this family and Spin$^c$-structure $\mathfrak{s}_X$ is the families Seiberg-Witten invariant:
\[
FSW^{\mathbb{Z}_2}(E_X , \mathfrak{s}_X) \in H^{-d(X,\mathfrak{s}_X)}(B , \mathbb{Z}_2).
\]
If we consider $(X,\mathfrak{s}_X)$ and the $G$-action on $X$ as fixed data, then we may think of $FSW^{\mathbb{Z}_2}(E_X , \mathfrak{s}_X)$ as a map which assigns to each principal $G$-bundle $E\to B$ over a base $B$ with $dim(B) < b^+(X) - 1$ a cohomology class in $H^{-d(X,\mathfrak{s}_X)}(B , \mathbb{Z}_2)$. It turns out that this assignment is a characteristic class, that is, there exists a cohomology class
\[
SW_G^{\mathbb{Z}_2}(X , \mathfrak{s}_X) \in H^{-d(X,\mathfrak{s}_X)}( BG , \mathbb{Z}_2)
\]
such that if $f : B \to BG$ is the classifying map of $E \to B$ then
\begin{equation}\label{equ:charclass}
FSW^{\mathbb{Z}_2}(E_X , \mathfrak{s}_X) = f^*( SW_G^{\mathbb{Z}_2}(X , \mathfrak{s}_X) ).
\end{equation}
This follows easily from the existence of the characteristic classes $\mathbb{SW}^{\mathbb{Z}_2}(X , \mathfrak{s}_X) \in H^{-d(X,\mathfrak{s}_X)}( BDiff(X , \mathfrak{s}_X) , \mathbb{Z}_2)$ contructed by Konno \cite{konno}. Recall that if $E_X \to B$ is any family of $4$-manifolds with structure group $Diff(X , \mathfrak{s}_X)$ (the group of diffeomorphisms of $X$ preserving the isomorphism class of $\mathfrak{s}_X$), the families Seiberg-Witten invariant is given by
\[
FSW^{\mathbb{Z}_2}(E_X , \mathfrak{s}_X) = g^*( \mathbb{SW}^{\mathbb{Z}_2}(X , \mathfrak{s}_X)),
\]
where $g : B \to BDiff(X , \mathfrak{s}_X)$ is the classifying map of $E_X \to B$. In the case that $E_X$ is of the form $E_X = E \times_G X$, we have that $g$ factors as
\[
B \buildrel f \over \longrightarrow BG \buildrel \psi \over \longrightarrow BDiff(X , \mathfrak{s}_X),
\]
where $\psi$ is the map induced from the homomorphism $G \to Diff(X , \mathfrak{s}_X)$ determined by the action of $G$ on $X$. Then Equation (\ref{equ:charclass}) follows immediately if we define $SW_G^{\mathbb{Z}_2}(X , \mathfrak{s}_X)$ to be:
\begin{equation}\label{equ:gSW}
SW_G^{\mathbb{Z}_2}(X , \mathfrak{s}_X) = \psi^*( \mathbb{SW}^{\mathbb{Z}_2}(X , \mathfrak{s}_X) ).
\end{equation}
\begin{definition}
We define the $G$-equivariant Seiberg-Witten invariant of $(X , \mathfrak{s}_X)$ with $\mathbb{Z}_2$-coefficients to be the class $SW_G^{\mathbb{Z}_2}(X, \mathfrak{s}_X) \in H^{-d(X,\mathfrak{s}_X)}(BG , \mathbb{Z}_2)$ defined by Equation (\ref{equ:gSW}).
\end{definition}
We can also consider a twisted version of the invariant $SW_G^{\mathbb{Z}_2}$ as follows. Let $X$ be a smooth compact oriented $4$-manifold and let $G$ be a group acting on $X$ by orientation preserving diffeomorphisms. Let $\rho : G \to \mathbb{Z}_2$ be a homomorphism and suppose that $G$ preserves a Spin$^c$-structure $\mathfrak{s}_X$ up to a sign factor determined by $\rho$:
\[
g^*(\mathfrak{s}_X) = (-1)^{\rho(g)}\mathfrak{s}_X, \text{ for all } g \in G.
\]
Suppose as before that $d(X , \mathfrak{s}_X) \le 0$ and that $b^+(X) > -d(X , \mathfrak{s}_X) + 1$. Then we obtain a $\rho$-twisted $G$-equivariant Seiberg-Witten invariant
\[
SW_G^{\mathbb{Z}_2}( X , \mathfrak{s}_X , \rho) \in H^{-d(X , \mathfrak{s}_X)}( BG , \mathbb{Z}_2).
\]

One can also define the corresponding $\mathbb{Z}$-valued invariants under suitable orientability assumptions. In all applications of these invariants that we have considered so far, the $\mathbb{Z}_2$-invariant suffices. Hence we omit the details of the construction of the $G$-equivariant $\mathbb{Z}$-valued Seiberg-Witten invariant.

The following proposition illustrates one of the possible applications of the $G$-equivariant Seiberg-Witten invariants:
\begin{proposition}
If $X$ admits a $G$-invariant metric of positive scalar curvature then $SW_G^{\mathbb{Z}_2}(X , \mathfrak{s}_X , \rho) = 0$, whenever it is defined.
\end{proposition}
\begin{proof}
Assume that $X$ admits such a metric. To show that $SW_G^{\mathbb{Z}_2}(X , \mathfrak{s}_X , \rho ) = 0$ it suffices to show that its pairing with any class $\beta \in H_{-d(X , \mathfrak{s}_X)}(BG , \mathbb{Z}_2)$ is zero. By the solution to the Steenrod problem with $\mathbb{Z}_2$-coefficients, there exists a compact smooth manifold $B$ of dimension $-d(X,\mathfrak{s}_X)$ and a continuous map $f : B \to BG$ such that $f_*[B] = \beta $. Let $E = f^*(EG)$ be the pullback of the universal bundle $EG \to BG$ and let $E_X = E \times_G X$ be the family associated to $E$. Then
\[
\langle SW_G^{\mathbb{Z}_2}(X , \mathfrak{s}_X , \rho) , \beta \rangle = \langle f^*(SW_G^{\mathbb{Z}_2}(X , \mathfrak{s}_X) , \rho) , [B] \rangle = \langle FSW^{\mathbb{Z}_2}(E_X , \mathfrak{s}_X , \rho ) , [B] \rangle.
\]
So it suffices to show that $FSW^{\mathbb{Z}_2}(E_X , \mathfrak{s}_X , \rho ) = 0$. But a $G$-invariant metric $g$ on $X$ with positive scalar curvature determines a family $\{ g_b \}_{b \in B}$ of fibrewise metrics on $E_X$ with positive scalar curvature. Since $b^+(X) > dim(B) + 1 = -d(X,\mathfrak{s}_X) + 1$ (by assumption) we can choose a families perturbation $\{ \eta_b \}_{b \in B}$ which avoids the wall of reducibles and we can in addition take the $\mathcal{C}^0$ norm of the perturbations $\{ \eta_b \}_{b \in B}$ to be sufficiently small so that there are no irreducible solutions of the Seiberg-Witten equations for $(g_b , \eta_b)$ for every $b \in B$. Thus the families Seiberg-Witten moduli space is empty and hence $FSW^{\mathbb{Z}_2}(E_X , \mathfrak{s}_X , \rho ) = 0$.
\end{proof}

\section{Setup for the general gluing construction}
\label{sec:setup}

In this section we describe the formulation of a gluing result used to calculate the families invariant.
This is a generalisation of Ruberman's argument given in \cite{rub3}.
(See also \cite{rub1,rub2} for the Yang-Mills version.)

Our gluing result can be proved for families that are obtained by a families connected sum.
Let $M,N$ be compact, smooth, oriented $4$-manifolds and set $X = M \# N$.
We assume throughout that $b_1(N) = 0$.
Let $B$ be a compact smooth manifold of dimension $d$ and suppose we have smooth fibrewise oriented families $\pi_M : E_M \to B$, $\pi_N : E_N \to B$ whose fibres are $M,N$ respectively. Let $T(E_M/B) = Ker( {\pi_M}_* )$, $T(E_N/B) = Ker( {\pi_N}_* )$ be the vertical tangent bundles. We wish to form a connected sum family. For this, suppose we are given the following additional data:
\begin{itemize}
\item{Sections $\iota_M : B \to E_M$ and $\iota_N : B \to E_N$.}
\item{An orientation reversing vector bundle isomorphism $\phi : \iota_M^*( T(E_M/B)) \to \iota_N^*( T(E_N/B))$.}
\end{itemize}
Let $V = \iota_M^*( T(E_M/B))$ be the normal bundle of $\iota_M$. Then $V$ is a real rank $4$-vector bundle on $B$. Note that $V$ is oriented by our assumption that $E_M \to B$ is fibrewise oriented. Fix once and for all a metric on $V$ so that $V$ becomes an $SO(4)$-vector bundle. By the isomorphism $\phi$ we can identify $\iota_N^*(T(E_N/B))$ with $V$ but with the opposite orientation.
Let $D(V) \to B$ be the unit open disc bundle of $V$ and $S(V) \to B$ the unit sphere bundle. Let $U_M \subset E_M$, $U_N \subset E_N$ be tubular neighbourhoods of $\iota_M, \iota_N$ equipped with diffeomorphisms $e_M : D(V) \to U_M$, $e_N : D(V) \to U_N$, where $e_M$ is orientation preserving and $e_N$ is orientation reversing.

Let $M'$ denote $M$ with a small open ball removed and similarly define $N'$. Then $M',N'$ are compact $4$-manfolds with boundary $\partial M' = \partial N' = S^3$. We may regard $E_{M'} = E_M \setminus U_M$ as a family of $4$-manifolds diffeomorphic to $M'$ over $B$ and similarly regard $E_{N'} = E_N \setminus U_N$ as a family of $4$-manifolds diffeomorphic to $N'$. The boundary of $E_{M'}$ is the $S^3$-bundle $S(V) \to B$. Similarly, the boundary of $E_{N'}$ is also $S(V) \to B$, except the fibres of $S(V) \to B$ inherit the opposite orientation from the family $E_{N'}$.

Let $\hat{M}, \hat{N}$ be the cylindrical end $4$-manifolds obtained from $M',N'$ by attaching the half-infinite cylinders $[0, \infty) \times S^3$ to $M'$ and $(-\infty , 0] \times S^3$ to $N'$. We obtain a family $E_{\hat{M}} \to B$ of cylindrical end $4$-manifolds by attaching to $E_{M'}$ the family of cylinders $[0,\infty) \times S(V)$ and we obtain a family $E_{\hat{N}} \to B$ similarly.

The fixed metric on $V$ determines families of fibrewise metrics on $D(V)$ and $S(V)$ which in turn define fibrewise metrics on the cylinders $[0,\infty) \times S(V)$, $(-\infty , 0] \times S(V)$ by taking the product with the standard metric $(dt)^2$ on the intervals $[0 , \infty), (-\infty , 0]$. Choose a collar neighbourhood $(-\epsilon, 0] \times S(V)$ of $\partial E_{M'}$ and on this neighbourhood choose the fibrewise metric $(dt)^2 + g_{S(V)}$ and extend this to a fibrewise metric on all of $E_{M'}$. Construct a fibrewise metric similarly for $E_{N'}$. Clearly the fibrewise metrics on $E_{M'}, E_{N'}$ can be extended to a family of cylindrical end metrics on the families $E_{\hat{M}}, E_{\hat{N}}$, which we denote by $\{g_{\hat{M},b}\}_{b \in B}, \{g_{\hat{N},b}\}_{b \in B}$, or more simply just as $g_{\hat{M}}, g_{\hat{N}}$.

For each $r > 0$, let $\hat{M}(r) = \hat{M} \setminus (r+2,\infty) \times S^3$ and let $\hat{N}(r) = \hat{N} \setminus (\infty, -r-2) \times S^3$. Let $\hat{X}(r)$ be the $4$-manifold obtained by identifying the end portions of the necks of $\hat{M}(r)$ and $\hat{N}(r)$ via the map $[r,r+2] \times S^3 \to [-r-2 , -r] \times S^3$ given by $t \mapsto t - 2r-2$. For each $r > 0$, $\hat{X}(r)$ is diffeomorphic to $X = M \# N$, but the family of metrics $g_{\hat{X}(r)} = g_{\hat{M}} \#_r g_{\hat{N}}$ depends on $r$. As $r \to \infty$ the length of the neck joining $M$ and $N$ is stretched out. We can carry out this construction in families giving a family $E_{\hat{X}(r)} \to B$ which depends on $r$.

The family $E_{\hat{X}(r)}$ that we have just constructed depends on numerous choices (for example, on the choice of tubular neighbourhoods $U_M, U_N$). However it is not hard to see that up to isomorphism of families, the only data on which $E_{\hat{X}(r)}$ depends is the tuple $(E_M , E_N , \iota_M , \iota_N , \phi)$, where two families are considered isomorphic if there is a diffeomorphism of their total spaces covering the identity on $B$. We will usually just write the family as $E_X$, hiding the dependence of the various choices involved in its construction.

Let $\mathfrak{s}_M, \mathfrak{s}_N$ be Spin$^c$-structures on $M,N$ and let $\mathfrak{s}_X$ be the Spin$^c$-structure on $X$ obtained as the connected sum of $\mathfrak{s}_M$ and $\mathfrak{s}_N$. We will assume that $\mathfrak{s}_M,\mathfrak{s}_N$ are monodromy invariant under the monodromy action of $\pi_1(B)$ on the set of Spin$^c$-structures on $M$ and $N$ induces by the families $E_M,E_N$. Put another way, we assume $\mathfrak{s}_M,\mathfrak{s}_N$ extend to give continuously varying families of Spin$^c$-structures on the fibres of $E_M,E_N$. Then likewise $\mathfrak{s}_X$ extends to a continuous family of Spin$^c$-structures on the fibres of $E_{\hat{X}(r)}$.

Let us first consider the $\mathbb{Z}_2$-valued families Seiberg-Witten invariant. We will assume that
\[
d(M,\mathfrak{s}_M) = \frac{ c_1(\mathfrak{s}_M)^2 - \sigma(M)}{4} -1 - b^+(M) + b_1(M)
\]
is even and non-negative, and define $m \ge 0$ by $d(M,\mathfrak{s}_M)=2m$.
Next, we make the following assumptions about $N$. Assume that:
\[
c_1(\mathfrak{s}_N)^2 = \sigma(N)
\]
so that the index of the Spin$^c$-Dirac operator on $N$ is zero. Assume in addition that:
\[
0 < b^+(N) \le dim(B) = d.
\]
Recall also that we are assuming $b_1(N) = 0$. It follows that
\[
d(N,\mathfrak{s}_N) = \frac{ c_1(\mathfrak{s}_N)^2 - \sigma(N)}{4} -1 - b^+(N) + b_1(N) = -1 - b^+(N) < 0.
\]
Let $\mathcal{M} = \mathcal{M}(E_{\hat{X}(r)} , g_{\hat{X}(r)} , \eta_{\hat{X}(r)} , \mathfrak{s}_X )$ be the Seiberg-Witten moduli space for the family $E_{\hat{X}(r)}$ and a choice of generic perturbation $\eta_{\hat{X}(r)}$ (for now $\eta_{\hat{X}(r)}$ is an arbitrary generic perturbation. Later, we will choose $\eta_{\hat{X}(r)}$ much more carefully). Then
\begin{equation*}
\begin{aligned}
dim(\mathcal{M}) &= d(M , \mathfrak{s}_M) + d(N , \mathfrak{s}_N) + 1 + dim(B) \\
&= 2m + dim(B) - b^+(N).
\end{aligned}
\end{equation*}
Note that for a well-defined families Seiberg-Witten invariant attached to the family $E_X$ we need to assume that:
\[
b^+(M) > ( dim(B) - b^+(N) ) + 1.
\]
Let us set
\[
s = dim(B) - b^+(N) \ge 0.
\]
So $dim(\mathcal{M}) = 2m + s$ and we require $b^+(M) > s+1$.

Let $\theta \in H^s( B , \mathbb{Z}_2)$.
Recall that the families invariant $FSW^{\mathbb{Z}_2}(E_X , \mathfrak{s}_X , \theta)$ is defined in the following cases:
\begin{itemize}
\item[(i)]{$d(M , \mathfrak{s}_M) = 2m = 0$, or}
\item[(ii)]{$b_1(M) = 0$ (hence also $b_1(X)=0$) and $\mathfrak{s}_X$ extends to a Spin$^c$-structure on $T(E_X/B)$, which we denote by $\tilde{\mathfrak{s}}_X$.}
\end{itemize}

The main gluing theorem of this paper is as follows:

\begin{theorem}[$\mathbb{Z}_{2}$-valued gluing formula]\label{thm:rub}
Suppose that either
\begin{itemize}
\item[(i)]{$d(M , \mathfrak{s}_M) = 2m = 0$, or}
\item[(ii)]{$b_1(M) = b_1(X) = 0$ and $\mathfrak{s}_X$ extends to a Spin$^c$-structure on $T(E_X/B)$.}
\end{itemize}
Then we have:
\begin{itemize}
\item[(1)]{For any $\theta \in H^s(B , \mathbb{Z}_2)$, we have
\[
FSW^{\mathbb{Z}_2}_m(E_X , \mathfrak{s}_X , \theta ) = SW( M , \mathfrak{s}_M) \cdot \langle \theta \cup w_{b^+(N)}(H^+(N)) , [B] \rangle
\]
in $\mathbb{Z}_{2}$.
Here $SW( M , \mathfrak{s}_M)$ denotes the mod $2$ Seiberg-Witten invariant.
In particular, in case (ii), $FSW^{\mathbb{Z}_2}(E_X , \mathfrak{s}_X , \theta )$ does not depend on the choice of extension $\tilde{\mathfrak{s}_X}$ of $\mathfrak{s}_X$ to a Spin$^c$-structure on $T(E_X/B)$.}
\item[(2)]{For any integer $k > 0$ and any $\theta \in H^{s+2k}(B , \mathbb{Z}_2)$, we have
\[
FSW^{\mathbb{Z}_2}_{m-k}( E_X , \mathfrak{s}_X , \theta) = 0.
\]
}
\end{itemize}
\end{theorem}

\begin{remark}
In case (ii), the fact that $FSW^{\mathbb{Z}_2}_{m}(E_X , \mathfrak{s}_X , \theta )$ does not depend on the choice of $\tilde{\mathfrak{s}}_X$ for any $\theta \in H^s(B , 
\mathbb{Z}_2)$ can also be deduced from the fact that $FSW^{\mathbb{Z}_2}_{m-k}(E_X , \mathfrak{s}_X , \theta ) = 0$ for any $\theta \in H^{s+2k}(B , \mathbb{Z}_2)$, where $k>0$.
\end{remark}

We will prove parts (1) and (2) of Theorem~\ref{thm:rub} separately.

Similarly, we can also get $\rho$-twisted version of Theoreom~\ref{thm:rub}.
Instead of assuming that $\mathfrak{s}_M,\mathfrak{s}_N$ are monodromy invariant under the monodromy action of $\pi_1(B)$,
we here assume that the monodromy preserves $\mathfrak{s}_M, \mathfrak{s}_N, \mathfrak{s}_X$ up to a common sign $\rho : \pi_1(B) \to \mathbb{Z}_2$.
Let us keep all other assumptions supposed before the statement of Theorem~\ref{thm:rub}.
Then we have the following:

\begin{theorem}[$\rho$-twisted gluing formula]\label{thm:rubtwisted}
Assume that the monodromy preserves $\mathfrak{s}_M, \mathfrak{s}_N, \mathfrak{s}_X$ up to a common sign $\rho : \pi_1(B) \to \mathbb{Z}_2$.
Assume also that $d(M , \mathfrak{s}_M) = 2m = 0$, and suppose that we are given $\theta \in H^s(B , \mathbb{Z}_2)$.
Then we have:
\[
FSW^{\mathbb{Z}_2}(E_X , \mathfrak{s}_X , \theta, \rho) = SW( M , \mathfrak{s}_M) \cdot \langle \theta \cup w_{b^+(N)}(H^+(N) \otimes \mathbb{R}_\rho ) \rangle
\]
in $\mathbb{Z}_{2}$, where $\mathbb{R}_\rho \to B$ is the real line bundle classified by $\rho \in H^1(B , \mathbb{Z}_2)$.
\end{theorem}

As well as the $\mathbb{Z}_{2}$ case, we can show a similar gluing result also for $\mathbb{Z}$-valued invariant.
Besides all assumptions supposed before the statement of Theorem~\ref{thm:rub}, let us assume that $B, det^+(E_M), det^+(E_N)$ are oriented and homology orientations of $M$ and $N$ are fixed.
Note that orientations of $X=M\#N$ and $det^+(E_X)$ are automatically determined.

\begin{theorem}[$\mathbb{Z}$-valued gluing formula]\label{thm:rubz}
Suppose that either
\begin{itemize}
\item[(i)]{$d(M , \mathfrak{s}_M) = 2m = 0$, or}
\item[(ii)]{$b_1(M) = b_1(X) = 0$ and $\mathfrak{s}_X$ extends to a Spin$^c$-structure on $T(E_X/B)$.}
\end{itemize}
Then we have:
\begin{itemize}
\item[(1)]{For any $\theta \in H^s(B , \mathbb{Z})$, we have
\[
FSW^{\mathbb{Z}}_m(E_X , \mathfrak{s}_X , \theta ) = SW( M , \mathfrak{s}_M) \cdot \langle \theta \cup e(H^+(N)) , [B] \rangle
\]
in $\mathbb{Z}$.
Here $SW( M , \mathfrak{s}_M)$ denotes the Seiberg-Witten invariant and $e(H^+(N))$ the Euler class of $H^+(N)$.
In particular, in case (ii), $FSW^{\mathbb{Z}}(E_X , \mathfrak{s}_X , \theta )$ does not depend on the choice of extension $\tilde{\mathfrak{s}}_X$ of $\mathfrak{s}_X$ to a Spin$^c$-structure on $T(E_X/B)$.}
\item[(2)]{For any integer $k > 0$ and any $\theta \in H^{s+2k}(B , \mathbb{Z})$, we have
\[
FSW^{\mathbb{Z}}_{m-k}( E_X , \mathfrak{s}_X , \theta) = 0.
\]
}
\end{itemize}
\end{theorem}

\begin{remark}
One of the differences between Ruberman's original gluing theorem \cite{rub1,rub2,rub3} and ours is that we do not assume that our families of $M\#N$ consist of {\it trivial} families of $M$ and families of $N$.
In \cite{rub1,rub2,rub3}, Ruberman considered diffeomorphisms which are supported on $N$, strictly speaking on $N'$, and therefore the corresponding families obtained by mapping tori can be regarded as the families connected sum of trivial families of $M$ and (typically non-trivial) families of $N$.
On the other hand, we shall also consider diffeomorphisms on $M\#N$ acting non-trivially not only on $N$ but also on $M$.
In the subsequent sections, this generalisation is effectively used when we shall give examples of $4$-manifolds which admit positive scalar curvature metrics but do not admit $\mathbb{Z}_{2}$-invariant positive scalar curvature metrics for suitable $\mathbb{Z}_{2}$-actions on them.
\end{remark}

\section{Proof of gluing formula, first step}\label{sec:step1}

We first concern ourselves with part (1) of Theorem~\ref{thm:rub}. In this section we will show that it suffices to prove part (1) of Theorem~\ref{thm:rub} in the case that $s = 0$, that is, in the case $dim(B) = b^+(N)$.\\

Suppose the theorem holds in the case $s=0$. Now we will prove the general theorem. Let $\theta \in H^s(B , \mathbb{Z}_2)$. By the solution to the Steenrod problem, $\theta$ is Poincar\'e dual to a homology class of the form $f_*[S] \in H_{b^+(N)}( B , \mathbb{Z}_2)$, where $S$ is a compact smooth manifold of dimension $dim(S) = b^+(N)$ and $f : S \to B$ is a smooth map. Then
\begin{equation*}
\begin{aligned}
FSW_m^{\mathbb{Z}_2}(E_X , \mathfrak{s}_X , \theta ) &= \langle c_1(\mathcal{L})^m \cup \pi^* \theta, [\mathcal{M}] \rangle \\
&= \langle \pi_* ( c_1(\mathcal{L})^m \cup \pi^* (\theta)) , [B] \rangle \text{ (by definition of the wrong-way map $\pi_*$)} \\
&= \langle \pi_* ( c_1(\mathcal{L})^m ) \cup \theta , [B] \rangle \\
&= \langle \pi_* ( c_1(\mathcal{L})^m ) , [S] \rangle \text{ ($\theta$ is Poinar\'e dual to $[S]$)} \\
&= \langle  c_1(\mathcal{L})^m  , [\mathcal{M}|_S] \rangle \text{ (by definition of the wrong-way map $\pi_*$)} \\
&= FSW_m^{\mathbb{Z}_2}( E_X|_S , \mathfrak{s}_X , 1).
\end{aligned}
\end{equation*}
Now we consider the pullbacks $f^*(E_M), f^*(E_N)$ of the families to $S$. By assumption, Theorem~\ref{thm:rub} part (1) holds on $S$, since $dim(S) = b^+(N)$, hence
\begin{equation*}
\begin{aligned}
FSW_m^{\mathbb{Z}_2}(E_X , \mathfrak{s}_X , \theta) &= FSW_m^{\mathbb{Z}_2}( E_X|_S , \mathfrak{s}_X , 1) \\
&= SW( M , \mathfrak{s}_M) \cdot \langle w_{b^+(N)}(H^+(N)) , [S] \rangle \text{ (by Theorem~\ref{thm:rub} (1) on $S$)} \\
&= SW( M, \mathfrak{s}_M) \cdot \langle \theta \cup w_{b^+(N)}(H^+(N)) , [B] \rangle \text{ ($\theta$ is Poinar\'e dual to $[S]$)}.
\end{aligned}
\end{equation*}
The last equality is exactly part (1) of Theorem~\ref{thm:rub} on $B$, which is what we wanted to prove.

\begin{remark}
Almost the same argument works in the case of $\mathbb{Z}$-coefficients. The main difference is that the Steenrod problem can only be solved up to a multiple. That is, given $\theta \in H^s( B , \mathbb{Z})$, there is some positive integer $N$ such that we can represent $N\theta$ as $f_*[S]$ for some $f : S \to B$. But if the $\mathbb{Z}$-valued version of Theorem~\ref{thm:rub} holds for $N\theta$, then by cancelling out the factor of $N$ on both sides, we get that the theorem also holds for $\theta$ itself.
\end{remark}

So for the remainder of the proof of part (1) of Theorem~\ref{thm:rub} we may assume that:
\[
dim(B) = d = b^+(N).
\]

\section{Gluing theory for the Seiberg-Witten equations}\label{sec:gluing}

\subsection{Review of gluing theory}
\label{section Review of gluing theory}

To proceed further with the proof of Theorem \ref{thm:rub}, we need to employ a parametrised version of the gluing construction for the Seiberg-Witten equations on connected sums. Our main reference will be Nicolaescu's book~\cite{nic}, and we will refer to \cite{nic} as ``Nicolaescu's book" or just simply ``Nicolaescu".
In this section, let us first briefly review the ordinary case of this construction, i.e. the unparametrised case.

We have $4$-manifolds $M,N$ and their connected sum $X = M \# N$. Recall that we defined the cylindrical end $4$-manifolds $\hat{M}, \hat{N}$ and we attached them to form $\hat{X}(r)$, where $r$ is a positive real number. Each $\hat{X}(r)$ is diffeomorphic to $X$, but the length of the neck of $\hat{X}(r)$ grows proportionally with $r$. To make the analysis work well on cylindrical end manifolds, one considers weighted Sobolev spaces $L^{k,p}_{\mu}$, where $\mu \in \mathbb{R}$ is the weight, and the extended Sobolev spaces $L^{k,p}_{\mu, ex}$  following Atiyah-Patodi-Singer~\cite{aps}, called the asymptotically cylindrical sections in Nicolaescu's book (page 299 in Subsection~4.1.4), to deal with asymptotic values.

Here $\eta_{\hat{M}}$ is the self-dual $2$-form perturbation, which is assumed to be supported away from the neck and $\mathfrak{s}_{\hat{M}}$ is an asymptotically cylindrical Spin$^c$-structure on $\hat{M}$. Isomorphism classes of such Spin$^c$-structures are in bijection with Spin$^c$-structures on $M$ (this follows from the fact that our $3$-manifold used for gluing, denoted by $\partial_\infty \hat{M}$ in Nicolaescu's book, is just $S^3$ and $H^1(S^3 , \mathbb{Z}) = H^2(S^3 , \mathbb{Z}) = 0$. See Nicolaescu for more details on asymptotically cylindrical Spin$^c$-structures). Therefore we will simply identify $\mathfrak{s}_{\hat{M}}$ with $\mathfrak{s}_M$ without further mention. We have isomorphisms
\[
H^2_{L^2}(\hat{M}) \cong H^2( \hat{M} , [0 , \infty) \times S^3 , \mathbb{R}) \cong H^2(\hat{M} , \mathbb{R}) \cong H^2(M , \mathbb{R}),
\]
where $H^2_{L^2}(\hat{M})$ is the space of $L^2$-integrable harmonic $2$-forms on $\hat{M}$. The first isomorphism is shown by Atiyah-Patodi-Singer~\cite{aps}, the second isomorphism is just the long exact sequence for cohomology of the pair $( \hat{M} , [0 , \infty) \times S^3)$ and the third isomorphism can be obtained from the Mayer-Vietoris sequence. Thus the $L^2$-second Betti number of $\hat{M}$ coincides with the ordinary second Betti number of $M$. Note also that this isomorphism depends only on the choice of metric $g_{\hat{M}}$.

Next, observe that $H^2_{L^2}(\hat{M})$ is equipped with a natural bilinear form, the $L^2$-intersection form:
\[
(\alpha , \beta) = \int_{\hat{M}} \alpha \wedge \beta.
\]
On the other hand, Poincar\'e Lefschetz duality implies that the cup product determines a non-degenerate bilinear form on the image of $H^2( \hat{M} , [0,\infty) \times S^3 , \mathbb{R} )$ in $H^2( \hat{M} , \mathbb{R} )$. In our case, the image is all of $H^2(\hat{M} , \mathbb{R})$ and under the isomorphism with $H^2(M , \mathbb{R})$, coincides with the usual intersection form on $M$. The isomorphism given by Atiyah-Patodi-Singer between $H^2_{L^2}(\hat{M})$ and the image of $H^2( \hat{M} , [0,\infty) \times S^3 , \mathbb{R} )$ in $H^2( \hat{M} , \mathbb{R} )$ can be shown to respect the intersection forms. Therefore, in our case we have an isomorphism $H^2_{L^2}(\hat{M}) \cong H^2(M , \mathbb{R})$ such that the $L^2$-intersection form on $H^2_{L^2}(\hat{M})$ coincides with the topological intersection form on $H^2(M , \mathbb{R})$. Let $H^+_{L^2}(\hat{M})$ denote the subspace of $H^2_{L^2}(\hat{M})$ consisting of self-dual $L^2$-harmonic $2$-forms. It follows immediately that the dimension of $H^+_{L^2}(\hat{M})$ equals $b^+(M)$.

\begin{remark}
In the families setting, given the family $E_{\hat{M}} \to B$ and family $\{ g_{\hat{M}} \}$ of fibrewise metrics, we have that $H^+_{L^2}(\hat{M})$ and $H^2(M , \mathbb{R})$ can be thought of as vector bundles on $B$, equipped with bilinear forms. The isomorphism $H^+_{L^2}(\hat{M}) \cong H^2(M , \mathbb{R})$ depends only on the choice of metric $g_{\hat{M}}$ and so in the families setting, we get that $H^+_{L^2}(\hat{M}) \cong H^2(M , \mathbb{R})$ is an isomorphism of vector bundles with bilinear forms.
\end{remark}

It is shown in Nicolaescu that if $b^+(M) > 0$, then the perturbation $\eta_{\hat{M}}$ can be chosen so as to avoid any reducible solutions of the Seiberg-Witten equations, and for generic such $\eta_{\hat{M}}$, the moduli space $\mathcal{M}( \hat{M} , g_{\hat{M}} , \eta_{\hat{M}} , \mathfrak{s}_{\hat{M}} , \mu )$ is a smooth compact manifold of dimension $d(M , \mathfrak{s})$ (the same dimension as the Seiberg-Witten moduli space for $(M , \mathfrak{s}_M)$). Compactness depends crucially on the assumption that $\partial_\infty \hat{M} = S^3$. More precisely, within the space of $2$-form perturbations on $\hat{M}$ supported away from the neck, the ones for which a reducible solution exists form a closed subspace of codimension $b^+(M)$. We will discuss this in more detail in Section~\ref{sec:gluefamily}.

Given a solution
\[
\hat{C} = (\hat{A} , \hat{\psi}) \in \mathcal{M}( \hat{M} , g_{\hat{M}} , \eta_{\hat{M}} , \mathfrak{s}_{\hat{M}} , \mu )
\]
of the Seiberg-Witten equations on $\hat{M}$, one can consider the deformation theory of $\hat{C}$. One obtains a three-term deformation complex, denoted by $\hat{\mathcal{K}}_{\hat{C}}$ in Nicolaescu's book, which controls the deformation theory. For $i=0,1,2$, let $H^i_{\hat{C}}$ denote the $i$-th cohomology of $\hat{\mathcal{K}}_{\hat{C}}$. These groups have the usual interpretations: $H^0_{\hat{C}}$ is the Lie algebra of infinitesimal automorphisms, hence $H^0_{\hat{C}} = 0$ if $\hat{C}$ is irreducible and $H^0_{\hat{C}} = \mathbb{R}$ if $\hat{C}$ is reducible. $H^1_{\hat{C}}$ is the virtual tangent space and $H^2_{\hat{C}}$ is the obstruction space. The local structure of the moduli space around $\hat{C}$ is given by a Kuranishi model $H^1_{\hat{C}} \supseteq U \buildrel Q \over \longrightarrow H^2_{\hat{C}}$, for some obstruction map $Q$.

As in the compact case, an application of Sard-Smale shows that for generic perturbations $\eta_{\hat{M}}$, we have $H^2_{\hat{C}} =0$ for every {\em irreducible} solution of the $\eta_{\hat{M}}$-perturbed Seiberg-Witten equations on $\hat{M}$.

Now we turn to gluing (see Nicolaescu's book \textsection 4.5). The idea is roughly as follows: let $\hat{C}_M$ be a solution of the Seiberg-Witten equations on $\hat{M}$ and $\hat{C}_N$ be a solution of the Seiberg-Witten equations on $\hat{N}$ which agree asymptotically. If $\hat{C}_M$ and $\hat{C}_N$ were cylindrical, then we could identify them on the necks of $\hat{M}(r), \hat{N}(r)$ to obtain a genuine solution of the Seiberg-Witten equations on $\hat{X}(r)$. In general, $\hat{C}_M, \hat{C}_N$ are only asymptotically cylindrical. We can still glue them together on $\hat{X}(r)$ with the aid of cutoff functions to form a glued configuration $\hat{C}_r = \hat{C}_M \#_r \hat{C}_N$ on $\hat{X}(r)$. But the cutoff functions introduce error terms so that $\hat{C}_r$ is only an approximate solution of the Seiberg-Witten equations on $\hat{X}(r)$. If $r$ is very large, then the error is very small and $\hat{C}_r$ is very close to being a solution. If $\eta_{\hat{M}}$, $\eta_{\hat{N}}$ are perturbations on $\hat{M},\hat{N}$ supported away from the necks, then they can be glued together to form a perturbation $\eta_{\hat{X}(r)} = \eta_{\hat{M}} \#_r \eta_{\hat{N}}$ on $\hat{X}(r)$ (the gluing is straightforward because $\eta_{\hat{M}},\eta_{\hat{N}}$ both vanish along the necks). Using a variant of the Kuranishi model, one obtains a local description on the moduli space of genuine solutions which are sufficiently close to $\hat{C}_r$. Since this step is of crucial importance, we will carefully review the construction in the next section.

\begin{remark}
Henceforth the notation $\#_{r}$ means the gluing by a specified cutoff functions according to Nicolaescu's book.
Note that in some literatures the operation $\#_{r}$ may be called a {\it pre-gluing} rather than gluing.
(In such a case the word ``gluing " is used only for a genuine solution near an approximating configuration obtained from $\#_{r}$.)
However in this paper we often call the operation $\#_{r}$ simply a gluing.
\end{remark}

\subsection{Gluing monopoles: local theory}
\label{sec:glueloc}

In Section~\ref{sec:gluefamily}, we shall describe a gluing argument for families.
Roughly speaking, our argument is to estimate some errors occurring from the base space direction compared with the unparametrised case.
As a preliminary of this forthcoming argument, in this section we shall recall some estimates in the local theory on gluing in unparameterized case.
We will mostly adopt the notations of Section~4 in Nicolaescu's book.

Choose strongly cylindrical connections (see Nocolaescu's book Example~4.1.2) $\hat{A}_{0,M}, \hat{A}_{0,N}$ on $det(\mathfrak{s}_M), det(\mathfrak{s}_N)$ respectively which can be assumed compatible (recall that {\em compatible} means their asymptotic values agree) and set
\[
\hat{A}_0 = \hat{A}_0(r) = \hat{A}_{0,M} \#_r \hat{A}_{0,N}.
\]
Let $\hat{\mathcal{C}}_{\mu,ex}(\hat{M}), \hat{\mathcal{C}}_{\mu,ex}(\hat{N})$ be the configuration spaces consisting of asymptotically cylindrical data, defined on page 369 in Subsection~4.3.1 in Nicolaescu's book.
For $\hat{C} \in \hat{\mathcal{C}}_{\mu,ex}(\hat{M})$ or $\hat{\mathcal{C}}_{\mu,ex}(\hat{N})$, we will denote by $\partial_\infty \hat{C}$ the asymptotic value of $\hat{C}$.
If $\hat{C}_M \in \hat{\mathcal{C}}_{\mu,ex}(\hat{M})$ and $\hat{C}_N \in \hat{\mathcal{C}}_{\mu,ex}(\hat{N})$ are two smooth monopoles such that
\[
\partial_\infty \hat{C}_{M} = \partial_\infty \hat{C}_{N},
\]
then we can form
\[
\hat{C}_r = (\hat{\psi}_r , \hat{A}_r) = \hat{C}_M \#_r \hat{C}_N = (\hat{\psi}_M \#_r \hat{\psi}_N , \hat{A}_M \#_r \hat{A}_N).
\]
on $\hat{X}(r)$, where $\hat{X}(r)$ is the closed $4$-manifold obtained from $\hat{M}$ and $\hat{N}$ having a neck of length $r$.
Of course, the resulting configuration $\hat{C}_r$ is usually not a solution of the Seiberg-Witten equations since we used cutoff functions to form $\hat{C}_{r}$ as we explained in Section~\ref{section Review of gluing theory}.
We would like to show there exist genuine monopoles near $\hat{C}_r$. For this we consider a variant of the Kuranishi model for a non-linear Fredholm map $f : V \to W$ between Banach spaces. Normally one assumes $f(0) = 0$ and studies the structure of the zero set $f^{-1}(0)$ near $0$. In our case, we are still interested in the zero set $f^{-1}(0)$ but we can have $f(0) \neq 0$.

Let $\hat{C}_r \in \hat{\mathcal{C}}(\hat{X}(r))$ be the configuration obtained by gluing two monopoles using cutoff functions, as above.
Let us denote by $S^{+}$ and $S^{-}$ the positive and negative spinor bundles respectively.
We also use the notation $\wedge^2_+$ to indicate the self-dual part.
As in Subsection~4.5.2 in Nicolaescu's book,
form the non-linear map
\[
\mathcal{N} : L^{2,2}( \hat{X}(r) , S^+ \oplus iT^*\hat{X}(r) ) \to L^{1,2}( \hat{X}(r) , S^- \oplus i\wedge^2_+ T^*\hat{X}(r) \oplus i\mathbb{R} )
\]
given by
\[
\mathcal{N}(\hat{C}) = \widehat{SW}( \hat{C} + \hat{C}_r) \oplus \mathcal{L}_{\hat{C}_r}^*( \hat{C} ).
\]
Here the term $\widehat{SW}( \hat{C} + \hat{C}_r)$ denotes the Seiberg-Witten equations and $\mathcal{L}_{\hat{C}_r}^*( \hat{C} )$ does the gauge fixing condition relative to $\hat{C}_r$.
Namely, when we write $\hat{C}$ and $\hat{C}_{r}$ as $\hat{C} = (\hat{a}, \hat{\psi})$ and $\hat{C}_{r} = (\hat{a}_{r}, \hat{\psi}_{r})$ respectively, $\widehat{SW}$ and $\mathcal{L}_{\hat{C}_r}^*$ are defined as
\[
\widehat{SW}( \hat{C} + \hat{C}_r) = (F_{\hat{A}_{0}+\hat{a}+\hat{a}_{r}}^{+}-q(\hat{\psi}+\hat{\psi}_{r},\hat{\psi}+\hat{\psi}_{r}),D_{\hat{A}_{0}+\hat{a}+\hat{a}_{r}}(\hat{\psi}+\hat{\psi}_{r}))
\]
and
\[
\mathcal{L}_{\hat{C}_r}^*( \hat{C} ) = -2d^{\ast} \hat{a} - i{\rm Im}\langle\hat{\psi}_{r},\hat{\psi}\rangle.
\]
(On this gauge fixing condition, see Nicolaescu Subsection~2.2.2.)
Although here we just consider the unperturbed Seiberg-Witten equations, the necessary modification to deal with perturbed Seiberg-Witten equations will be done in Section~\ref{sec:gluefamily}, where we describe families gluing with suitable perturbations.
Note that $\mathcal{N}$ is a well-defined smooth map because of Sobolev multiplication theorem for $L^{2,2} \times L^{2,2} \to L^{1,2}$.
We are of course interested in the zero set $\mathcal{N}^{-1}(0)$.

Denote by $\hat{T}_r = \hat{T}_{\hat{C}_r}$ the linearisation of $\mathcal{N}$ at $0$, which is given by the map
\[
\hat{T}_r( \hat{C} ) = D_{\hat{C}_r}\widehat{SW}(\hat{C}) \oplus \mathcal{L}_{\hat{C}_r}^*(\hat{C})
\]
whose domain and codomain are same to those of $\mathcal{N}$.
Here the notation $D_{\hat{C}_r}$ denotes the differential at $\hat{C}_{r}$.
Observe that from $\hat{C}_r = \hat{C}_1 \#_r \hat{C}_2$, we find that
\[
\hat{T}_r = \hat{T}_{\hat{C}_1} \#_r \hat{T}_{\hat{C}_2}.
\]
Here the notation $\#_{r}$ means the gluing operation given by using the cutoff functions which we used to define $\hat{C}_{r}$.
Now write $\mathcal{N}$ in the form
\[
\mathcal{N}(\hat{C}) = \mathcal{N}(0) + \hat{T}_r(\hat{C}) + R(\hat{C}),
\]
so $R$ is the remainder term in the linear approximation of $\mathcal{N}$ at $0$.

\begin{lemma}[Lemma 4.5.6 of Nicolaescu's book]
\label{lem Lemma 4.5.6 of Nicolaescu's book}
There exists a constant $C > 0$ which depends only on the geometry of $\hat{M}, \hat{N}$ such that
\[
||R(\hat{C})||_{L^{1,2}} \le C r^{3/2} ||\hat{C}||^2_{L^{2,2}}
\]
and
\[
|| R(\hat{C}) - R(\hat{C}') ||_{L^{1,2}} \le Cr^{3/2}\left( ||\hat{C}||_{L^{2,2}} + ||\hat{C}' ||_{L^{2,2}}\right) ||\hat{C} - \hat{C}'||_{L^{2,2}} 
\]
hold for any $\hat{C}, \hat{C}' \in L^{2,2}( \hat{X}(r) , S^+ \oplus iT^*\hat{X}(r) )$.
\end{lemma}

\begin{remark}
To state this kind of inequalities between Sobolev spaces correctly, we choose connections of bundles for whose section we consider Sobolev norms so that, on the neck part, the connections are the pull-back of some fixed connections on $S^{3}$, and define Sobolev norms using the covariant derivatives corresponding to such connections.
Then we can compare Sobolev spaces varying $r$, and the statement of Lemma~\ref{lem Lemma 4.5.6 of Nicolaescu's book}, involving $r$, makes sense.
\end{remark}

\begin{remark}
The proof of Lemma 4.5.6 is a consequence of the inequality (4.5.1) on page 430 of Nicolaescu's book. The inequality (4.5.1) is asserted to hold for $p$ in the range $1 < p \le 6$, however we were only able to prove the inequality for $1 < p \le 4$ because of the limitation coming from the Sobolev embedding. Fortunately the proof of Nicolaescu's Lemma 4.5.6 only needs the $p=4$ case of (4.5.1), so this is not a problem for Lemma~\ref{lem Lemma 4.5.6 of Nicolaescu's book}.
\end{remark}

Now we introduce the following notation:

\begin{definition}
\label{def of Xk}
\[
X^k_+ = L^{k,2}(\hat{X}(r) , S^+ \oplus iT^*\hat{X}(r)), \quad X^k_- = L^{k,2}(\hat{X}(r) , S^- \oplus i\wedge^2_+ T^*\hat{X}(r) \oplus i\mathbb{R}),
\]
and also set
\[
X^k = X^k_+ \oplus X^k_-.
\]
\end{definition}

Form the closed, densely defined operator
\[
\hat{L}_r : X^0 \to X^0
\]
with block decomposition
\[
\hat{L}_r = \left[ \begin{matrix} 0 & \hat{T}^*_r \\ \hat{T}_r & 0 \end{matrix} \right].
\]
It may help to think of $\hat{L}_r$ as being a ``Dirac operator". It is self-adjoint and induces bounded Fredholm operators
\[
X^{k+1} \to X^k.
\]
This is obvious because $\hat{X}(r)$ is a closed manifold, so we can use ordinary Hodge theory.

\begin{definition}
\label{def of Hr}
Denote by $H_r$ the subspace of $X^0$ spanned by the eigenvectors of $\hat{L}_r$ corresponding to eigenvalues in the interval $(-r^{-2} , r^{-2})$.
The decomposition $X^0 = X^0_+ \oplus X^0_-$ induces a decomposition
\[
H_r = H_r^{+} \oplus H_r^{-}.
\]
\end{definition}

Note that $H_r$ consists entirely of smooth sections.

\begin{remark}
In the ordinary Kuranishi model, instead of $H_r^+$, we would normally just use $Ker(\hat{T}_r)$ and instead of $H_r^-$, we would normally use $Coker(\hat{T}_r) = Ker(\hat{T}_r^*)$. However, in this situation it is more convenient to use $H_r^\pm$ because we can use linear gluing theory, which gives us means of computing the dimensions of these spaces.
\end{remark}

\begin{definition}
\label{def of Y}
Let $Y^0(r)$ be the $L^2$-orthogonal complement of $H_r$ in $X^0$. Define $Y^k(r) = Y^0(r) \cap X^k$.
Then $Y^k(r)$ decomposes as
\[
Y^k(r) = Y^k_+(r) \oplus Y^k_-(r).
\]
\end{definition}

We also have decompositions
\[
X^k_+ = H_r^+ \oplus Y^k_+(r), \quad X^k_- = H_r^- \oplus Y^k_-(r).
\]
We shall sometimes simply write $Y^k, Y^{k}_{\pm}$ instead of  $Y^k(r), Y^k_{\pm}(r)$.

\begin{remark}
\label{rem Tr preserves}
Note that $\hat{T}_r$ sends $H_r^+$ to $H_r^-$ and $Y^0_+(r)$ to $Y^0_-(r)$. Similarly $\hat{T}_r^*$ sends $H_r^-$ to $H_r^+$ and $Y^0_-(r)$ to $Y^0_+(r)$. To see this, note that since $L_r$ is self-adjoint, then the eigenvalues of $\hat{L}_r^2$ are all non-negative real numbers. Moreover, the $\lambda^2$-eigenspace of $\hat{L}_r^2$ is easily seen to be the sum of the $\lambda$ and $-\lambda$ eigenspaces of $\hat{L}_r$. It follows that $H_r$ can alternatively be characterised as the span of the eigenspaces of $\hat{L}_r^2$ with eigenvalues in the range $[0 , r^{-4})$. But note that
\[
\hat{L}_r^2 = \left[ \begin{matrix} \hat{T}_r^* \hat{T}_r & 0 \\ 0 & \hat{T}_r \hat{T}_r^* \end{matrix} \right].
\]
This verifies the earlier made claim that $H_r$ decomposes as $H_r = H_r^+ \oplus H_r^-$. Moreover, it characterises $H_r^+$ as the span of the eigenspaces of $\hat{T}_r^* \hat{T}_r$ with eigenvalues in the range $[ 0 ,r^{-4})$ and similarly $H_r^-$ as the span of the eigenspaces of $\hat{T}_r \hat{T}^*_r$ with eigenvalues in the range $[ 0 ,r^{-4})$. Now to prove the claim made in the beginning of this remark, it suffices to show that $\hat{T}_r$ sends the $\lambda^2$-eigenspace of $\hat{T}_r^* \hat{T}_r$ to the $\lambda^2$-eigenspace of $\hat{T}_r \hat{T}_r^*$ (and similarly for $\hat{T}_r^*$). Now if $\hat{T}_r^* \hat{T}_r x = \lambda^2 x$, then $\hat{T}_r \hat{T}_r^*( \hat{T}_r x) = \hat{T}_r (\hat{T}_r^* \hat{T}_r x) = \hat{T}_r( \lambda^2 x) = \lambda^2 ( \hat{T}_r x)$ (and a similar computation holds for $\hat{T}^*_r)$.
\end{remark}

\begin{definition}
\label{def of P and Q}
Denote by $P_{\pm}^{r}$ the $L^2$-orthogonal projections $X_{\pm}^0 \to H_r^{\pm}$ and let $Q_{\pm}^{r} = 1 - P_{\pm}^{r}$ be the projections to $Y^0_{\pm}(r)$.
\end{definition}

We often write $P_{\pm}^{r}$ and $Q_{\pm}^{r}$ simply as $P_{\pm}$ and $Q_{\pm}$.
Observe that $Q_{\pm}(X^k_{\pm}) = Y^k_{\pm}(r)$.

For each $\hat{C} \in X^0_+$, we decompose it as $\hat{C} = \hat{C}_0 + \hat{C}_1$, where
\[
\hat{C}_0 = P_+ \hat{C}, \quad \hat{C}_1 = Q_+ \hat{C}.
\]
By Remark~\ref{rem Tr preserves}, one can observe that
\[
P_- \hat{T}_r \hat{C} = \hat{T}_r P_+ \hat{C} = \hat{T}_r \hat{C}_0, \quad Q_- \hat{T}_r \hat{C} = \hat{T}_r Q_+ \hat{C} = \hat{T}_r \hat{C}_1.
\]
For every $k > 0$, $\hat{T}_r$ induces (by restriction) a bounded operator
\[
\hat{T}_r : Y^{k+1}_+ \to Y^k_-.
\]
Clearly this operator is invertible with bounded inverse and we denote the inverse as $S : Y^k_- \to Y^{k+1}_+$.
We note that the following estimate
(in Nicolaescu, this is the inequality (4.5.5), but there is a typo: it uses $||Su||_{L^{k+2,1}}$ instead of $||Su||_{L^{k+1,2}}$):

\begin{lemma}
\label{lem: estimate for S}
We have an estimate
\[
|| Su ||_{L^{k+1,2}} \le C_k r^2 ||u||_{L^{k,2}}
\]
for some $C_k > 0$ which is independent of $r$ and for all $u \in Y^k_-$
\end{lemma}

\begin{proof}
To obtain this estimate simply note that the eigenvalues of $\hat{T}^*_r\hat{T}_r$ restricted to $Y_+$ lie in the range $[r^{-4} , \infty)$.
\end{proof}

The equation $\mathcal{N}(\hat{C}) = 0$ is equivalent to the pair of equations
\[
P_- \mathcal{N}(\hat{C}) = 0, \quad Q_- \mathcal{N}(\hat{C}) = 0.
\]
Expanding $\mathcal{N}$ and $\hat{C}$, these equations become
\begin{equation*}
\begin{aligned}
Q_- \mathcal{N}(0) + \hat{T}_r \hat{C}_1 + Q_- R( \hat{C}_0 + \hat{C}_1) &= 0, \\
P_- \mathcal{N}(0) + \hat{T}_r \hat{C}_0 + P_- R(\hat{C}_0 + \hat{C}_1) & = 0.
\end{aligned}
\end{equation*}
Let us define $U \in Y^0_+$ to be
\[
U = -SQ_- \mathcal{N}(0).
\]
Applying $S$ to the first of the two equations above, we end up with
\[
\hat{C}_1 = U - SQ_-R(\hat{C}_0 + \hat{C}_1).
\]
For the moment, fix some $\hat{C}_0 \in H_r^+$. Then define $\mathcal{F} : Y^2_+ \to Y^2_+$ as
\[
\mathcal{F}(\hat{C}_1) = U - SQ_- R(\hat{C}_0 + \hat{C}_1).
\]
Note that $\mathcal{F}$ is well-defined as a map from $Y^2_+$ to itself.
First of all, it is easy to see that $U \in Y^2_+$, because $\mathcal{N}(0) \in X_-^1$, $Q_-$ sends $X^1_-$ to $Y^1_-$ and $S$ sends $Y^1_-$ to $Y^2_+$. Similarly $SQ_- R(\hat{C}_0 + \hat{C}_1) \in Y^2_+$ because $R(\hat{C}_0 + \hat{C}_1) \in X^1_-$.

Let us define
\[
B_1(r^{-4}) = \{ \hat{C}_1 \in Y^2_+ \; | \; ||\hat{C}_1||_{L^{2,2}} \le r^{-4} \} \subset Y^2_+(r)
\]
and similarly
\[
B_0(r^{-4}) = \{ \hat{C}_0 \in H_r^+ \; | \; ||\hat{C}_0||_{L^{2,2}} \le r^{-4} \} \subset H_r^+.
\]

\begin{remark}
Nicolaescu uses balls of radius $r^{-3}$ instead of $r^{-4}$. Unfortunately $r^{-3}$ does not seem sufficient to prove that $\mathcal{F}$ is a contraction.
The mistake seems to be in equation (4.5.8), page 438, which has a factor $r^{5/2}$, but the correct factor should be $r^{3/2}r^2 = r^{7/2}$.
\end{remark}

\begin{lemma}
\label{lem:fcontract}
For all sufficiently large $r$, and for all $\hat{C}_0 \in B_0(r^{-4})$, the map $\mathcal{F}$ sends $B_1(r^{-4})$ to itself and is a contraction mapping.
\end{lemma}

\begin{proof}
First we check that $\mathcal{F}$ sends $B_1(r^{-4})$ to itself, provided $\hat{C}_0 \in B_0(r^{-4})$.
Throughout we use $C$ to indicate a constant which may depend on the geometry of the $\hat{N}_i$, but is independent of $r$. The value of $C$ may increase from line to line.
If there is no risk of confusion, we drop the notations indicating the functional spaces from our norms.

We first have
\begin{equation*}
\begin{aligned}
|| \mathcal{F}(\hat{C}_1) || &\le ||S Q_- \mathcal{N}(0) || + ||SQ_- R(\hat{C}_0 + \hat{C}_1)|| \\
& \le C r^2 \left( ||Q_- \mathcal{N}(0) || + ||Q_- R(\hat{C}_0 + \hat{C}_1) || \right) \\
& \le C r^2 \left( ||\mathcal{N}(0) || + ||R(\hat{C}_0 + \hat{C}_1) || \right), \; \text{since $Q_-$ is a projection operator}.
\end{aligned}
\end{equation*}
We also have (see Nicolaescu, Lemma 4.5.5)
\[
||\mathcal{N}(0)|| \le Ce^{-\mu r}
\]
and from Lemma~\ref{lem Lemma 4.5.6 of Nicolaescu's book}
\[
|| R(\hat{C}_0 + \hat{C}_1) || \le C r^{3/2} (r^{-4})^2 = C r^{-13/2}.
\]
Hence
\[
|| \mathcal{F}(\hat{C}_1) || \le C( r^2 e^{-\mu r} + r^{-13/2} ).
\]
The right hand side is less than or equal to $r^{-4}$ for all sufficiently large $r$, hence $\mathcal{F}$ sends $B_1(r^{-4})$ to itself.

Now we show $\mathcal{F}$ is a contraction.
Using Lemma~\ref{lem Lemma 4.5.6 of Nicolaescu's book} again, we have
\begin{equation*}
\begin{aligned}
|| \mathcal{F}(\hat{C}_1) - \mathcal{F}(\hat{C}_1') || &\le ||SQ_-( R(\hat{C}_0 + \hat{C}_1) - R(\hat{C}_0 + \hat{C}_1') || \\
&\le Cr^2 || R(\hat{C}_0 + \hat{C}_1) - R(\hat{C}_0 + \hat{C}_1') || \\
&\le Cr^2 r^{3/2} r^{-4} || \hat{C}_1 - \hat{C}_1' || \\
&\le C r^{-1/2} ||\hat{C}_1 - \hat{C}_1' ||.
\end{aligned}
\end{equation*}
This shows that for all sufficiently large $r$, $\mathcal{F}$ is indeed a contraction.
\end{proof}

Thus for each $\hat{C}_0 \in B_0(r^{-4})$, there is a uniquely determined fixed point of $\mathcal{F}$, which will be denoted as $\hat{C}_1 = \Phi(\hat{C}_0) \in B_1(r^{-4})$.
It can be shown that $\Phi(\hat{C}_0)$ depends differentiably on $\hat{C}_0$ (by the implicit function theorem).

Now consider again the pair of equations
\[
P_- \mathcal{N}(\hat{C}_0 + \hat{C}_1) = 0, \quad Q_- \mathcal{N}(\hat{C}_0 + \hat{C}_1) = 0.
\]
The second of these equations is solved by the substitution $\hat{C}_1 = \Phi(\hat{C}_0)$ and we are left with just the single equation for $\hat{C}_0$:
\[
P_- \mathcal{N}( \hat{C}_0 + \Phi(\hat{C}_0) ) = 0.
\]
This defines a map
\[
\kappa_r : B_0(r^{-4}) \to H_r^-, \quad \hat{C}_0 \mapsto P_- \mathcal{N}( \hat{C}_0 + \Phi(\hat{C}_0) )
\]
which we call the {\em Kuranishi map}. We call $H_r^-$ the {\em obstruction space}.

We have just proven the following: the set of all monopoles on $\hat{X}(r)$ located in the slice at $\hat{C}_r$ at an $L^{2,2}$-distance of $\le r^{-4}$ from $\hat{C}_r$ is given by the graph of the function $\Phi$ over $\kappa_r^{-1}(0)$.

If $\kappa_r = 0$ (for instance, if the obstruction space is zero), then the moduli space of monopoles on $\hat{X}(r)$ is locally given as the graph of a function $\Phi$ defined in a neighbourhood of $0 \in H_r^+$.
In this case we can identify $H_r^+$ with the tangent space of the moduli space at the point $\hat{C}_r + \Phi(0)$.

In the unobstructed case $H_r^- = 0$, we can consider the map $\Phi$ as $\hat{C}_r$ varies within the space of configurations obtained by gluing.
The map $\hat{C}_r \mapsto \hat{C}_r + \Phi(0)$ sends a naively glued configuration by cutoff fuction to a genuine solution of the Seiberg-Witten equations and for all sufficiently large $r$ this mapping will actually yield a diffeomorphism between glued configurations and the Seiberg-Witten moduli space on $\hat{X}(r)$.
The fact that this map is a local diffeomorphism is not so hard to show from the Kuranishi model.
The fact that it is 1-1 and onto requires further analysis, carried out in Nicolaescu \textsection 4.5.4.

Using the grid of $3 \times 3$-asymptotic exact sequences (Nicolaescu, page 440--441) one finds the following:

\begin{proposition}
\label{prop:hplusminus}
Suppose that $\hat{C}_{\hat{M}}$ is irreducible and $\hat{C}_{\hat{N}}$ is reducible.
Then, for any sufficiently large $r$, there are isomorphisms
\[
H_r^+ \cong H^1_{\hat{C}_{\hat{M}}} \oplus H^1_{\hat{C}_{\hat{N}}}, \quad H_r^- \cong H^2_{\hat{C}_{\hat{M}}} \oplus H^2_{\hat{C}_{\hat{N}}}.
\]
\end{proposition}

\begin{proof}
Let us use the notations appearing in page 440--441 of Nicolaescu's book.
From our assumption on irreducibility and reducibility, one can check that
\begin{align}
\label{eq cal of cal C}
\mathfrak{C}_{1}^{+} \cong \mathbb{R},\ \mathfrak{C}_{1}^{-} = 0,\ \mathfrak{C}_{2}^{+} = 0,\ \mathfrak{C}_{2}^{-} \cong \mathbb{R},
\end{align}
and
\[
T\hat{G}_{1} = 0,\ T\hat{G}_{2} \cong \mathbb{R},\ T\hat{G}_{\infty} \cong \mathbb{R},
\]
and thus we also have
\begin{align}
\label{eq cal of Ls}
L_{1}^{+} = L_{1}^{-} = L_{2}^{+} = L_{2}^{-} = 0.
\end{align}
Using the calculations \eqref{eq cal of cal C}, \eqref{eq cal of Ls}, the two $3 \times 3$-asymptotic exact sequences in Nicolaescu page 440--441, and Proposition~4.1.20 in Nicolaescu, we can deduce the existence of isomorphisms stated in the proposition.
\end{proof}

\begin{example}
Suppose $N = S^4$ is the $4$-sphere. Then $X = M \# N \cong M$.
Suppose also that $b^+(M) > 0$.
Now take a generic perturbation $\eta_{\hat{M}}$ which avoids reducibles and for which the moduli space $\mathcal{M}( \hat{M} , g_{\hat{M}} , \eta_{\hat{M}} , \mathfrak{s}_{\hat{M}} , \mu )$ is smooth and take $\eta_{\hat{N}} = 0$. Note that $\hat{N}$ can be given a cylindrical end metric which has everywhere positive scalar curvature.
Since $\hat{N}$ is simply-connected, the Seiberg-Witten equations on $\hat{N}$ admit a unique solution $\hat{C}_{\hat{N}}$ which is reducible and $H^i_{\hat{C}_{\hat{N}}} = 0$ for $i=1,2$. Now let $\hat{C}_{\hat{M}} \in \mathcal{M}( \hat{M} , g_{\hat{M}} , \eta_{\hat{M}} , \mathfrak{s}_{\hat{M}} , \mu )$, so we can assume $\hat{C}_{\hat{M}}$ is irreducible and $H^2_{\hat{C}_{\hat{M}}} = 0$.
Then by Proposition~\ref{prop:hplusminus}, we get
\[
H_r^+ \cong H^1_{\hat{C}_{\hat{M}}}, \quad H_r^- \cong 0.
\]
Using this, one can show that the gluing map gives a diffeomorphism (for all sufficiently large $r$)
\[
\mathcal{M}( \hat{M} , g_{\hat{M}} , \eta_{\hat{M}} , \mathfrak{s}_{\hat{M}} , \mu ) \to \mathcal{M}( \hat{X}(r) , g_{\hat{X}(r)} , \eta_{\hat{X}(r)} , \mathfrak{s}_{\hat{X}(r)}  )
\]
But now we observe that $\hat{X}(r) = M$. Therefore, if $b^+(M) > 0$ we deduce that the Seiberg-Witten moduli spaces on $M$ and $\hat{M}$ are diffeomorphic, for suitably chosen metrics and perturbations.
\end{example}

Unfortunately in the families setting, we can not expect the spaces $H_r^-$ (defined fibrewise) to vanish.
Because of this, adapting the gluing construction to the families setting requires some additional work as we will explain in the following sections.

\section{Gluing in families}
\label{sec:gluefamily}

Now we turn to gluing in the families setting. Suppose that we have our families $E_M \to B$, $E_N \to B$ satisfying all the assumptions stated in Section~\ref{sec:setup}.
Now we carry out the gluing construction in families. First of all, we note that it is possible to define families moduli spaces for the cylindrical end families $E_{\hat{M}}, E_{\hat{N}}$.
Let us denote these moduli spaces as $\mathcal{M}( E_{\hat{M}} , g_{\hat{M}} , \eta_{\hat{M}} , \mathfrak{s}_{M} , \mu ), \mathcal{M}( E_{\hat{N}} , g_{\hat{N}} , \eta_{\hat{N}} , \mathfrak{s}_{N} , \mu )$.
Some comments are in order:
\begin{itemize}
\item{As usual we make no distinction between Spin$^c$-structures on $M$ and $\hat{M}$ and simply write either one as $\mathfrak{s}_M$. Similarly for $\mathfrak{s}_N$.}
\item{The weight $\mu$ is chosen to be positive, but sufficiently small so that Lockhart-McOwen theory, summarized in Subsection~4.1.4 of Nicolaescu's book, can be applied to the deformation complexes of all configurations in the families moduli spaces.
More precisely, we take $\mu$ so that Lockhart-McOwen theory works for all operators appearing in the deformation complexes of regularity $L^{k,2}_{\mu} \to L^{k-1,2}_{\mu}$ for $k=1,2,3$.
Since the base $B$ is compact, it is not hard to see that such a $\mu$ actually exists.}
\item{$g_{\hat{M}}, g_{\hat{N}}$ are families of cylindrical end metrics, which may be constructed as in Section \ref{sec:setup}.}
\item{$\eta_{\hat{M}}, \eta_{\hat{N}}$ are families of $2$-form perturbations which vanish on the necks.
These glue together to form for a family $\eta_{\hat{X}(r)} = \eta_{\hat{M}} \#_r \eta_{\hat{N}}$ of $2$-form perturbations for the family $E_{\hat{X}(r)}$, for each $r > 0$.
}
\item{To keep notation simple, we will describe families gluing only in the case that $\rho$ is trivial. The more general case is a straightforward extension.}
\end{itemize}

Recall the isomorphism 
\[
H^2_{L^2}(\hat{N}) \cong H^2(N) \cong H^2(\hat{N} , [0 , \infty) \times S^3).
\]
The first Chern class $c_1(\mathfrak{s}_N) \in H^2(N)$ of the determinant line on $N$ can then be identified with a class in $H^2_{L^2}(\hat{N})$.
Analytically, one can find a connection $A$ on the determinant line of class $L^{2,2}_\mu$ such that $F_A$ is a harmonic $2$-form.
This follows by noting that the curvature of any connection is a closed $2$-form and considering the Hodge decomposition of the curvature of a given connection.
Then $\frac{i}{2\pi} F_A$ represents the Chern class in $L^2$-cohomology. Let $\mathcal{W}_{\hat{N}} \in H^+_{L^2}(\hat{N})$ be $2\pi$ times the projection of this class to its self-dual part, thus as harmonic forms
\[
\mathcal{W}_{\hat{N}} = iF^+_A.
\]
Next, let us observe that the map
\[
d^+ : L^{k,2}_\mu( \hat{N} , T^*\hat{N}) \to L^{k-1,2}_\mu(\hat{N} , \wedge^2_+ T^*\hat{N})
\]
has closed image for $1 \leq k \leq 2$ by Lockhart-McOwen theory.
(Although this holds for any $k \geq 1$, we need only the cases of lower regularities.)
The cokernel of this map is therefore isomorphic to $H^+_{L^2}(\hat{N})$ as well as the case of closed $4$-manifolds.
Namely, we have the $L^{2}_{\mu}$-orthogonal decomposition
\begin{align}
L^{k-1,2}_\mu(\hat{N} , \wedge^2_+ T^*\hat{N})
= H^{+}_{L^{2}}(\hat{N}) \oplus d^{+}(L^{k,2}_\mu( \hat{N} , T^*\hat{N})).
\label{eq decomposition of lambda^plus}
\end{align}

Let $\eta_{\hat{N}}$ be a self-dual $2$-form perturbation of class $L^{1,2}_\mu$ which vanishes on the neck.
Then there exists a reducible solution to the $\eta_{\hat{N}}$-perturbed Seiberg-Witten equations on $\hat{N}$ of class $L^{2,2}_\mu$ if and only if 
\begin{align}
[ \eta_{\hat{N}} ] = \mathcal{W}_{\hat{N}},
\label{eq iff exists reducible}
\end{align}
where $[ \eta_{\hat{N}} ]$ means the class of the projection of $\eta_{\hat{N}} \in L^{1,2}_\mu(\hat{N} , \wedge^2_+ T^*\hat{N})$ into $H^+_{L^2}(\hat{N})$ along the decomposition
\eqref{eq decomposition of lambda^plus}.
This is because if the equality \eqref{eq iff exists reducible} holds for $\eta_{\hat{N}}$, then $\eta_{\hat{N}} = iF_{A}^{+} + d^{+}a = iF_{A+a}^{+}$ for some imaginary valued $1$-form $a$, and vice versa.

We call $\mathcal{W}_{\hat{N}}$ the {\em wall} and we say $\eta_{\hat{N}}$ {\em lies on the wall} if $[ \eta_{\hat{N}} ] = \mathcal{W}_{\hat{N}}$.
Then from what we have just discussed, the moduli space of $\eta_{\hat{N}}$-perturbed Seiberg-Witten equations on $\hat{N}$ admits reducible solutions if and only if $\eta_{\hat{N}}$ lies on the wall.
Moreover, if $b_1(N) = 0$, then up to gauge equivalence there is a unique such reducible solution whenever $\eta_{\hat{N}}$ lies on the wall.

Now consider the connected sum $\hat{X}(r) \cong M \# N$. From the Mayer-Vietors sequence, we get an isomorphism
\[
H^2(\hat{X}(r)) \cong H^2(\hat{X}(r) , \mathbb{R} ) \cong H^2( M , \mathbb{R}) \oplus H^2( N , \mathbb{R}) \cong H^2_{L^2}(\hat{M}) \oplus H^2_{L^2}(\hat{N}).
\]
This isomorphism is obtained by composing a series of well-defined isomorphisms depending only on $r$, and the cylindrical end metrics $g_{\hat{M}}, g_{\hat{N}}$ on $\hat{M}$ and $\hat{N}$.
Let
\[
\varphi_r : H^2(\hat{X}(r) ) \to H^2_{L^2}(\hat{M}) \oplus H^2_{L^2}(\hat{N})
\]
denote this isomorphism.

Inside $H^2_{L^2}(\hat{M}), H^2_{L^2}(\hat{N})$ and $H^2(\hat{X}(r))$ we have the subspaces $H^+_{L^2}(\hat{M}), H^+_{L^2}(\hat{N})$ and $H^+(\hat{X}(r))$ of harmonic self-dual $2$-forms.
Unfortunately, we can not expect $H^+_{L^2}(\hat{M}) \oplus H^+_{L^2}(\hat{N})$ to coincide with $\varphi_r(H^+(\hat{X}(r)))$.
The problem is that we can not directly apply the Mayer-Vietoris sequence to harmonic forms.
However, the linear gluing theory, such as Cappell-Lee-Miller theory referred in Nicolaescu (see also Section~3 of Donaldson's book~\cite{don}), implies that $\varphi_r(H^+(\hat{X}(r)))$ and $H^+_{L^2}(\hat{M}) \oplus H^+_{L^2}(\hat{N})$ are {\em asymptotically isomorphic} in the terminology of Nicolaescu (see Nicolaescu \textsection 4.1.5 for this type of linear gluing theory and page 304 for the definition of an asymptotic isomorphism).
Let us set
\begin{align*}
&\mathcal{V} = H^2_{L^2}(\hat{M}) \oplus H^2_{L^2}(\hat{N}),\\
&\mathcal{V}_1 = H^+_{L^2}(\hat{M}) \oplus H^+_{L^2}(\hat{N}), \text{ and}\\
&\mathcal{V}_2(r) = \varphi_r( H^+(\hat{X}(r)) ).
\end{align*}
Note that $\mathcal{V}_2(r)$ depends on $r$ while $\mathcal{V}_1$ does not.
Equip $\mathcal{V}$ with the natural $L^2$-inner product and let $\pi_1: \mathcal{V} \to \mathcal{V}_{1}$, $\pi_2(r) : \mathcal{V} \to \mathcal{V}_{2}(r)$ be the orthogonal projections to $\mathcal{V}_1, \mathcal{V}_2(r)$.
The fact that $\mathcal{V}_1$, $\mathcal{V}_2(r)$ are asymptotically isomorphic implies that $\pi_2(r) \to \pi_1$ as $r \to \infty$ in the operator norm.

In the families setting the various cohomology groups that we are considering define vector bundles over $B$.
In particular, we think of $H^2_{L^2}(\hat{M})$, $H^2_{L^2}(\hat{N})$ and $H^2(\hat{X}(r) )$ as being vector bundles over $B$ and $H^+_{L^2}(\hat{M})$, $H^+_{L^2}(\hat{N})$ and $H^+(\hat{X}(r) )$ as being sub-bundles. Furthermore the walls $\mathcal{W}_{\hat{M}}, \mathcal{W}_{\hat{N}}, \mathcal{W}_{\hat{X}(r)}$ define sections of $H^+_{L^2}(\hat{M})$, $H^+_{L^2}(\hat{N})$ and $H^+(\hat{X}(r) )$ respectively. Similarly, the cohomology classes of the perturbations $\eta_{\hat{M}}, \eta_{\hat{N}}, \eta_{\hat{X}(r)}$ also define sections of these vector bundles.

Here we note that, if at least one of given perturbations on $\hat{M}$ and $\hat{N}$ does not lies on the wall, the glued perturbation also does not lie on the wall after stretching the neck sufficiently:

\begin{lemma}
Let $\eta_{\hat{M}}, \eta_{\hat{N}}$ be perturbations on $\hat{M}, \hat{N}$ supported away from the necks and set $\eta_{\hat{X}(r)} = \eta_{\hat{M}} \#_r \eta_{\hat{N}}$.
Suppose that $([\eta_{\hat{M}}] , [\eta_{\hat{N}}]) \neq (\mathcal{W}_{\hat{M}} , \mathcal{W}_{\hat{N}} )$, that is, suppose that at least one of $\eta_{\hat{M}}$ and $\eta_{\hat{N}}$ does not lie on the wall. 
Then there exists $r_0$ such that $\eta_{\hat{X}(r)}$ does not lie on the wall $\mathcal{W}_{\hat{X}(r)}$ for all $r \ge r_0$.
The result also holds in the families setting over a compact base $B$.
\end{lemma}

\begin{proof}
Let $\tilde{\varphi} : H^2(\hat{X}(r) , \mathbb{R} ) \to H^2_{L^2}(\hat{M}) \oplus H^2_{L^2}(\hat{N})$ be the isomorphism
 defined as the composition of the maps appearing in the definition of $\varphi$, namely
\[
\tilde{\varphi} : H^2(\hat{X}(r) , \mathbb{R} ) \cong H^2( M , \mathbb{R}) \oplus H^2( N , \mathbb{R}) \cong H^2_{L^2}(\hat{M}) \oplus H^2_{L^2}(\hat{N}).
\]
Since $\hat{X}(r)$ is independent of $r$ as a topological space, the map $\tilde{\varphi}$ is also independent of $r$.
At the level of cohomology classes supported away from the necks, we have that $\varphi_r$ is just the Mayer-Vietoris isomorphism.
In particular, we have
\[
\varphi_r [\eta_{\hat{X}(r)}] = ([\eta_{\hat{M}}] , [\eta_{\hat{N}}]), \text{ and } \tilde{\varphi}( c_1(\mathfrak{s}_X) ) = ( c_1(\mathfrak{s}_M ), c_1(\mathfrak{s}_N) ).
\]
Applying $(2\pi) \cdot \pi_1$ to the second of these, we get
\[
( \mathcal{W}_{\hat{M}} , \mathcal{W}_{\hat{N}} ) = (2\pi) \cdot \pi_1( \tilde{\varphi}( c_1(\mathfrak{s}_X) ) = \varphi( \mathcal{W}_{\hat{X}(r)} ) + u_r,
\]
where
\[
u_r = (2\pi) \cdot (\pi_1 - \pi_2(r))( \tilde{\varphi}( c_1(\mathfrak{s}_X) ).
\]
Now suppose that $([\eta_{\hat{M}}] , [\eta_{\hat{N}}]) \neq (\mathcal{W}_{\hat{M}} , \mathcal{W}_{\hat{N}} )$. Let $\delta$ be the distance between these two points.
Or in the families setting, let $\delta$ be the infimum of
\[
||([\eta_{\hat{M}}] , [\eta_{\hat{N}}]) - (\mathcal{W}_{\hat{M}} , \mathcal{W}_{\hat{N}} )||
\]
over $b \in B$.
Since $B$ is compact, we have $\delta > 0$.
Then
\begin{equation*}
\begin{aligned}
|| \varphi_r( [\eta_{\hat{X}(r)}] ) - \varphi_r(\mathcal{W}_{\hat{X}(r)}) || &= || ([\eta_{\hat{M}}] , [\eta_{\hat{N}}]) - (\mathcal{W}_{\hat{M}} , \mathcal{W}_{\hat{N}}) + u_r  || \\
& \ge || ([\eta_{\hat{M}}] , [\eta_{\hat{N}}]) - (\mathcal{W}_{\hat{M}} , \mathcal{W}_{\hat{N}}) || - ||u_r|| \\
& \ge \delta - ||u_r||.
\end{aligned}
\end{equation*}
The result follows by observing that $u_r \to 0$ (uniformly with respect to $B$), since $\pi_2(r) \to \pi_1$ (uniformly with respect to $B$).
So there exists an $r_0$ such that $||u_r|| < \delta/2$ for all $r \ge r_0$ and all $b \in B$.
\end{proof}

Now we are almost ready to consider families gluing.
To make the analysis much simpler we will show that the metrics and perturbations can be chosen sufficiently nicely.
The setting of functional spaces will be given just after the statement of Proposition~\ref{prop:perturb}.
Recall by Section \ref{sec:step1}, that it was sufficient to assume $dim(B) = d = b^+(N)$.

\begin{proposition}\label{prop:perturb}
Let $\delta > 0$ be given. It is possible to choose the families metrics $\hat{g}_{\hat{M}}, \hat{g}_{\hat{N}}$, families perturbations $\eta_{\hat{M}}, \eta_{\hat{N}}$ and positive number $R_{\rm vanish}>0$ satisfying the following conditions:
\begin{itemize}
\item[(A1)]{$[\eta_{\hat{N}}]$ meets the wall $\mathcal{W}_{\hat{N}}$ transversely in a finite number of points $R = \{b_1, \dots , b_l\} \subseteq B$.}
\item[(A2)]{For each $b_i \in R$, we have that $\eta_{\hat{M}}(b_i)$ does not lie on the wall $\mathcal{W}_{\hat{M}}$ and is moreover a generic perturbation in the unparametrised sense (so that the fibre over $b_i$ of the families moduli space for the family $E_{\hat{M}}$ is smooth)}.
\item[(A3)]{There are disjoint open subsets $U_1 , \dots , U_l$ of $B$ satisfying that $b_i \in U_i$, each $U_{i}$ is diffeomorphic to the standard disk $D^{d}$, and the families $E_{\hat{M}}, E_{\hat{N}}$ admit trivialisations over each $U_i$.}
\item[(A4)]{With respect to the given trivialisations $E_{\hat{M}}|_{U_i} \cong \hat{M} \times U_i$, $E_{\hat{N}}|_{U_i} \cong \hat{N} \times U_i$, we have that $g_{\hat{M}}, g_{\hat{N}}, \eta_{\hat{M}}$ are constant on each $U_i$.}
\item[(A5)]{Let us set 
\[
\eta_i = \eta_{\hat{N}}(b_i) \in L^{1,2}_\mu(\hat{N} , \wedge^2_+ T^*\hat{N})
\]
and define $L_i(t) \in L^{1,2}_\mu(\hat{N} , \wedge^2_+ T^*\hat{N})$ for $t \in U_i$ by
\[
\eta_{\hat{N}}(t) = \eta_i + L_i(t).
\]
Then $L_{i}(t)$ satisfies the following conditions with respect to the given trivialisation $E_{\hat{N}}|_{U_i} \cong \hat{N} \times U_i$ and an identification of $U_i$ with an open neighbourhood of the origin in $T_{b_i} B$ (so we view $U_i$ as a subset of $T_{b_i} B$) and a fixed metric on $T_{b_{i}}B$:
\begin{description}
\item[(a)] The linear map $T_{b_i} B \to H^+_{L^2}( \hat{N} )$ defined as the projection of the differential $L'_i(0)$ to $H^+_{L^2}(\hat{N})$ is an isomorphism.
\item[(b)] There exists some $\delta' \in(0, \delta)$
such that
\[
|| (d^+)^* (\eta_{\hat{N}}(t) - \eta_{\hat{N}}(t')) ||_{L^{2}_{\mu}(\hat{N})} \le \delta' ||t -t' ||
\]
holds for all $t,t'\in U_i$.
(In particular, by setting $t'=b_{i}$, we have
\begin{align}
\label{eq: almost harmonic}
|| (d^+)^* L_{i}(t)||_{L^{2}_{\mu}(\hat{N})} \le \delta' ||t||
\end{align}
for any $t \in U_{i}$.)
\item[(c)] There exists some $\epsilon' \in (0,\epsilon)$ such that
\[
\| \pi^{\perp} L_{i}(t) \|_{L^{2}_{\mu}(\hat{N})} \leq \epsilon' \| \pi L_{i}(t)\|_{L^{2}_{\mu}(\hat{N})}
\]
holds for any $t \in U_{i}$.
Here $\epsilon$ is a constant introduced in Lemma~\ref{lem: pre for testimate}, and
\begin{align*}
&\pi : L^{2}_\mu(\hat{N} , \wedge^2_+ T^*\hat{N}) \to H^{+}_{L^{2}}(\hat{N}),\\
&\pi^{\perp} : L^{2}_\mu(\hat{N} , \wedge^2_+ T^*\hat{N}) \to H^{+}_{L^{2}}(\hat{N})^{\perp}
\end{align*}
are the $L^{2}_{\mu}$-orthogonal projections, where $H^{+}_{L^{2}}(\hat{N})^{\perp}$ is the $L^{2}_{\mu}$-orthogonal complement of $H^{+}_{L^{2}}(\hat{N})$ in $L^{2}_\mu(\hat{N} , \wedge^2_+ T^*\hat{N})$.
\item[(d)] There exists $\epsilon' \in (0,1)$ such that  
\[
\| P_{-}^{r} \circ \Psi_{r} \circ \pi^{\perp} L_{i}(t)\|_{L^{2}(\hat{X}(r))}
\leq \epsilon' \| P_{-}^{r} \circ \Psi_{r} \circ \pi L_{i}(t)\|_{L^{2}(\hat{X}(r))}
\]
holds for any $t \in U_{i}$ and any $r>1$.
Here $P_{-}^{r}$ is defined in Definition~\ref{def of P and Q} and $\Psi_{r}$ is the map defined in \eqref{eq; def of gluing map psi}.
\end{description}}
\item[(A6)]{For each $b_i \in R$, let $(A , 0)$ be the unique (under the gauge fixing) reducible solution of the Seiberg-Witten equations perturbed by $\eta_{\hat{N}}(b_i) = \eta_i$. Then $Coker( D_{A} ) = 0$ holds, where $D_A : L^{2,2}_\mu( \hat{N} , S^+ ) \to L^{1,2}_\mu(\hat{N} , S^-)$ is the associated Dirac operator.}
\item[(A7)]{The perturbation $\eta_{\hat{N}}$ is generic in the families sense: for every irreducible solution of the perturbed Seiberg-Witten moduli space we have that its families deformation theory is unobstructed (i.e. $H^2$ of the families deformation complex vanishes).}
\item[(A8)]{For any $b \in B$, $\eta_{\hat{N}}(b)$ is supported outside the subspace $[R_{\rm vanish},\infty) \times S^{3}$ of the neck of $\hat{N}$.}
\end{itemize}
\end{proposition}

The main part of the proof of Proposition~\ref{prop:perturb} is to ensure that we can take $\eta_{\hat{N}}$ satisfying the desired conditions.
To do this, we need to show some lemmas.
To state the lemmas, let us fix some notation and exhibit the functional space where our desired perturbation $\eta_{\hat{N}}$ lives.
The open sets $U_{1}, \ldots, U_{l}$ and families metrics $g_{\hat{M}}, g_{\hat{N}}$ appearing in the statement of Proposition~\ref{prop:perturb} are given here.
Fix a generic section of $H^{+}_{L^{2}}(\hat{N}) \to B$.
Then the zero set of this section consists of finite points $x_{1}, \ldots, x_{l}$ for some $l\geq0$.
For each of these points $x_{i}$ $(1 \leq i \leq l)$, choose a small open neighborhood $U_{i}$ in $B$ such that $U_{i}$ is diffeomorphic to the standard disk $D^{d}$, $E_M, E_N$ admit local trivialisations over $U_{i}$, and $U_{i} \cap U_{j} = \emptyset$ if $i \neq j$.
Let $U \subset B$ be the union of these disk-like neighborhoods $U_{i}$.
We can assume the sections $\iota_M,\iota_N$ are constant on $U$ with respect to these trivialisations. Then we obtain also trivialisations of $E_{\hat{M}}, E_{\hat{N}}$. We also assume the families of metrics $g_M, g_N$ are chosen so as to be constant over $U$. Then similarly we can carry out the construction of $E_{\hat{M}}, E_{\hat{N}}$ in such a way that $g_{\hat{M}}, g_{\hat{N}}$ are also constant on $U$.
Let $\mathcal{E} \to B$ be the Hilbert bundle whose fiber on $b \in B$ is given by
\[
L^{1,2}_{\mu}(\hat{N},i\wedge^{2}_{+_{g_{\hat{N}}(b)}}T^{*}\hat{N}).
\]
Let $\mathcal{C}^{1}(B,\mathcal{E})$ denotes the space of $\mathcal{C}^{1}$-class sections of $\mathcal{E}$ on $B$, where our desired perturbation $\eta_{\hat{N}}$ lives.
Note that $\mathcal{E}$ has a natural trivialisation on $U$ because of the choice of $g_{\hat{N}}$ above, and hence the restricted sections belonging to $\mathcal{C}^{1}(B,\mathcal{E})$ on each $U_{i}$ modeled by the space of functions
\[
\mathcal{C}^{1}(U_{i}, L^{1,2}_{\mu}(\hat{N},i\wedge^{2}_{+_{g_{\hat{N}}(x_{i})}}T^{*}\hat{N})).
\]

\begin{lemma}
\label{lem: preparation1 for pertub}

Let $\mathcal{O} \subset \mathcal{C}^{1}(B, \mathcal{E})$ be the set of $\eta_{\hat{N}} \in \mathcal{C}^{1}(B, \mathcal{E})$ satisfying the following conditions (i)-(iv):
\begin{itemize}
\item[(i)]{$\eta_{\hat{N}}$ meets the wall transversally.}
\item[(ii)]{$\eta_{\hat{N}}$ meets the wall once over $U_{i}$ for each $i$, and does not meet the wall over $B\setminus U$.}
\item[(iii)]{For any point $b_{i} \in B$ where $\eta_{\hat{N}}$ lies on the wall, condition (A5) holds.}
\item[(iv)]{For any point $b_{i} \in B$ where $\eta_{\hat{N}}$ lies on the wall, condition (A6) holds.}
\end{itemize}
Then $\mathcal{O}$ is an open set in $\mathcal{C}^{1}(B, \mathcal{E})$.
\end{lemma}

\begin{proof}
We shall see that each of conditions (i)-(iv) on $\eta_{\hat{N}} \in \mathcal{C}^{1}(B, \mathcal{E})$ is an open condition with respect to $\mathcal{C}^{1}$-norm.
Conditions (i) is evidently an open condition.
Since $U$ is an open set in $B$, conditions (ii) is also an open condition.
The fact that condition~(iv) is also an open condition follows from the fact that the map $\tau \mapsto dim(Coker (D_{\tau}))$ is upper semi continuous for a continuous family of Fredholm operatos $\{D_{\tau}\}_{\tau}$.
So let us check the rest part: (iii) is an open condition.
First note that (A5)-(a), (c), and (d) are obviously open conditions, so we have to care only about (A5)-(b).
Suppose that $\eta_{\hat{N}}$ satisfies (A5)-(b) and $\xi \in \mathcal{C}^{1}(B, \mathcal{E})$ satisfies
\begin{align}
\|\xi\|_{\mathcal{C}^{1}} \leq \frac{\delta - \delta'}{2(1+\|(d^{+})^{\ast}\|_{\rm op})},
\label{eq: open (A5) step0}
\end{align}
where $\|\cdot\|_{\rm op}$ denotes the operator norm.
Then, for any $t, t' \in U_{i}$, we have
\begin{align}
&|| (d^+)^* ((\eta_{\hat{N}} + \xi)(t) - (\eta_{\hat{N}} + \xi)(t')) ||_{L^{2}_{\mu}}\nonumber\\
\leq& \delta'\|t-t'\| + \|(d^{+})^{\ast}\|_{\rm op} \| \xi(t) -\xi(t')\|_{L^{1,2}_{\mu}}.
\label{eq: open (A5) step1}
\end{align}
Under the identification of $U_{i}$ with $D^{d}$, we can find the line connecting $t$ with $t'$, and therefore the mean value theorem implies that there exists $\bar{t} \in U_{i}$ satisfying that
\begin{align}
\| \xi(t) -\xi(t')\|_{L^{1,2}_{\mu}} \leq \| \nabla\xi(\bar{t}) \|_{L^{1,2}_{\mu}} \|t-t'\| \leq \| \xi \|_{\mathcal{C}^{1}} \|t-t'\|.
\label{eq: open (A5) step2}
\end{align}
The inequalities \eqref{eq: open (A5) step0}, \eqref{eq: open (A5) step1}, and \eqref{eq: open (A5) step2} imply that
\[
|| (d^+)^* ((\eta_{\hat{N}} + \xi)(t) - (\eta_{\hat{N}} + \xi)(t')) ||_{L^{2}_{\mu}}
\leq \frac{\delta + \delta'}{2}\|t-t'\|.
\]
This proves that $\eta_{\hat{N}}+\xi$ also meets (A5)-(b), and hence (A5)-(b) is an open condition.
Therefore we can conclude that (A5) is also an open condition.
\end{proof}

\begin{lemma}
\label{lem: preparation2.5 for pertub}
For each $i \in \{1,\ldots,l\}$, there exists $\eta_{i} \in L^{1,2}_{\mu}(\hat{N},i\wedge^{2}_{+_{g_{\hat{N}}(x_{i})}}T^{*}\hat{N})$ such that the (unique) reducible solution $(A,0)$ to the $\eta_{i}$-perturbed Seiberg-Witten equations satisfies that $Coker(D_{A})=0$.
\end{lemma}

\begin{proof}
For a closed $4$-manifold, the analogous statements can be found in, for example, Lemma~2.5 in \cite{bar}.
For our case, namely, a cylindrical $4$-manifold, this follows from Proposition~9.2 in Strle~\cite{str}.
(Although Strle stated the proposition for a $4$-manifold with $b^{+}=1$ and with some special cylinder, the proof of the proposition works without any change in more general setting, including our case.
We also note that since we assumed that the index of the Dirac operator on $N$ is zero, $Coker(D_{A})=0$ is equivalent to $Ker(D_{A})=0$.)
\end{proof}

\begin{lemma}
\label{lem: preparation2 for pertub}
The set $\mathcal{O}$ contains an element $\eta_{\hat{N}}$ satisfying that there exists $R_{\rm vanish}>0$ such that $\eta_{\hat{N}}(b)$ is supported outside the subspace $[R_{\rm vanish},\infty) \times S^{3}$ of the neck of $\hat{N}$ for any $b \in B$.
\end{lemma}

\begin{proof}
Let us fix $i \in \{1, \ldots, l\}$ and identify $U_{i}$ with a neighborhood of the origin in $T_{x_{i}}B$.
By  Lemma~\ref{lem: preparation2.5 for pertub}, we can take $\eta_{i} \in L^{1,2}_{\mu}(\hat{N},i\wedge^{2}_{+_{g_{\hat{N}}(x_{i})}}T^{*}\hat{N})$ such that the (unique) reducible solution $(A,0)$ to the $\eta_{i}$-perturbed Seiberg-Witten equations satisfies that $Coker(D_{A})=0$.
Since the condition $Coker(D_{A})=0$ is an open condition with respect to $\eta_{i}$, by taking an approximating sequence consisting of smooth and compact supported sections, we can assume that $\eta_{i}$ is supported outside the subspace $[R_{i}, \infty) \times S^{3}$ of the neck for some $R_{i}>0$.
Set $R_{\rm vanish} := \max_{i}{R_{i}}$.
Let us take a linear isomorphism $L_{i} : T_{x_{i}}B \to H^{+}_{L^{2}}(\hat{N})$ whose operator norm satisfies that $\|L_{i}\|_{\rm op} < \delta/2$.
Here $\delta$ is the positive number fixed in the statement of Proposition~\ref{prop:perturb} and we think of $H^{+}_{L^{2}}(\hat{N})$ as a subspace of $L^{1,2}_{\mu}(\hat{N},i\wedge^{2}_{+_{g_{\hat{N}}(x_{i})}}T^{*}\hat{N})$ via \eqref{eq decomposition of lambda^plus}.
Let us define 
\[
\eta_{\hat{N},i} \in \mathcal{C}^{1}(U_{i}, L^{1,2}_{\mu}(\hat{N},i\wedge^{2}_{+_{g_{\hat{N}}(x_{i})}}T^{*}\hat{N}))
\]
by
$\eta_{\hat{N},i}(t) := \eta_{i} + L_{i}(t)$
for $t \in U_{i}$.

By extending the maps $\eta_{\hat{N},i}$ using cutoff functions supported in a small neighborhood of the closure of $U$, we define $\eta_{\hat{N}} \in \mathcal{C}^{1}(B,\mathcal{E})$.
Then, for any $b \in B$, $\eta_{\hat{N}}(b)$ is supported outside $[R_{\rm vanish},\infty) \times S^{3}$.
We shall check that $\eta_{\hat{N}}$ belongs to $\mathcal{O}$.
The reason why $\eta_{\hat{N}}$ satisfies condition (i) is that $L_{i}$ is an linear isomorphism.
It is obvious that $\eta_{\hat{N}}$ satisfies condition (ii).
Since for $t,t \in U_{i}$ we have
\begin{align*}
|| (d^+)^* (\eta_{\hat{N}}(t) - \eta_{\hat{N}}(t')) ||_{L^{2}_{\mu}}
\leq \|L_{i}\|_{\rm op} \|t-t'\| \leq \frac{\delta}{2}\|t-t'\|,
\end{align*}
$\eta_{\hat{N}}$ satisfies condition (iii).
(Note that (A5)-(c), (d) are obviously satisfied since $\pi^{\perp}L_{i}=0$.)
Because of the choice of $\eta_{i}$, it is obvious that $\eta_{\hat{N}}$ satisfies condition (iv).
\end{proof}

\begin{proof}[Proof of Proposition~\ref{prop:perturb}]
Let $\mathcal{E}'$ be the Banach subbundle of $\mathcal{E} \to B$ whose fiber on $b \in B$ is given by
\[
\{\eta \in L^{1,2}_{\mu}(\hat{N},i\wedge^{2}_{+_{g_{\hat{N}}(b)}}T^{*}\hat{N}) \mid \eta|_{[R_{\rm vanish},\infty) \times S^{3}} = 0 \}.
\]
Set $\mathcal{O}' := \mathcal{O} \cap \mathcal{C}^{1}(B,\mathcal{E'})$.
Because of Lemma~\ref{lem: preparation1 for pertub} and Lemma~\ref{lem: preparation2 for pertub}, $\mathcal{O}'$ is a non-empty open subset of $\mathcal{C}^{1}(B,\mathcal{E'})$.
Let $\mathcal{D} \subset \mathcal{C}^{1}(B,\mathcal{E})$ be the set of generic perturbations in the families sense, and define
 $\mathcal{D}' := \mathcal{D} \cap \mathcal{C}^{1}(B,\mathcal{E'})$.
Note that $\mathcal{D}'$ is a dense subset of $\mathcal{C}^{1}(B,\mathcal{E}')$.
This follows from a similar argument which ensures that, in unparameterized setting, one can take a generic perturbation supported on a given open set, found in, for example, Remark~8.18 in \cite{sal}.
So the intersection $\mathcal{O}' \cap \mathcal{D}'$ is non-empty.
Take $\eta_{\hat{N}}$ from this intersection.

On each $U_i$ choose $\eta_{\hat{M}}$ to be a constant perturbation which is generic in the unparametrised setting and belongs outside the wall of $\hat{M}$.
Then choose some extension of $\eta_{\hat{M}}$ to all of $B$ (e.g, by using cutoff functions and extending by zero).

Then $\eta_{\hat{N}}$ and $\eta_{\hat{M}}$ satisfy the all conditions (A1)-(A8).
\end{proof}

Assume that the families metrics $g_{\hat{M}}, g_{\hat{N}}$, families perturbations $\eta_{\hat{M}}, \eta_{\hat{N}}$, points $b_{1}, \ldots, b_{l} \in B$, and open sets $U_{1},\ldots, U_{l} \subset B$ have been chosen by Proposition~\ref{prop:perturb}.

\begin{remark}
\label{rem:linear}
By the implicit function theorem, if we replace $U_{i}$ by possibly smaller open neighbourhoods, we can take coordinates of $U_{i}$ so that the projection of $L_{i}(t)$ to $H^{+}_{L^{2}}(\hat{N})$ is linear in $t$ with respect to the coordinate via the fixed identification between $U_{i}$ and the origin of $T_{b_{i}}B$ given in (A5).
We shall employ this coordinate system of $U_{i}$.
Henceforth we fix all of these data $g_{\hat{M}}, g_{\hat{N}}$, $\eta_{\hat{M}}, \eta_{\hat{N}}$, $b_{1}, \ldots, b_{l}$, and $U_{1},\ldots, U_{l}$, and shall consider only sufficiently large $r$ in the sense that $r \geq R_{\rm vanish} + 1$.
\end{remark}

In the unparametrised setting, the results of Nicolaescu \textsection 4.5.3 imply there exists an $r_0$ such that if $r > r_0$ then every genuine solution of the Seiberg-Witten equations is close to a glued solution.
In particular, the moduli space of Seiberg-Witten equations on $\hat{N}$ for a given $(g_{\hat{N}} , \eta_{\hat{N}})$ is empty, then the moduli space of Seiberg-Witten equations on $\hat{X}(r)$ is empty for all $r \ge r_0$.
In the families setting, the value of $r_0 = r_0(b)$ may depend on the point $b \in B$. However, we have:

\begin{lemma}\label{lem:confine}
Suppose that $C \subseteq B$ is a closed subset of $B$ and suppose that for each $c \in C$, there are no solutions to the Seiberg-Witten equations on $\hat{N}$ for $(g_{\hat{N}}(c) , \eta_{\hat{N}}(c) )$. Then there exists an $r_0$ such that for all $r \ge r_0$, all solutions of the Seiberg-Witten equations for the family $E_{\hat{X}(r)}$ lie over $B \setminus C$.
\end{lemma}

\begin{proof}
Since $C$ is compact, one can find such $r_{0}$ by repeating Nicolaescu \textsection 4.5.3 uniformly over $C$.
\end{proof}

\begin{lemma}\label{lem:finiteset}
Every element of $\mathcal{M}( E_{\hat{N}} , g_{\hat{N}} , \eta_{\hat{N}} , \mathfrak{s}_N , \mu)$ is reducible.
\end{lemma}

\begin{proof}
Let $\hat{C}_{\hat{N}} \in \mathcal{M}( E_{\hat{N}} , g_{\hat{N}} , \eta_{\hat{N}} , \mathfrak{s}_N , \mu)$.
Let $\hat{H}^i_{\hat{C}_{\hat{N}}}$ for $i=0,1,2$ denote the cohomology groups of the families deformation complex at $\hat{C}_{\hat{N}}$. The virtual dimension of the families moduli space $\mathcal{M}( E_{\hat{N}} , g_{\hat{N}} , \eta_{\hat{N}} , \mathfrak{s}_N , \mu)$ is
\[
d( N , \mathfrak{s}_N ) + d = -1 - d + d = -1,
\]
and hence
\[
-dim(\hat{H}^0_{\hat{C}_{\hat{N}}}) +dim(\hat{H}^1_{\hat{C}_{\hat{N}}}) -dim(\hat{H}^2_{\hat{C}_{\hat{N}}}) = -1.
\]
But the genericity assumption (A7) implies that $\hat{H}^2_{\hat{C}_{\hat{N}}} = 0$. Hence the only way the above equation can be satisfied is if $dim(\hat{H}^0_{\hat{C}_{\hat{N}}}) \neq 0$, so $\hat{C}_{\hat{N}}$ is reducible.
\end{proof}

By Lemma~\ref{lem:finiteset} the moduli space $\mathcal{M}( E_{\hat{N}} , g_{\hat{N}} , \eta_{\hat{N}} , \mathfrak{s}_N , \mu)$ is just a finite set of points, which is in bijection with the finite set $\{ b_1, b_2 , \dots , b_l\}$.

\begin{corollary}\label{cor:confine}
There exists an $r_0$ such that for all $r \ge r_0$ we have that every solution of the Seiberg-Witten equations for the family $E_{\hat{X}(r)}$ lies over $U$.
\end{corollary}

\begin{proof}
Because of Lemma~\ref{lem:finiteset}, any solutions to the Seiberg-Witten equations on $\hat{N}$ appears on $b_{1}, \ldots, b_{l}$.
So this corollary immediately follows from Lemma~\ref{lem:confine} by setting $C = B \setminus U$ in Lemma~\ref{lem:confine}.
\end{proof}

Thus to describe the families moduli space $\mathcal{M}( E_{\hat{X}(r)} , g_{\hat{X}(r)} , \eta_{\hat{X}(r)} , \mathfrak{s}_X )$, we just need to study its restriction to each of the open sets $U_1 , \dots , U_l$. We pick one of the open sets $U_i$ and study the moduli space over $U_i$.
The notation $t$ indicates an element of $U_{i}$.
Let $\hat{C}_{N}$ be the unique reducible solution of the Seiberg-Witten equations on $\hat{N}$ over the point $b_i$ and let $\hat{C}_{M}$ be any solution of the Seiberg-Witten equations on $\hat{M}$ over $b_i$.
We wish to study the solutions of families Seiberg-Witten equations on $E_{\hat{X}(r)}|_{U_i}$ which are close to $\hat{C}_r = \hat{C}_{M} \#_r \hat{C}_{N}$. By the assumptions given in Proposition \ref{prop:perturb}, we may assume some simplifications, summarized as follows:
\begin{itemize}
\item{We have trivialisations $E_{\hat{M}}|_{U_i} \cong \hat{M} \times U_i$, $E_{\hat{N}}|_{U_i} \cong \hat{N} \times U_i$. Note this also induces a trivialisation $E_{\hat{X}(r)}|_{U_i} \cong \hat{X}(r) \times U_i$.}
\item{We may identify $U_i$ with an open ball at the origin in $T_{b_i}B$ such that $b_i \in U_i$ corresponds to $0 \in T_{b_i}B$.}
\item{To be specific, we can fix some Euclidean norm on $T_{b_i}B$ and take $U_i = \{ t \in T_{b_i}B \; | \; || t || < \tau \}$, for some $\tau > 0$.}
\item{The fibrewise metrics $g_{\hat{M}}, g_{\hat{N}}$ are independent of $t \in U_i$.}
\item{$\eta_{\hat{M}}$ is a constant which does not lie on the wall $\mathcal{W}_{\hat{M}}$.
Moreover $\eta_{\hat{M}}$ is generic in the unparametrised sense.
But this also implies that over each $U_i$, $\eta_{\hat{M}}$ is generic in the parametrised sense as well.
Thus we have an isomorphism
\[
\mathcal{M}( E_{\hat{M}}|_{U_i} , g_{\hat{M}} , \eta_{\hat{M}} , \mathfrak{s}_M , \mu) \cong \mathcal{M}( \hat{M} , g_{\hat{M}} , \eta_{\hat{M}} , \mathfrak{s}_M ) \times U_i,
\]
where $\mathcal{M}( \hat{M} , g_{\hat{M}} , \eta_{\hat{M}} , \mathfrak{s}_M )$ is the unparametrised moduli space on $\hat{M}$ which by our assumptions is smooth and contains no reducibles.}
\item{Over $U_i$, we have $\eta_{\hat{N}}(t) = \eta_0 + L(t)$, where $L(t) = L_i(t)$ is as described in Proposition \ref{prop:perturb} and $\eta_{\hat{N}}(t)$ is generic in the families sense.}
\end{itemize}

Now we proceed in much the same way as we did in the unparametrised case explained in Section~\ref{sec:glueloc}.
For each $t \in U_{i}$, form the non-linear map
\[
\mathcal{N}(\cdot, t) : L^{2,2}( \hat{X}(r) , S^+ \oplus iT^*\hat{X}(r) ) \to L^{1,2}(\hat{X}(r), S^- \oplus i\wedge^2_+ T^*\hat{X}(r) \oplus i\mathbb{R} )
\]
given by
\[
\mathcal{N}(\hat{C} , t) = \widehat{SW}_t( \hat{C} + \hat{C}_r) \oplus \mathcal{L}_{\hat{C}_r}^*( \hat{C} ),
\]
where $\widehat{SW}_t$ denotes the Seiberg-Witten equations taken with respect to the perturbation $\eta_{\hat{X}(r)}(t)$.
Namely, $\widehat{SW}_t$ is defied as
\[
\widehat{SW}_{t}( \hat{C} + \hat{C}_r)
= (F_{\hat{A}_{0}+\hat{a}+\hat{a}_{r}}^{+}-q(\hat{\psi}+\hat{\psi}_{r},\hat{\psi}+\hat{\psi}_{r}) -i\eta_{\hat{X}(r)}(t), D_{\hat{A}_{0}+\hat{a}+\hat{a}_{r}}(\hat{\psi}+\hat{\psi}_{r}))
\]
 where we write $\hat{C}$ and $\hat{C}_{r}$ as $\hat{C} = (\hat{a}, \hat{\psi})$ and $\hat{C}_{r} = (\hat{a}_{r}, \hat{\psi}_{r})$.
The term $\mathcal{L}_{\hat{C}_r}^*( \hat{C} )$ is the gauge fixing condition exactly as in the unparametrised case considered in Section~\ref{sec:glueloc}.
Since this term does not involve the perturbation $\eta_{\hat{X}(r)}(t)$, it is independent of $t$. As in the unparametrised case, we are interested in the zero set $\mathcal{N}^{-1}(0)$.

Write $\eta_{\hat{X}(r)}(t)\in L^{1,2}(\hat{X}(r),i\wedge^{2}_{+_{g_{\hat{X}(r)}(b_{i})}}T^{*}\hat{X}(r))$ in the form
\[
\eta_{\hat{X}(r)}(t) = \eta_{\hat{X}(r)}(0) + K_r(t).
\]
Since $\eta_{\hat{M}}$ is constant on $U_{i}$, we have $K_r(t) = 0 \#_r L(t)$.

We will usually omit the subscript and write this as $K(t)$.
Here recall the estimate \eqref{eq: almost harmonic}, which is a consequence of condition (A5): $L(t)$ is approximately harmonic in the sense that $|| (d^+)^* L(t) || < \delta ||t||$, where $\delta > 0$ is some specified real number.
Then $K_r(t)$ will also be approximately harmonic.
More precisely:

\begin{lemma}
We have the estimate
\begin{equation}\label{equ:almostharmonic}
|| (d^+)^* K(t) ||_{L^{2}(\hat{X}(r))} \le Ce^{-\mu r} + \delta ||t||,
\end{equation}
where $C$ is a constant which is independent of $r$.
(Note that $C$ may depend on choice of $\eta_{\hat{N}}$.)
\end{lemma}

\begin{proof}
Recall the gluing operation $\#_{r}$ given in Nicolescu page~305.
This gives us the following description of 
$K(t) = 0 \#_r L(t)$ on the neck of $\hat{X}(r)$.
Let $\tau \in [0,r]$ denote the time-direction coordinate on the neck $[0,r] \times S^{3}$, regarded as a subset of $\hat{X}(r)$.
Then, on the neck, we have
\begin{align}
K(t)(\tau,y) = \alpha_{r}(\tau)L(t)(\tau,y)
\label{eq: expression of K}
\end{align}
for $(\tau,y) \in [0,r] \times S^{3}$, where $\alpha_{r} : [0,\infty) \to [0,1]$ is a cutoff function such that
\begin{align*}
&\alpha_{r}(\tau) = 1 \quad \text{on} \quad [0,r-1],\quad{\rm and }\\
&\alpha_{r}(\tau) = 0 \quad \text{on} \quad [r, \infty).
\end{align*}
(Note that, although $L^{1,2}_{\mu}$ cannot be embedded in $\mathcal{C}^{0}$, by taking an approximating sequence converging to $L(t)$ and consisting of smooth sections, we can define the pointwise multiplication of $\alpha_{r}$ with $L(t)$ as the limit  of the sequence multiplied by $\alpha_{r}$.
Therefore we can define $\alpha(\tau)L(t)$ by \eqref{eq: expression of K} even if $L(t)$ is not embedded in $\mathcal{C}^{0}$.)
Therefore, by the inequality \eqref{eq: almost harmonic}, we have
\begin{align}
|| (d^+)^* K(t) ||_{L^{2}(\hat{X}(r))}
&\leq \|\nabla\alpha_{r} \cdot L(t) \|_{L^{2}(\hat{N}(r))}
+ \|(d^+)^* L(t)\|_{L^{2}(\hat{N}(r))} \nonumber\\
&\leq C \|L(t) \|_{L^{2}([r-1,r] \times S^{3})}
+ \|(d^+)^* L(t)\|_{L^{2}_{\mu}(\hat{N}(r))} \nonumber\\
&< C \|L(t) \|_{L^{2}([r-1,r] \times S^{3})}
+ \delta\|t\|.
\label{eq: estimate for K Step1}
\end{align}
Recall that we take $r$ to be $r \geq R_{\rm vanish} + 1$.
Therefore we have $L(t) = -\eta_{i}$.
By setting $t=0$ in \eqref{eq: almost harmonic}, we have $(d^{+})^{\ast}\eta_{i}=0$.
Since $\eta_{i}$ is an $L^{2}$-self-dual $2$-form, this implies that $\eta_{i}$ is harmonic.
Therefore we get an exponential decay estimate
\[
\|\eta_{i}\|_{L^{2}([r-1,r] \times S^{3})} \leq C'e^{-\mu r} \|\eta_{i}\|_{L^{2}(\hat{N})} = Ce^{-\mu r}.
\]
This inequality and \eqref{eq: estimate for K Step1} imply the desired inequality \eqref{equ:almostharmonic}.
\end{proof}

Here we give a notation $\Psi_{r}$ for the operation used to get $K(t)$ from $L(t)$, namely:
define
\begin{align}
\label{eq; def of gluing map psi}
\Psi_{r} : H^{+}_{L^{2}}(\hat{N}) \to L^{2}(\hat{X}(r))
\end{align}
by
\[
(\Psi_{r}(\eta))(\tau,y) := \alpha_{r}(\tau)\eta(\tau,y)
\]
on the neck, $\Psi_{r}(\eta) = \eta$ outside the neck on the $N$ side and is zero outside the neck on the $M$ side.

Let $\hat{T}_{r}$ denote the differential of $\mathcal{N}(\cdot,0)$ at the origin.
Namely, $\hat{T}_{r}$ is the map give as
\[
d\mathcal{N}(\cdot,0)|_{0} : L^{2,2}( \hat{X}(r) , S^+ \oplus iT^*\hat{X}(r) ) \to L^{1,2}(\hat{X}(r), S^- \oplus i\wedge^2_+ T^*\hat{X}(r) \oplus i\mathbb{R} ).
\]
We also define a map $R$ whose domain and range are the same as $\hat{T}_{r}$ by
\[
\mathcal{N}(\hat{C},0)
 =  \mathcal{N}(0,0) + \hat{T}_{r}(\hat{C}) + R(\hat{C}).
\]
Now we  define $D_r(\hat{C},t)$ by
\[
\mathcal{N}(\hat{C} , t) = \mathcal{N}(0,t) + D_r(\hat{C},t) + R(\hat{C}).
\]
Note that $D_r(\hat{C} , t)$ is written as
\begin{align*}
D_r(\hat{C} , t) = \hat{T}_r(\hat{C}) + iK_r(t).
\end{align*}
Note that, compared with the unparamtrised case,
$D_r(\hat{C} , t)$ is the only part of $\mathcal{N}(\hat{C} , t)$ that has changed.

Let us define $X^k = X^k_+ \oplus X^k_-$, $H_r = H_r^+ \oplus H_r^-$, $Y^k(r) = Y^k_+(r) \oplus Y^y_-(r)$ and the projections $P_{\pm}^{r}, Q_{\pm}^{r}$ just as in Definitions~\ref{def of Xk}, \ref{def of Hr}, \ref{def of Y}, and \ref{def of P and Q}.
We again write $P_{\pm}^{r}$ and $Q_{\pm}^{r}$ simply as $P_{\pm}$ and $Q_{\pm}$.

For each $\hat{C} \in X^0_+$, we decompose it as $\hat{C} = \hat{C}_0 + \hat{C}_1$, where
\[
\hat{C}_0 = P_+ \hat{C}, \quad \hat{C}_1 = Q_+ \hat{C}.
\]
Let $S : Y^k_- \to Y^{k+1}_+$ be defined exactly as in Section~\ref{sec:glueloc}.

\begin{lemma}\label{lem:estimate}
We have an estimate of the form
\begin{align}
|| \hat{T}_r^* iK(t) ||_{L^{2}(\hat{X}(r))} \le C e^{-\mu r} + \delta ||t ||
\label{eq: ineq T ast K}
\end{align}
and an estimate of the form
\[
|| Q_-^{r} iK(t) ||_{L^{2}(\hat{X}(r))} \le C r^2 (e^{-\mu r} + \delta ||t||).
\]
\end{lemma}

\begin{proof}
Write $\hat{C}_r = (\hat{A}_r , \hat{\psi}_r )$. We have
\[
\hat{T}_r^*( iK(t) ) = ( (d^+)^* iK(t) , \beta^*( iK(t) ) ),
\]
where $\beta : S^+ \to i \wedge^2_+ T^*\hat{X}(r)$ is given by
\begin{align}
\beta( \varphi ) = -\frac{1}{2}q(\hat{\psi}_r , \varphi)
\label{eq: expression of beta}
\end{align}
Here $q$ is the quadratic term of spinors in the Seiberg-Witten equations.
Here note that $\hat{\psi}_r$ vanishes on $N(r)$, regarded as a subspace of $\hat{X}(r)$, since $\hat{\psi}_{r}$ was made from the reducible on $\hat{N}$ just by using cutoff functions.
On the other hand, since we took $r$ to be $r \geq R_{\rm vanish}+1$, $K(t)$ vanishes outside $N(r)$.
Therefore the expression \eqref{eq: expression of beta} of $\beta$ implies that $\beta^*( iK(t) ) = 0$.
The estimate \eqref{eq: ineq T ast K} then follows from the inequality \eqref{equ:almostharmonic}.

Now decompose $iK(t)$ as $iK(t) = P_ -K(t) + Q_- iK(t)$.
Applying $\hat{T}^*_r$, we get $\hat{T}^*_r(iK(t)) = \hat{T}^*_r P_-(iK(t)) + \hat{T}^*_r Q_-(iK(t))$.
But the two terms on the right are orthogonal, so the estimate \eqref{eq: ineq T ast K} yields
\[
|| \hat{T}^*_r Q_-(iK(t)) || \le C e^{-\mu r} + \delta ||t||.
\]
By combining this estimate with Lemma~\ref{lem: estimate for S},
we can get
\[
||Q_-(iK(t))|| \le C r^2 || \hat{T}^*_r Q_-(iK(t)) || \le C r^2 (e^{-\mu r} + \delta ||t||).
\]
\end{proof}

Now we repeat the story of Section~\ref{sec:glueloc} as follows.
The equation $\mathcal{N}(\hat{C},t) = 0$ is equivalent to the pair of equations
\[
P_- \mathcal{N}(\hat{C},t) = 0, \quad Q_- \mathcal{N}(\hat{C},t) = 0.
\]
Expanding $\mathcal{N}$ and $\hat{C}$, these equations become
\begin{equation*}
\begin{aligned}
Q_- \mathcal{N}(0) + \hat{T}_r \hat{C}_1 + Q_- R( \hat{C}_0 + \hat{C}_1) + Q_- iK(t) &= 0 \\
P_- \mathcal{N}(0) + \hat{T}_r \hat{C}_0 + P_- R(\hat{C}_0 + \hat{C}_1) + P_- iK(t) & = 0.
\end{aligned}
\end{equation*}
Let $S : Y^k_- \to Y^{k+1}_+$ be the inverse operator of the restricted operator $\hat{T}_r : Y^{k+1}_+ \to Y^k_-$.
Define $U \in Y^0_+$ to be
\[
U = -SQ_- \mathcal{N}(0).
\]
Applying $S$ to the first of the two equations above, we end up with
\[
\hat{C}_1 = U - SQ_-R(\hat{C}_0 + \hat{C}_1) -SQ_-iK(t).
\]
For the moment, fix some $\hat{C}_0 \in H_r^+$ and some $t \in U_i$. Then define $\widetilde{\mathcal{F}} : Y^2_+ \to Y^2_+$ as
\[
\widetilde{\mathcal{F}}(\hat{C}_1) = U - SQ_- R(\hat{C}_0 + \hat{C}_1) -SQ_-iK(t).
\]

Note that $\widetilde{\mathcal{F}}$ is well-defined as a map from $Y^2_+$ to itself.
This is easy to see since $\widetilde{\mathcal{F}}$ differs from the map $\mathcal{F}$ defined in Section~\ref{sec:glueloc} only by the term $-SQ_-iK(t)$.
As in Section~\ref{sec:glueloc}, define
\[
B_1(r^{-4}) = \{ \hat{C}_1 \in Y^2_+ \; | \; ||\hat{C}_1||_{2,2} \le r^{-4} \} \subset Y^2_+(r)
\]
and similarly
\[
B_0(r^{-4}) = \{ \hat{C}_0 \in H_r^+ \; | \; ||\hat{C}_0||_{2,2} \le r^{-4} \} \subset H_r^+.
\]

\begin{lemma}
\label{lem: pre0 for testimate}
The operator
\[
P_{-}^{r} \circ \Psi_{r} \circ L : T_{b_{i}}B \to H^{-}_{r}
\]
is an isomorphism.
\end{lemma}

\begin{proof}
First, note that an estimate given in the last paragraph of page 170 of Nicolaescu~\cite{nic2} implies that the map
\[
P_{-}^{r} \circ \Psi_{r} : H^{+}_{L^{2}}(\hat{N}) \to H^{-}_{r}
\]
is isomorphic.
Because of Remark~\ref{rem:linear} and condition (A5)-(a),
the map $t \mapsto \pi \circ L(t)$ is also isomorphic.
Therefore we can conclude that
\[
P_{-}^{r} \circ \Psi_{r} \circ \pi \circ L: H^{+}_{L^{2}}(\hat{N}) \to H^{-}_{r}
\]
is also isomorphic.

On the other hand, because of condition (A5)-(d), we have
\[
\| P_{-}^{r} \circ \Psi_{r} \circ \pi \circ L_{i}(t)
- P_{-}^{r} \circ \Psi_{r} \circ L_{i}(t)\|_{L^{2}(\hat{X}(r))}
< \| P_{-}^{r} \circ \Psi_{r} \circ \pi \circ L_{i}(t)\|_{L^{2}(\hat{X}(r))}
\]
for any $t$, therefore we have
\[
\| P_{-}^{r} \circ \Psi_{r} \circ \pi \circ L_{i}
- P_{-}^{r} \circ \Psi_{r} \circ L_{i}\|
< \| P_{-}^{r} \circ \Psi_{r} \circ \pi \circ L_{i}\|,
\]
here the norms of the both sides are the operator norms.
Since $P_{-}^{r} \circ \Psi_{r} \circ \pi \circ L$ is isomorphic, this inequality implies that $P_{-}^{r} \circ \Psi_{r} \circ L_{i}$ is also isomorphic.
\end{proof}

From the following Lemma~\ref{lem: coming from the def of asymp map} to Lemma~\ref{lem: pre for testimate}, we shall give the argument to determine the constant $\epsilon$ in condition (A5)-(c).
For the purpose we shall note that the constants appearing from Lemma~\ref{lem: coming from the def of asymp map} to Lemma~\ref{lem: pre for testimate} do not depend on choice of $\eta_{\hat{N}}$.

\begin{lemma}
\label{lem: coming from the def of asymp map}
There exists a sequence $C(r)$, which is independent of choice of $\eta_{\hat{N}}$, of positive numbers such that $C(r) \to 0$ as $r \to \infty$ and
\[
\| Q_{-}^{r} \circ \Psi_{r} \circ \pi \circ L(t) \|_{L^{2}(\hat{X}(r))}
\leq C(r) \| \pi \circ L(t) \|_{L^{2}_{\mu}(\hat{N})}
\]
holds for any $t \in U_{i}$.
\end{lemma}

\begin{proof}
First recall that the definition of the term ``asymptotic map" in Nicolaescu's book.
Let $H_{0}$ and $H_{1}$ be Hilbert spaces.
Let $U_{r}$ and $V_{r}$ be families parameterised by $r>0$ of closed subspaces of $H_{0}$ and $H_{1}$ respectively.
Let $f_{r}: U_{r} \to H_{1}$ be densely defined linear operators with closed ranges.
For closed subspaces $U, V$ of $H_{1}$, let us define
\begin{align*}
\hat{\delta}(U,V)
 := \sup\{ {\rm dist}(u,V) \mid u \in U, \|u\|=1 \}
 = \sup\{ \|P^{\perp}(u)\| \mid u \in U, \|u\|=1 \}.
\end{align*}
Here $P^{\perp} : H_{1} \to V^{\perp}$ is the orthogonal projection onto the orthogonal complement $V^{\perp}$ of $V$.
We call the family of triples $(U_{r}, V_{r}, f_{r})_{r}$ an {\it asymptotic map} if $\hat{\delta}(im(f_{r}),V_{r}) \to 0$ as $r \to 0$.
In such a case, we use the notation $f_{r} : U_{r} \to^{a} V_{r}$.

As noted in Nicolaescu's book page 307, the map $\Psi_{r} : H^{+}_{L^{2}}(\hat{N}) \to L^{2}(\hat{X}(r))$ defines an asymptotic map $\Psi_{r} : H^{+}_{L^{2}}(\hat{N}) \to^{a} H^{-}_{r}$.
Therefore there exists a sequence $C(r)$ of positive numbers such that $C(r) \to 0$ as $r \to \infty$ and
\[
\left\| Q_{-}^{r} \circ \Psi_{r}\left( \frac{\pi \circ L(t)}{\|\pi \circ L(t)\|} \right) \right\| \leq C(r)
\]
holds for any $t \in U_{i}$.
This proves the lemma.
\end{proof}

\begin{lemma}
\label{lem: pre1 for testimate}
We have
\[
\|P_{-}^{r} \circ \Psi_{r}\|_{\rm op} \leq 1,
\]
where $\|\cdot\|_{\rm op}$ denotes the operator norm for operators whose domain is $L^{2}_{\mu}(\hat{N}, \wedge^{+}T^{\ast} \hat{N})$ and codomain is $L^{2}(\hat{X}(r),\wedge^{+}T^{\ast} \hat{X}(r))$.
\end{lemma}

\begin{proof}
Let us take $\eta \in L^{2}_{\mu}(\hat{N}, \wedge^{+}T^{\ast} \hat{N})$.
Then we have
\[
\| P_{-}^{r} \circ \Psi_{r}(\eta)  \|_{L^{2}(\hat{X}(r),\wedge^{+}T^{\ast} \hat{X}(r))}
\leq \| P_{-}^{r}\|_{\rm op} \|\Psi_{r}(\eta)  \|_{L^{2}(\hat{X}(r),\wedge^{+}T^{\ast} \hat{X}(r))}
\]
and $\| P_{-}^{r}\|_{\rm op} \leq 1$, since $P_{-}^{r}$ is the $L^{2}$-orthogonal projection.
In addition, we have
\begin{align*}
\|\Psi_{r}(\eta)  \|_{L^{2}(\hat{X}(r))}
= \| \alpha_{r} \eta\|_{L^{2}(\hat{X}(r))}
\leq \| \alpha_{r} \eta\|_{L^{2}(\hat{N})}
\leq \| \eta\|_{L^{2}(\hat{N})}
\leq \| \eta\|_{L^{2}_{\mu}(\hat{N})}.
\end{align*}
This proves the lemma.
\end{proof}

\begin{lemma}
\label{lem: est for pi L }
There exists $C_{0}>0$, which depends only on the geometry of $\hat{N}$, such that 
\[
C_{0} \| \pi \circ L(t) \|_{L^{2}_{\mu}(\hat{N})}
\leq \| \Psi_{\mu} \circ \pi \circ L(t) \|_{L^{2}(\hat{X}(r))}
\]
holds for any $t \in U_{i}$ and all sufficiently large $r$.
\end{lemma}

\begin{proof}
First, since $H^{+}_{L^{2}}(\hat{N})$ is finite dimensional, the $L^{2}$-norm on $H^{+}_{L^{2}}(\hat{N})$ and $L^{2}_{\mu}$-norm on $H^{+}_{L^{2}}(\hat{N})$ are equivalent.
Therefore it suffices to show that there exists $C>0$, which depends only on the geometry of $\hat{N}$, such that 
\[
C \| \pi \circ L(t) \|_{L^{2}(\hat{N})}
\leq \| \Psi_{\mu} \circ \pi \circ L(t) \|_{L^{2}(\hat{X}(r))}
\]
holds for all sufficiently large $r$.
\begin{align}
&\| \pi \circ L(t) \|_{L^{2}(\hat{N})}
- \| \Psi_{\mu} \circ \pi \circ L(t) \|_{L^{2}(\hat{X}(r))}\nonumber\\
=&\| \pi \circ L(t) \|_{L^{2}(\hat{N})}
- \| \Psi_{\mu} \circ \pi \circ L(t) \|_{L^{2}(\hat{N})}\nonumber\\
\leq&\| \pi \circ L(t) 
- \Psi_{\mu} \circ \pi \circ L(t) \|_{L^{2}(\hat{N})}\nonumber\\
=&\| (1-\alpha_{r}) \cdot \pi \circ L(t) \|_{L^{2}(\hat{N})}\nonumber\\
=&\| (1-\alpha_{r}) \cdot \pi \circ L(t) \|_{L^{2}([r,r+1] \times S^{3})} \nonumber\\
\leq&\| \pi \circ L(t) \|_{L^{2}([r,r+1] \times S^{3})}.
\label{eq: ineq for lem: est for pi L }
\end{align}
Note that, because of exponential decay estimate for $L^{2}$-harmonic forms, there exists $r_{0}>0$ such that 
\[
\|\eta\|_{L^{2}([r,r+1] \times S^{3})}
\leq \frac{1}{2}\|\eta\|_{L^{2}(\hat{N})}
\]
holds for any $\eta \in H^{+}_{L^{2}}(\hat{N})$ and any $r \geq r_{0}$.
Thus we have $\| \pi \circ L(t) \|_{L^{2}([r,r+1] \times S^{3})} \leq \frac{1}{2} \| \pi \circ L(t) \|_{L^{2}(\hat{N})}$ for sufficiently large $r$.
This inequality and the inequality \eqref{eq: ineq for lem: est for pi L } imply
\[
\| \pi \circ L(t) \|_{L^{2}(\hat{N})}
- \| \Psi_{\mu} \circ \pi \circ L(t) \|_{L^{2}(\hat{X}(r))}
\leq \frac{1}{2}\| \pi \circ L(t) \|_{L^{2}(\hat{N})},
\]
and thus we have
\[
\frac{1}{2}\| \pi \circ L(t) \|_{L^{2}(\hat{N})}
\leq \| \Psi_{\mu} \circ \pi \circ L(t) \|_{L^{2}(\hat{X}(r))}
\]
for sufficiently large $r$.
\end{proof}

The constant $\epsilon$ in (A5)-(c) is determined in the proof of the following lemma:

\begin{lemma}
\label{lem: pre for testimate}
There exists $C>0$, which is independent of $r$, such that
\[
\| P_{-}^{r} \circ \Psi_{r} \circ L(t) \|_{L^{2}(\hat{X}(r))}
\geq C \|\pi \circ L(t)\|_{L^{2}_{\mu}(\hat{N})}
\]
holds for any $t \in U_{i}$.
\end{lemma}

\begin{proof}
From the decomposition $L(t) = \pi \circ L(t) + \pi^{\perp} \circ L(t)$, we get
\begin{align*}
&\| P_{-}^{r} \circ \Psi_{r} \circ L(t) \|_{L^{2}(\hat{X}(r))}\\
\geq& \| P_{-}^{r} \circ \Psi_{r} \circ\pi \circ  L(t) \|_{L^{2}(\hat{X}(r))} -\| P_{-}^{r} \circ \Psi_{r} \circ \pi^{\perp} \circ L(t) \|_{L^{2}(\hat{X}(r))}\\
\geq& \| P_{-}^{r} \circ \Psi_{r} \circ\pi \circ  L(t) \|_{L^{2}(\hat{X}(r))} -\| \pi^{\perp} \circ L(t) \|_{L^{2}_{\mu}(\hat{N})}\\
\geq& \| P_{-}^{r} \circ \Psi_{r} \circ\pi \circ  L(t) \|_{L^{2}(\hat{X}(r))} - \epsilon\| \pi \circ L(t) \|_{L^{2}_{\mu}(\hat{N})}\\
\geq& \| \Psi_{r} \circ\pi \circ  L(t) \|_{L^{2}(\hat{X}(r))} - \| Q_{-}^{r} \circ \Psi_{r} \circ\pi \circ  L(t) \|_{L^{2}(\hat{X}(r))} - \epsilon\| \pi \circ L(t) \|_{L^{2}_{\mu}(\hat{N})}\\
\geq& \| \Psi_{r} \circ\pi \circ  L(t) \|_{L^{2}(\hat{X}(r))} - (C(r) + \epsilon) \| \pi \circ L(t) \|_{L^{2}_{\mu}(\hat{N})}\\
\geq& (C_{0} - C(r) - \epsilon) \| \pi \circ L(t) \|_{L^{2}_{\mu}(\hat{N})}.
\end{align*}
Here, in the second, third, fifth, and sixth inequality, we used Lemma~\ref{lem: pre1 for testimate}, (A5)-(c), Lemma~\ref{lem: coming from the def of asymp map}, and Lemma~\ref{lem: est for pi L }.
Henceforth let us focus only on $r$ satisfying $C(r) \leq C_{0}/2$.
Then 
$C_{0} - C(r) - \epsilon \geq C_{0}/2 -\epsilon$.
Here, we fix $\epsilon$ so that $C_{0}/2 -\epsilon > 0$.
Then the above estimates prove the lemma.
\end{proof}

\begin{lemma}
\label{lem:testimate}
For all sufficiently large $r$, we have that if $\hat{C}_0 \in B_0(r^{-4})$ and $\hat{C}_1 \in B_1(r^{-4})$, then there is a constant $C$ independent of $r$ for which
\[
|| t || \le C r^{-6}
\]
whenever $\mathcal{N}( \hat{C}_0 , \hat{C}_1 , t ) = 0$.
\end{lemma}

\begin{proof}
Expanding the equation $P_- \mathcal{N}(\hat{C},t) = 0$ out, we get
\begin{equation}\label{equ:equfort}
P_- \mathcal{N}(0) + \hat{T}_r( \hat{C}_0  ) + P_- R(\hat{C}_0 + \hat{C}_1) + P_-(iK(t)) = 0.
\end{equation}
Let
\[
\rho_r : T_{b_{i}}B \to H^{-}_{r}
\]
be the operator defined by
$\rho_r(t) = P_{-}^{r} \circ \Psi_{r} \circ L(t) = P_{-}^{r} \circ K_{r}(t)$.
Because of Lemma~\ref{lem: pre0 for testimate}, $\rho_r$ is invertible.
Let $J_r$ be the inverse of $\rho_r$. 
Lemma~\ref{lem: pre for testimate} implies that
\[
\| \rho_{r}(t) \|_{L^{2}(\hat{X}(r))}
\geq C \|\pi \circ L(t)\|_{L^{2}_{\mu}(\hat{N})}
\]
for any $t \in T_{b_{i}}B$.
Here, since $\pi \circ L$ is isomorphic, there exists $C_{1}>0$ such that
\[
\|\pi \circ L(t)\|_{L^{2}_{\mu}(\hat{N})} \geq C_{1} \|t\|
\]
for any $t$.
Therefore we have
\[
\| \rho_{r}(t) \|_{L^{2}(\hat{X}(r))}
\geq C_{2} \|t\|
\]
for some $C_{2}>0$.
Thus we have $\| \rho_{r}\|_{\rm op} \geq C_{2}$, and hence $\| J_{r}\|_{\rm op} \leq 1/C_{2}$.

Applying $J_r$ to the equation (\ref{equ:equfort}), we get
\[
t = - J_rP_- \mathcal{N}(0) -J_r \hat{T}_r( \hat{C}_0  ) -J_r P_- R(\hat{C}_0 + \hat{C}_1).
\]
Then we estimate:
\begin{equation*}
\begin{aligned}
||t|| &\le || J_r||( ||P_- \mathcal{N}(0) || + ||\hat{T}_r(\hat{C}_0)|| + ||P_- R(\hat{C}_0 + \hat{C}_1)||) \\
&\le C_{3}( e^{-\mu r} + r^{-2}||\hat{C}_0|| + ||R(\hat{C}_0 + \hat{C}_1)|| ) \\
&\le C_{4}( e^{-\mu r} + r^{-6} + r^{3/2} (r^{-4} + r^{-4})^2 ) \\
&\le C_{5} r^{-6}
\end{aligned}
\end{equation*}
for all large enough $r$.
\end{proof}

By this lemma, when considering solutions of the Seiberg-Witten equations over $U_i$ such that $||\hat{C}_0||, ||\hat{C}_1|| < r^{-4}$, we may as well assume $||t|| \le C r^{-6}$.

Until this point, we have not given any restriction on $\delta$.
In the following lemma, we specify the size of $\delta$ to be small enough:

\begin{lemma}
For all sufficiently large $r$, all sufficiently small $\delta$, all $\hat{C}_0 \in B_0(r^{-4})$ and all $t \in U_i$ satisfying $||t|| \le Cr^{-6}$, the map $\widetilde{\mathcal{F}}$ sends $B_1(r^{-4})$ to itself and is a contraction mapping.
\end{lemma}

\begin{proof}
Since $\widetilde{\mathcal{F}}$ differs from $\mathcal{F}$ only by a constant term, we immediately have that $\widetilde{\mathcal{F}}$ acts as a contraction on $B_1(r^{-4})$ for all sufficiently large $r$, provided that we can show that it sends $B_1(r^{-4})$ to itself.
Recall that we had obtained an estimate of the form
\[
|| \mathcal{F}(\hat{C}_1)|| \le C( r^2 e^{-\mu r} + r^{-13/2} ).
\]
Hence, we immediately have
\[
|| \widetilde{\mathcal{F}}(\hat{C}_1)|| \le C(r^2 e^{-\mu r} + r^{-13/2}) + ||SQ_- iK(t) ||.
\]
But $|| SQ_- iK(t) || \le C r^2 || Q_- K(t)|| \le C r^4 e^{-\mu r} + Cr^2 \delta ||t||$ by Lemma \ref{lem:estimate}.
Hence $\widetilde{\mathcal{F}}$ sends $B_1(r^{-4})$ to itself for all large enough $r$ and small enough $\delta$.
\end{proof}

Thus for each $\hat{C}_0 \in B_0(r^{-4})$ and each $t \in U_i$ with $||t|| \le Cr^{-6}$, there is a uniquely determined fixed point of $\widetilde{\mathcal{F}}$, which will be denoted as $\hat{C}_1 = \Phi(\hat{C}_0 , t) \in B_1(r^{-4})$.
It can be shown that $\Phi(\hat{C}_0 , t)$ depends differentiably on $\hat{C}_0$ and $t$ (by the implicit function theorem).

Now consider again the pair of equations
\[
P_- \mathcal{N}(\hat{C}_0 + \hat{C}_1 , t) = 0, \quad Q_- \mathcal{N}(\hat{C}_0 + \hat{C}_1 , t) = 0.
\]
The second of these equations is solved by the substitution $\hat{C}_1 = \Phi(\hat{C}_0,t)$ and we are left with just the single equation for $\hat{C}_0$ and $t$:
\[
P_- \mathcal{N}( \hat{C}_0 + \Phi(\hat{C}_0,t) , t) = 0.
\]
Expanding this out, the equation becomes
\begin{equation}\label{equ:secondeqn}
P_- \mathcal{N}(0) + \hat{T}_r( \hat{C}_0  ) + P_- R(\hat{C}_0 + \Phi(\hat{C}_0 , t)) + P_-(iK(t)) = 0.
\end{equation}
We apply $J_r$ to the equation \eqref{equ:secondeqn} to obtain:
\begin{equation}\label{equ:fixt}
t = -J_rP_- \mathcal{N}(0) - J_r \hat{T}_r( \hat{C}_0  ) - J_r P_- R(\hat{C}_0 + \Phi(\hat{C}_0 , t)) = -J_r \mathcal{N}(C_0 + \Phi(\hat{C}_0 , t)).
\end{equation}
Fix $\hat{C}_0 \in B_0(r^{-4})$.
We try to solve this as an equation for $t$ using the contraction mapping principle.
Define $G : U_i \to T_{b_i} B$ by:
\[
G(t) = -J_r P_-\mathcal{N}(C_0 + \Phi(\hat{C}_0 , t)).
\]
Then equation \eqref{equ:fixt} reduces to the fixed point equation $G(t) = t$.

\begin{lemma}
There is a constant $C$ independent of $r$ such that for all sufficiently large $r$ we have that $G$ sends $V = \{ t \; | \; ||t|| \le Cr^{-6} \}$ to itself.
\end{lemma}

\begin{proof}
This is just a consequence of Lemma \ref{lem:testimate}.
\end{proof}

Thus for all large enough $r$ we have that $G$ defines a map $G : V \to V$.
Next we want to show this is a contraction. 

\begin{lemma}
\label{lem:contractest}
Let $\mathcal{F}_t$ be a family of contractions on $B_1(r^{-4})$ of the form 
\[
\mathcal{F}_t(\hat{C}_1) = \mathcal{F}(\hat{C}_1) + \lambda(t)
\]
where $\mathcal{F}$ is a contraction.
Let $0 < \kappa < 1$ be such that
\[
||\mathcal{F}(\hat{C}_1) - \mathcal{F}(\hat{C}_2)|| \le \kappa ||\hat{C}_1 - \hat{C}_2||.
\]
Let $\Phi(t)$ be the unique fixed point of $\mathcal{F}_t$.
Then we have
\[
|| \Phi(t_1) - \Phi(t_2) || \le \frac{1}{1-\kappa} || \lambda(t_1) - \lambda(t_2) ||.
\]
\end{lemma}

\begin{proof}
We start with the defining relations $\mathcal{F}_{t_i}( \Phi(t_i) ) = \Phi(t_i)$ for $i=1,2$.
Thus
\begin{equation*}
\begin{aligned}
|| \Phi(t_1) - \Phi(t_2) || &\le || \mathcal{F}_{t_1}( \Phi(t_1) ) - \mathcal{F}_{t_2}(\Phi(t_2)) || \\
&\le || \mathcal{F}(\Phi(t_1)) - \mathcal{F}(\Phi(t_2)) + \lambda(t_1) - \lambda(t_2) || \\
&\le ||\mathcal{F}(\Phi(t_1)) - \mathcal{F}(\Phi(t_2)) || + ||\lambda(t_1) - \lambda(t_2) || \\
&\le \kappa || \Phi(t_1) - \Phi(t_2) || + ||\lambda(t_1) - \lambda(t_2) ||.
\end{aligned}
\end{equation*}
Hence $(1-\kappa)|| \Phi(t_1) - \Phi(t_2) || \le ||\lambda(t_1) - \lambda(t_2) ||$ and the result follows.
\end{proof}

\begin{corollary}
\label{cor:est}
For any $t_1,t_2 \in V$, we have an estimate of the form
\[
||\Phi( \hat{C}_0 , t_1) - \Phi(\hat{C}_0 , t_2)|| \le C r^2|| t_1 - t_2||.
\]
\end{corollary}

\begin{proof}
We apply Lemma \ref{lem:contractest}, where $\mathcal{F}$ is as usual and $\lambda(t) = -SQ_- (iK(t))$. Note that by the proof of Lemma \ref{lem:fcontract}, we can take $\kappa = 1/2$. Therefore, we get an estimate:
\[
|| \Phi(\hat{C}_0 , t_1) - \Phi(\hat{C}_0 , t_2) || \le 2 || SQ_-( iK(t_1) - iK(t_2) ) || \le Cr^2 ||t_1 - t_2||.
\]
\end{proof}

\begin{lemma}
For all sufficiently large $r$, $G$ is a contraction on $V$.
\end{lemma}
\begin{proof}
Let $t_1,t_2 \in V$. We estimate:
\begin{equation*}
\begin{aligned}
|| G(t_1) - G(t_2) || &\le || J_r || ||P_- \mathcal{N}(C_0 + \Phi(\hat{C}_0 , t_1)) - P_- \mathcal{N}(C_0 + \Phi(\hat{C}_0 , t_2)) ||  \\
&\le C || P_- R(\hat{C}_0 + \Phi(\hat{C}_0 , t_1)) - P_- R(\hat{C}_0 + \Phi(\hat{C}_0 , t_2)) || \\
&\le C || R(\hat{C}_0 + \Phi(\hat{C}_0 , t_1)) - R(\hat{C}_0 + \Phi(\hat{C}_0 , t_2)) || \\
&\le C r^{3/2} r^{-4} || \Phi(\hat{C}_0 , t_1) - \Phi(\hat{C}_0 , t_2) || \\
&\le C r^{-1/2} ||t_1 - t_2|| \; \text{by Corollary \ref{cor:est}.}
\end{aligned}
\end{equation*}
Hence for all sufficiently large $r$, this is a contraction.
\end{proof}

Thus for all large enough $r$, we see that for each $\hat{C}_0 \in B_0(r^{-4})$ there is a unique $t = t( \hat{C}_0)$ such that $(\hat{C} = \hat{C}_0 + \Phi(\hat{C}_0 , t(\hat{C}_0)) , t(\hat{C}_0))$ is a solution of the Seiberg-Witten equations for the family $E_{\hat{X}(r)}$.
In particular, setting $\hat{C}_0 = 0$, we get a solution $(\hat{C} , t) = ( \Phi( 0 , t(0) ) , t(0))$, which we call $\Psi_r( \hat{C}_r)$.

\begin{proposition}\label{prop:decomp}
For all large enough $r$, the mapping $\hat{C}_r \mapsto \Psi_r( \hat{C}_r)$ defines a diffeomorphism between the disjoint union
\[
\coprod_{i=1}^l \mathcal{M}( \hat{M} , g_{\hat{M}}(b_i) , \eta_{\hat{M}}(b_i) , \mathfrak{s}_M , \mu)
\]
with the families moduli space
\[
\mathcal{M}( E_{\hat{X}(r)} , g_{\hat{X}(r)} , \eta_{\hat{X}(r)} , \mathfrak{s}_X ).
\]
\end{proposition}
\begin{proof} 
We just have to adapt the arguments of Nicolaescu's book \textsection 4.5.3 and 4.5.4.
\end{proof}

\begin{remark}
Recall also that each moduli space $\mathcal{M}( \hat{M} , g_{\hat{M}}(b_i) , \eta_{\hat{M}}(b_i) , \mathfrak{s}_M , \mu)$ is diffeomorphic to a smooth Seiberg-Witten moduli space for the compact manifold $M$.
\end{remark}

\noindent {\bf Completion of proof of Theorem \ref{thm:rub} part (1):}
Recall that we assume $dim(B) = d = b^+(N)$. Therefore $\theta \in H^0(B  , \mathbb{Z}_2)$ and we may as well take $\theta = 1$. For simplicity, we assume $\rho$ is trivial and consider the $\mathbb{Z}_2$-valued invariant. The gluing formula for the $\rho$-twisted and $\mathbb{Z}$-valued invariants are a straightforward extensions of this case. So we have to show that
\[
FSW_m^{\mathbb{Z}_2}( E_X , \mathfrak{s}_X , 1) = SW(M , \mathfrak{s}_M) \langle w_{b^+}(H^+(N)) , [B] \rangle \in \mathbb{Z}_2
\]
But $\langle w_{b^+}(H^+(N)) , [B] \rangle$ is just the count (mod $2$) of the number of zeros of a generic section of $H^+(N)$. Thus $\langle w_{b^+}(H^+(N)) , [B] \rangle = l$ and we must show
\[
FSW_m^{\mathbb{Z}_2}( E_X , \mathfrak{s}_X , 1) = SW(M , \mathfrak{s}_M) l.
\]
But from Proposition \ref{prop:decomp} we have diffeomorphisms
\begin{align*}
\mathcal{M}( E_{\hat{X}(r)} , g_{\hat{X}(r)} , \eta_{\hat{X}(r)} , \mathfrak{s}_X )
&\cong \coprod_{i=1}^l \mathcal{M}( \hat{M} , g_{\hat{M}}(b_i) , \eta_{\hat{M}}(b_i) , \mathfrak{s}_M , \mu)\\
&\cong \coprod_{i=1}^l \mathcal{M}( M , g_M(b_i) , \eta_{M}(b_i) , \mathfrak{s}_M )
\end{align*}
(for suitably defined metrics $g_M(b_i)$ and perturbations $\eta_M(b_i)$ on $M$).
As in Nicolaescu \textsection 4.5.4 this diffeomorphism can be seen the respect the isomorphism class of the line bundles over the moduli spaces.
It follows immediately that
\[
FSW_m^{\mathbb{Z}_2}( E_X , \mathfrak{s}_X , 1) = \sum_{i=1}^l SW( M , \mathfrak{s}_M) = SW(M , \mathfrak{s}_M) l
\]
and the proof is complete.

\section{Proof of Theorem \ref{thm:rub} part (2)}\label{sec:part2}

Let $\theta \in H^{s+2k}(B , \mathbb{Z}_2)$, where $s = dim(B) - b^+(N) \ge 0$ and $k > 0$.
By the solution to the Steenrod problem, $\theta$ is Poincar\'e dual to a homology class of the form $f_*[S] \in H_{b^+(N)-2k}( B , \mathbb{Z}_2)$, where $S$ is a compact smooth manifold of dimension $dim(S) = b^+(N)-2k < b^+(N)$ and $f : S \to B$ is a smooth map.
Then
\begin{equation*}
\begin{aligned}
FSW_{m-k}^{\mathbb{Z}_2}(E_X , \mathfrak{s}_X , \theta ) &= \langle c_1(\mathcal{L})^{m-k} \cup \pi^* \theta, [\mathcal{M}] \rangle \\
&= \langle \pi_* ( c_1(\mathcal{L})^{m-k} \cup \pi^* (\theta)) , [B] \rangle \text{ (by definition of the wrong-way map $\pi_*$)} \\
&= \langle \pi_* ( c_1(\mathcal{L})^{m-k} ) \cup \theta , [B] \rangle \\
&= \langle \pi_* ( c_1(\mathcal{L})^{m-k} ) , [S] \rangle \text{ ($\theta$ is Poinar\'e dual to $[S]$)} \\
&= \langle  c_1(\mathcal{L})^{m-k}  , [\mathcal{M}|_S] \rangle \text{ (by definition of the wrong-way map $\pi_*$)} \\
&= FSW_{m-k}^{\mathbb{Z}_2}( E_X|_S , \mathfrak{s}_X , 1).
\end{aligned}
\end{equation*}
This reduces the problem to computing the families Seiberg-Witten invariant for the restricted family $E_X|_S$.
We can analyse the families moduli space associated to $E_X|_S$ using gluing theory.
However, this time since $\dim(S) < b^+(N)$, we have that $\eta_{\hat{N}}$ can be chosen so as to never meet the wall, so we can take $l = 0$ in Proposition~\ref{prop:perturb}.
Then by Corollary~\ref{cor:confine}, we have that the families moduli space for $E_{\hat{X}(r)}|_S$ is empty for all sufficiently large $r$.
Hence $FSW^{\mathbb{Z}_2}( E_X|_S , \mathfrak{s}_X , 1) = 0$.


\section{Applications of the gluing formula}\label{sec:applications}

\subsection{General setup}\label{sec:setup2}

In this section we briefly recall the setup for the gluing formula from Section \ref{sec:setup}. In the following sections we consider various applications of this formula.

Let $M,N$ be compact, smooth, oriented $4$-manifolds and set $X = M \# N$. Let $B$ be a compact smooth manifold of dimension $d$ and suppose we have smooth fibrewise oriented families $\pi_M : E_M \to B$, $\pi_N : E_N \to B$ whose fibres are $M,N$ respectively. Let $T(E_M/B) = Ker( {\pi_M}_* )$, $T(E_N/B) = Ker( {\pi_N}_* )$ be the vertical tangent bundles. To form the connected sum family, suppose we are given the following additional data:
\begin{itemize}
\item[(1)]{Sections $\iota_M : B \to E_M$ and $\iota_N : B \to E_N$.}
\item[(2)]{An orientation reversing vector bundle isomorphism $\phi : \iota_M^*( T(E_M/B)) \to \iota_N^*( T(E_N/B))$.}
\end{itemize}

Given this data, we can construct a fibrewise connected sum family $E_X \to B$ as described in more detail in Section \ref{sec:setup}.\\

Let $\mathfrak{s}_M, \mathfrak{s}_N$ be Spin$^c$-structures on $M,N$ and let $\mathfrak{s}_X$ be the corresponding Spin$^c$-structure on $X$. Assume that:
\begin{itemize}
\item[(3)]{the Spin$^c$-structures $\mathfrak{s}_M,\mathfrak{s}_N$ are monodromy invariant under the monodromy action of $\pi_1(B)$, up to a common sign factor given by a homomorphism $\rho : \pi_1(B) \to \mathbb{Z}_2$.}
\end{itemize}
Then $\mathfrak{s}_X$ is also invariant up to the same sign factor $\rho$.\\

Assume further that:
\begin{itemize}
\item[(4)]{$d(M,\mathfrak{s}_M) = 0$.}
\item[(5)]{$c_1(\mathfrak{s}_N)^2 = \sigma(N)$.}
\item[(6)]{$b_1(N) = 0$.}
\item[(7)]{$0 < b^+(N) \le dim(B) = d$.}
\item[(8)]{$b^+(M) > ( dim(B) - b^+(N) ) + 1$.}
\end{itemize}

\begin{remark}
The assumption that $d(M , \mathfrak{s}_M) = 0$ is not strictly necessary, but is useful for the following two reasons. First, it ensures that we can use Theorem \ref{thm:rub} (1) without any requirement that $\mathfrak{s}_X$ extends to a Spin$^c$-structure on $T(E_X/B)$. Second, at the time of writing, all known examples of pairs $(M , \mathfrak{s}_M)$ with $b^+(M) > 1$ and $SW(M , \mathfrak{s}_M) \neq 0$ satisfy $d(M , \mathfrak{s}_M) = 0$.
\end{remark}

By these assumptions, it follows that
\[
d(N,\mathfrak{s}_N) = -1 - b^+(N) < 0
\]
and that
\[
d(X , \mathfrak{s}_X) = -b^+(N).
\]

Note that assumption (8) ensures that we have a well-defined Seiberg-Witten invariant for the family $E_X$ which takes values in $H^{b^+(N)}(B , \mathbb{Z}_2)$.

\begin{theorem}[Gluing formula, $\mathbb{Z}_2$-valued case]
Under assumptions (1)-(8) we have:
\[
FSW^{\mathbb{Z}_2}(E_X , \mathfrak{s}_X) = SW( M , \mathfrak{s}_M) \cdot w_{b^+(N)}(H^+(N)) \in H^{b^+(N)}(B , \mathbb{Z}_2)
\]
in the untwisted case. More generally, in the $\rho$-twisted case we have:
\[
FSW^{\mathbb{Z}_2}(E_X , \mathfrak{s}_X , \rho) = SW( M , \mathfrak{s}_M) \cdot w_{b^+(N)}(H^+(N) \otimes \mathbb{R}_\rho ) \in H^{b^+(N)}(B , \mathbb{Z}_2)
\]
where $\mathbb{R}_\rho \to B$ is the real line bundle classified by $\rho \in H^1(B , \mathbb{Z}_2)$.
\end{theorem}

\subsection{A vanishing theorem}\label{sec:vanishing}

The gluing formula immediately gives us:
\begin{proposition}
Suppose assumptions (1)-(8) of \textsection \ref{sec:setup2} hold.
\begin{itemize}
\item[(i)]{If $SW(M,\mathfrak{s}_M) = 0  \; ({\rm mod} \; 2)$ or if $w_{b^+(N)}(H^+(N) \otimes \mathbb{R}_\rho) = 0$, then we have $FSW^{\mathbb{Z}_2}(E_X , \mathfrak{s}_X , \rho) = 0$.}
\item[(ii)]{Suppose $B, det^+(E_M)$ and $det^+(E_N) \cong det(H^+(N))$ are oriented and $\rho$ is trivial. If $SW(M , \mathfrak{s}_M) = 0$ or if $e(H^+(N)) = 0$, then $FSW^{\mathbb{Z}}(E_X , \mathfrak{s}_X) = 0$.}
\end{itemize}
\end{proposition}

\begin{remark}
Suppose that $B$ is simply-connected. Then $\rho$ is the trivial homomorphism and we have that $w_{b^+(N)}(H^+(N))$ and $e(H^+(N))$ both vanish. To see this, note that $H^2(N , \mathbb{R})$ defines a flat vector bundle with structure group $O( b^+(N) , b^-(N))$. If $B$ is simply-connected, then $H^2(N , \mathbb{R})$ is trivialisable as an $O( b^+(N) , b^-(N))$-bundle. Hence $H^+(N)$ defines a trivialisable vector bundle on $B$.
\end{remark}

The above remark shows that in order to use the gluing formula to get a non-zero families Seiberg-Witten invariant, we must take $B$ to have a non-trivial fundamental group.

\subsection{Smooth isotopy classes of diffeomorphisms}\label{sec:isotopy}

In this section, we consider the case that $B$ is $1$-dimensional. The families Seiberg-Witten invariants for $1$-dimensional families was studied by Ruberman in \cite{rub1,rub3} and a special case of the $1$-parameter gluing formula was proven in \cite{rub3}. Even for $1$-parameter families, our gluing formula is more general.

We assume that $B = S^1$ is the circle. Any family $E_X \to S^1$ is the mapping cylinder of a diffeomorphism $f : X \to X$, which we denote by $Cyl(f)$. Note that up to isomorphism the family $Cyl(f)$ depends only on the smooth isotopy class of $f$.\\

Let $f : X \to X$ be a diffeomorphism. Suppose that $b^+(X) > 2$ and suppose that $f$ preserves a Spin$^c$-structure $\mathfrak{s}_X$ up to sign:
\[
f^*(\mathfrak{s}_X) = (-1)^\rho \mathfrak{s}_X, \quad \rho \in \mathbb{Z}_2.
\]
Then we can consider the families Seiberg-Witten invariant attached to the family $Cyl(f)$ and Spin$^c$-structure $\mathfrak{s}_X$. Assume also that
\[
d(X , \mathfrak{s}_X) = -1.
\]
Then we obtain a $\mathbb{Z}_2$-valued invariant by setting:
\[
SW^{\mathbb{Z}_2}(X , \mathfrak{s}_X , f , \rho) = FSW^{\mathbb{Z}_2}( Cyl(f) , \mathfrak{s}_X , \rho) \in H^1(S^1 , \mathbb{Z}_2) \cong \mathbb{Z}_2.
\]
This is the $1$-parameter families Seiberg-Witten invariant considered by Ruberman in \cite{rub1}. Note that $SW^{\mathbb{Z}_2}(X , \mathfrak{s}_X , f , \rho)$ depends only on the smooth isotopy class of $f$. Note also that the condition $d(X , \mathfrak{s}_X) = -1$ forces $b^+(X)$ to be even. We can also think of $SW^{\mathbb{Z}_2}(X , \mathfrak{s}_X , f , \rho)$ as the equivariant Seiberg-Witten invariant associated to the $G = \mathbb{Z}$-action determined by $f$.\\

Suppose now that $X = M \# N$ is a connected sum and suppose that we have families $E_M = Cyl(f_M), E_N = Cyl(f_N)$ which are mapping cylinders associated to diffeomorphisms $f_M, f_N$. Assume that $f_M$ acts as the identity in the neighbourhood of a point $m \in M$ and similarly that $f_N$ acts as the identity in the neighbourhood of a point $n \in N$. Then we construct $X$ as the connected sum by removing small balls around $m$ and $n$ and identifying their boundaries. In this way, $f_M,f_N$ can be composed to give a diffeomorphism of $X$ which we denote as $f_X = f_M \# f_N$. Suppose that $f_M$, $f_N$ preserve Spin$^c$-structures $\mathfrak{s}_M, \mathfrak{s}_N$ up to a common sign $\rho \in \mathbb{Z}_2$ and set $\mathfrak{s}_X = \mathfrak{s}_M \# \mathfrak{s}_N$.\\ 

Observe that $f_N$ acts as an orientation preserving isometry of $H^2(N , \mathbb{R})$. As the space of maximal positive definite subspace of $H^2(N , \mathbb{R})$ is connected, it makes sense to say whether $f_N$ preserves or reverses the orientation of maximal definite subspaces and we can associate to $f_N$ a corresponding sign factor $\pm 1$. Abusing notation, we will denote this sign factor as $det(f_N|_{H^+(N)})$, even though $f_N$ may not preserve any particular choice of maximal positive definite subspace $H^+(N)$. Specialising the gluing formula to this setting, we obtain:

\begin{theorem}
Let $(M,N, f_M, f_N , \mathfrak{s}_M , \mathfrak{s}_N , \rho)$ be as above. Assume also that:
\begin{itemize}
\item[(i)]{$d(M,\mathfrak{s}_M) = 0$,}
\item[(ii)]{$c_1(\mathfrak{s}_N)^2 = \sigma(N)$,}
\item[(iii)]{$b_1(N) = 0$,}
\item[(iv)]{$b^+(N) = 1$,}
\item[(v)]{$b^+(M) > 1$.}
\end{itemize}
Then
\[
SW^{\mathbb{Z}_2}(X , \mathfrak{s}_X , f_X , \rho ) = \begin{cases} SW(M , \mathfrak{s}_M) \; ({\rm mod} \; 2) , &\text{ if } det(f_N|_{H^+(N)}) = (-1)^{\rho+1}, \\ 0, &\text{ if } det(f_N|_{H^+(N)}) = (-1)^{\rho}. \end{cases}
\]
\end{theorem}

We consider two special cases for $(N , \mathfrak{s}_N , f_N)$:
\begin{itemize}
\item[(i)]{Let $N = S^2 \times S^2$, let $\mathfrak{s}_N$ be the unique Spin$^c$-structure with $c_1(\mathfrak{s}_N) = 0$ and let $f_N$ be given by $f_N(x,y) = (rx,ry)$, where $r : S^2 \to S^2$ is a reflection. Observe that $c_1(\mathfrak{s}_N)^2 = \sigma(N)$ and that $f_N$ preserves $\mathfrak{s}_N$.}
\item[(ii)]{Let $N = \mathbb{CP}^2 \#^2  \overline{\mathbb{CP}}^2 = (S^2 \times S^2) \# \overline{\mathbb{CP}}^2$. Then $H^3(N , \mathbb{Z}) \cong \mathbb{Z}^3$, where the first two factors of $\mathbb{Z}$ correspond to $H^2(S^2 \times S^2 , \mathbb{Z})$ and the last factor corresponds to $H^2( \overline{\mathbb{CP}}^2 , \mathbb{Z})$. Let $\mathfrak{s}_N$ be the unique Spin$^c$-structure with $c_1(\mathfrak{s}_N) = (0,0,1)$. Then $c_1(\mathfrak{s}_N)^2 = \sigma(N)$. We define $f_N$ as an equivariant connected sum of a diffeomorphism $f_1 : S^2 \times S^2 \to S^2 \times S^2$ and $f_2 : \overline{\mathbb{CP}}^2 \to \overline{\mathbb{CP}}^2$ as follows. Let $f_1(x,y) = (y,x)$. Then $f_1$ is an involution with fixed point set $S^2$. Let $f_2[z_0,z_1,z_2] = [z_0,z_1,-z_2]$. Then $f_1$ is an involution with fixed point set $\mathbb{CP}^1 \cup {pt}$. Since $f_1$ and $f_2$ are involutions whose fixed point sets contain $2$-dimensional components, it is possible to form their equivariant connected sum $f_N = f_1 \# f_2$. Observe that $f_N$ preserves $\mathfrak{s}_N$.}
\end{itemize}
In both of these cases the diffeomorphism $f_N$ does not act as the identity in a neighbourhood of any point. However we can replace $f_N$ with a diffeomorphism which is smoothly isotopic to $f_N$ and which acts as the identity in a neighbourhood of some point. By abuse of notation we let this new diffeomorphism be denoted as $f_N$.

\begin{corollary}
Let $M$ be a compact oriented smooth $4$-manifold with $b^+(M) > 1$. Let $\mathfrak{s}_M$ be a Spin$^c$-structure with $d(M , \mathfrak{s}_M) = 0$. Let $(N , \mathfrak{s}_N , f_N)$ be given by either (i) or (ii) above. Then
\[
SW^{\mathbb{Z}_2}( M \# N , \mathfrak{s}_M \# \mathfrak{s}_N , id_M \# f_N , 0) = SW(M , \mathfrak{s}_M) \; ({\rm mod} \; 2).
\]
\end{corollary}

\begin{theorem}
Let $M, M'$ be compact simply-connected smooth $4$-manifolds with indefinite intersection forms. Suppose that $M,M'$ are homeomorphic and fix a homeomorphism $\phi : M \to M'$. Suppose that $b^+(M) > 1$ and that $\mathfrak{s}_M$ is a Spin$^c$-structure on $M$ with $d(M,\mathfrak{s}_M) = 0$ and that $SW(M,\mathfrak{s}_M) \neq SW(M' , \phi(\mathfrak{s}_M) ) \; ({\rm mod} \; 2)$. Lastly, suppose that $X = M \# (S^2 \times S^2)$ is diffeomorphic to $M' \# (S^2 \times S^2)$. Then there exists a diffeomorphism on $X$ which is continuously isotopic to the identity but not smoothly isotopic.
\end{theorem}
\begin{proof}
Let $(N , \mathfrak{s}_N , f_N)$ be as in (i) above. The homeomorphism $\phi : M \to M'$ determines a homeomorphism $\phi' : M \# (S^2 \times S^2) \to M \# (S^2 \times S^2)$. Since $M$ is indefinite, using a theorem of Wall \cite[Theorem 2]{wall}, we can choose the diffeomorphism $\psi : M \# (S^2 \times S^2) \to M' \# (S^2 \times S^2)$ such that $\psi$ and $\phi'$ induce the same action on cohomology groups. In particular, we have $\psi( \mathfrak{s}_M \# \mathfrak{s}_N) = \phi(\mathfrak{s}_M) \# \mathfrak{s}_N$.\\

To prove the result, it suffices to show there are diffeomorphisms $f_1,f_2$ of $X$ which induce the same map on $H^2(X,\mathbb{Z}_2)$ and are therefore continuously isotopic by \cite{qui}, but which have different families Seiberg-Witten invariants. Then $f_1 \circ f_2^{-1}$ will be continuously isotopic to the identity but not smoothly.\\

Let $f_1 = id_M \# f_N$ and let $f_2 = \psi^{-1} \circ (id_{M'} \# f_N ) \circ \psi$. Then $f_1,f_2$ act the same way on $H^2(X , \mathbb{Z})$ since $\psi^* = (\phi')^*$ commutes with $(f_N)^*$. It remains to check that $f_1, f_2$ have different Seiberg-Witten invariants. We have
\[
SW^{\mathbb{Z}_2}(X , \mathfrak{s}_X , f_1 , 0) = SW(M , \mathfrak{s}_M) \; ({\rm mod} \; 2).
\]
On the other hand, since $\psi$ gives an isomorphism between the mapping cylinders of $\psi^{-1} \circ (id_{M'} \# f_N ) \circ \psi$ and $id_{M'} \# f_N $, we find:
\begin{equation*}
\begin{aligned}
SW^{\mathbb{Z}_2}(X , \mathfrak{s}_X , f_2 , 0) &= SW^{\mathbb{Z}_2}( M \# (S^2 \times S^2) , \mathfrak{s}_M \# \mathfrak{s}_N , \psi^{-1} \circ (id_{M'} \# f_N ) \circ \psi , 0) \\
&= SW^{\mathbb{Z}_2}( M' \# (S^2 \times S^2) , \phi(\mathfrak{s}) \# \mathfrak{s}_N , id_{M'} \# f_N , 0) \\
&= SW(M' , \phi(\mathfrak{s}_M) ) \; ({\rm mod} \; 2).
\end{aligned}
\end{equation*}
This completes the proof, since $SW(M,\mathfrak{s}_M) \neq SW(M' , \phi(\mathfrak{s}_M) ) \; ({\rm mod} \; 2)$.
\end{proof}

\begin{corollary}
The following $4$-manifolds admit diffeomorphisms which are continuously isotopic to the identity, but are not smoothly isotopic to the identity:
\begin{itemize}
\item[(i)]{$X = \#^n(S^2 \times S^2) \#^n (K3)$, for any $n \ge 2$.}
\item[(ii)]{$X = \#^{2n} \mathbb{CP}^2 \#^m \overline{\mathbb{CP}}^2$, for any $n \ge 2$ and any $m \ge 10n+1$.}
\end{itemize}
\end{corollary}
\begin{proof}
We apply the previous corollary, where in case (i) $M$ is the elliptic surface $E(2n)$ and $M' = \#^{n-1} (S^2 \times S^2) \#^n (K3)$ and in case (ii) $M = E(n) \#^{m-10n} \overline{\mathbb{CP}}^2$ and $M' = \#^{2n-1} \mathbb{CP}^2 \#^{m-1} \overline{\mathbb{CP}}^2$. Then use the fact that if $V$ is a simply-connected elliptic surface, then $V \# (S^2 \times S^2)$ and $V \# \mathbb{CP}^2$ both dissolve (see Corollaries 8 and 9 of \cite{gom}).
\end{proof}

\subsection{The mod $2$ Seiberg-Witten invariant for spin structures}\label{sec:mod2spin}

In this section we will use the gluing formula to give a simple new proof of a theorem of Morgan and Szab\'o:

\begin{theorem}[\cite{mosa}]
Let $M$ be smooth compact spin $4$-manifold with $b_1(M) = 0$, $b^+(M) = 4n-1$, $b^-(M) = 20n-1$, where $n \ge 1$. Let $\mathfrak{s}_M$ be a Spin$^c$-structure on $M$ which comes from a spin structure. Then $SW(M , \mathfrak{s}_M)$ is odd if $n=1$ and is even otherwise.
\end{theorem}

\begin{remark}
Note that for a compact spin $4$-manifold with $b_1(M)=0$, the conditions $b^+(M) = 4n-1$, $b^-(M) = 20n-1$, for some $n \ge 1$ are equivalent to requiring that $d(M,\mathfrak{s}_M) = 0$.
\end{remark}

\begin{proof}
Let $N = \#^{n-1}(S^2 \times S^2) \#^n (K3)$ and observe that $b_1(N) = 0$, $b^+(N) = 4n-1$, $b^-(N) = 20n-1$. Let $\mathfrak{s}_N$ be the unique Spin$^c$-structure on $N$ which comes from a spin structure. Then $SW(N , \mathfrak{s}_N)$ is odd if $n=1$ (since in this case $N = K3$) and is zero if $n > 1$ (by the vanishing of the Seiberg-Witten invariants on connected sums with each summand having $b^+ > 0$). Thus it suffices to prove that $SW(M , \mathfrak{s}_M) = SW(N , \mathfrak{s}_N) \; ({\rm mod} \; 2)$. Let $\overline{N}$ denote $N$ with the opposite orientation and $\mathfrak{s}_{\overline{N}}$ the unique Spin$^c$-structure on $\overline{N}$ which comes from the spin structure. Let $X = M \# \overline{N} \# N$, $\mathfrak{s}_X = \mathfrak{s}_M \# \mathfrak{s}_{\overline{N}} \# \mathfrak{s}_N$. Consider the trivial family $E_X = X \times \mathbb{RP}^{24n-2}$ over $\mathbb{RP}^{24n-2}$ equipped with the twisting class $\rho$ given by the unique non-trivial element of $H^1(\mathbb{RP}^{24n-2} , \mathbb{Z}_2)$. Identify $H^{24n-2}(\mathbb{RP}^{24n-2} , \mathbb{Z}_2)$ with $\mathbb{Z}_2$. We will evaluate $FSW^{\mathbb{Z}_2}(E_X , \mathfrak{s}_X , \rho)$ using the gluing formula in two different ways. First, writing $X$ as:
\[
X = M \# \left( \overline{N} \# N \right)
\]
and noting that $\sigma( \overline{N} \# N) = 0$, $b^+(\overline{N} \# N) = 24n-2$, we see from the gluing formula that:
\[
FSW^{\mathbb{Z}_2}(E_X , \mathfrak{s}_X , \rho) = SW(M , \mathfrak{s}_M) \; ({\rm mod} \; 2).
\]
On the other hand, writing $X$ as:
\[
X = N \# \left( \overline{N} \# M \right)
\]
we similarly find that:
\[
FSW^{\mathbb{Z}_2}(E_X , \mathfrak{s}_X , \rho) = SW(N , \mathfrak{s}_N) \; ({\rm mod} \; 2)
\]
which completes the proof.
\end{proof}

\subsection{Relations between mod $2$ Seiberg-Witten invariants}\label{sec:relations}

The argument used in the previous section to compute mod $2$ Seiberg-Witten invariants for spin structures can be applied more generally. As a result we find a surprising relation between the mod $2$ Seiberg-Witten invariants on different $4$-manifolds. To state the result, let us first describe the setup.\\

Let $M,M',N$ be compact, smooth, oriented $4$-manifolds. Let $B$ be a compact smooth manifold of dimension $d$ and suppose we have smooth fibrewise oriented families $\pi_M : E_M \to B$, $\pi_{M'} : E_{M'} \to B$, $\pi_N : E_N \to B$. Suppose we are given:
\begin{itemize}
\item[(1)]{Sections $\iota_M : B \to E_M$, $\iota_{M'} : B \to E_{M'}$ and a pair of sections and $\iota_N, \iota'_N : B \to E_N$ with disjoint images.}
\item[(2)]{Orientation reversing isomorphisms between the normal bundles of $(\iota_M , \iota_N)$ and $(\iota_{M'} , \iota'_N)$.}
\end{itemize}

Let $X = M \# N \# M'$. From the above data, we can form the families connected sum of $E_M, E_{M'}, E_N$.\\

Let $\mathfrak{s}_M, \mathfrak{s}_{M'}, \mathfrak{s}_N$ be Spin$^c$-structures on $M, M', N$ and let $\mathfrak{s}_X$ be the corresponding Spin$^c$-structure on $X$. Assume that:
\begin{itemize}
\item[(3)]{the Spin$^c$-structures $\mathfrak{s}_M, \mathfrak{s}_{M'}, \mathfrak{s}_N$ are monodromy invariant up to a common sign factor $\rho : \pi_1(B) \to \mathbb{Z}_2$.}
\end{itemize}

Assume further that:
\begin{itemize}
\item[(4)]{$b^+(M) = b^+(M')$.}
\item[(5)]{$b_1(M) = b_1(M') = b_1(N) = 0$.}
\item[(6)]{$c_1(\mathfrak{s}_M)^2 - \sigma(M) = c_1(\mathfrak{s}_{M'})^2 - \sigma(M') = 4(b^+(M)+1) = -c_1(\mathfrak{s}_N)^2 + \sigma(N)$.}
\item[(7)]{$0 < b^+(N) + b^+(M) \le dim(B) = d$.}
\item[(8)]{$2b^+(M) > ( dim(B) - b^+(N) ) + 1$.}
\end{itemize}

Then we have a well-defined Seiberg-Witten invariant for the family $E_X$ which takes values in $H^{b^+(M) + b^+(N)}(B , \mathbb{Z}_2)$.

\begin{proposition}\label{prop:relation}
Under the above assumptions we have the following equality in $H^{b^+(M) + b^+(N)}(B , \mathbb{Z}_2)$:
\begin{equation}\label{equ:relation}
\begin{aligned}
&& SW(M , \mathfrak{s}_M) w_{b^+(M)}(H^+(M')\otimes \mathbb{R}_\rho)w_{b^+(N)}(H^+(N)\otimes \mathbb{R}_\rho) = \\
&& \quad SW(M' , \mathfrak{s}_{M'}) w_{b^+(M)}(H^+(M)\otimes \mathbb{R}_\rho)w_{b^+(N)}(H^+(N)\otimes \mathbb{R}_\rho).
\end{aligned}
\end{equation}
\end{proposition}

\begin{remark}
Let $\pi_M : E_M \to B$ satisfy the above assumptions. Since $c_1(\mathfrak{s}_M)^2 - \sigma(M) > 0$, the arguments used in \cite{bar} implies that $w_{b^+(M)}( H^+(M) ) = 0$ in the case that $\rho$ is trivial. In the case that $\rho$ is non-trivial, the same argument used in \cite{bar} can be adapted, provided $c_1(\mathfrak{s}_M)^2 - \sigma(M) = 8 \; ({\rm mod} \; 16)$ to show that $w_{b^+(M)}(H^+(M) \otimes \mathbb{R}_\rho) = 0$. Thus the statement of Proposition \ref{prop:relation} is not vacuous only when $\rho$ is non-trivial and $c_1(\mathfrak{s}_M)^2 - \sigma(M) = 0 \; ({\rm mod} \; 16)$. Note that this also implies $b^+(M) = 3 \; ({\rm mod} \; 4)$.
\end{remark}

\begin{proof}
We write $X = M \# N \# M'$ in two different ways, as $X = M \# \left( N \# M' \right)$ and as $X = M' \# \left( N \# M \right)$. Applying the gluing formula to each of these and equating gives (\ref{equ:relation}).
\end{proof}

\begin{corollary}\label{cor:relation}
Let $M,M'$ be compact, smooth, oriented $4$-manifolds. Let $B$ be a compact smooth manifold and suppose we have smooth fibrewise oriented families $\pi_M : E_M \to B$, $\pi_{M'} : E_{M'} \to B$. Suppose there exists sections $\iota_M : B \to E_M$, $\iota_{M'} : B \to E_{M'}$ whose normal bundles are trivial. Let $\mathfrak{s}_M, \mathfrak{s}_{M'}$ be Spin$^c$-structures on $M, M'$ Assume that $\mathfrak{s}_M, \mathfrak{s}_{M'}$ are monodromy invariant up to a common sign factor $\rho : \pi_1(B) \to \mathbb{Z}_2$. Assume further that:
\begin{itemize}
\item[(4)]{$b^+(M) = b^+(M')$.}
\item[(5)]{$b_1(M) = b_1(M') = 0$.}
\item[(6)]{$c_1(\mathfrak{s}_M)^2 - \sigma(M) = c_1(\mathfrak{s}_{M'})^2 - \sigma(M') = 4(b^+(M)+1)$.}
\item[(7)]{$b^+(M) = 3 \; ({\rm mod} \; 4)$.}
\item[(8)]{$1+ \frac{dim(B)}{2} \le b^+(M) \le dim(B)$.}
\end{itemize}

Then we have the following equality in $H^{b^+(M)}(B , \mathbb{Z}_2)$:
\[
SW(M , \mathfrak{s}_M) w_{b^+(M)}(H^+(M')\otimes \mathbb{R}_\rho) = SW(M' , \mathfrak{s}_{M'}) w_{b^+(M)}(H^+(M)\otimes \mathbb{R}_\rho).
\]
\end{corollary}
\begin{proof}
Since $b^+(M) = 3 \; ({\rm mod} \; 4)$, we have $4(b^+(M)+1) = 16k$ for some $k \ge 0$. Let $N = \#^k \overline{K3}$. Let $E_N$ be the trivial family over $B$. Set $B' = B \times \mathbb{RP}^{19k}$. Set $\rho' = (\rho , 1) \in H^1(B' , \mathbb{Z}_2) = H^1(B , \mathbb{Z}_2) \oplus \mathbb{Z}_2$. Pull the families $E_M, E_{M'}, E_N$ back to $B'$ and apply Proposition \ref{prop:relation}. The result now follows, noting that in this case the factor $w_{b^+(N)}(H^+(N)\otimes \mathbb{R}_\rho)$ can be cancelled from both sides of (\ref{equ:relation}).
\end{proof}

As an illustration of this result, we find remarkably that the existence of a certain type of involution on a $4$-manifold $M$ implies $M$ has a non-vanishing Seiberg-Witten invariant:
\begin{corollary}\label{cor:involutionnonvanish}
Let $M$ be a compact smooth $4$-manifold with $b_1(M) = 0$ and $b^+(M) = 3$. Suppose that $f : M \to M$ is an involutive diffeomorphism of $M$ whose fixed point set contains an isolated point. Suppose $\mathfrak{s}_M$ is a Spin$^c$-structure such that $f(\mathfrak{s}_M) = -\mathfrak{s}_M$ and $c_1(\mathfrak{s}_M)^2 - \sigma(M) = 16$. Then $SW(M , \mathfrak{s}_M) = 1 \; ({\rm mod} \; 2)$ if $f$ acts trivially on $H^+(M)$ and $SW(M , \mathfrak{s}_M) = 0 \; ({\rm mod} \; 2)$ otherwise.
\end{corollary}
\begin{proof}
Let $B = \mathbb{RP}^3$. We consider the family $E_M = M \times_{\mathbb{Z}_2} S^3$, where $\mathbb{Z}_2$ acts on $S^3$ by the antipodal map. Let $m \in M$ be an isolated fixed point. Then $m$ defines a section of $E_M$ whose normal bundle is isomorphic to $\mathbb{R}_-^4$, where $\mathbb{R}_-$ is the unique non-trivial line bundle on $\mathbb{RP}^3$. This is easily seen to be a trivialisable bundle. Now repeat this construction for $M' = K3$ equipped with an involution $f'$ with isolated fixed points and such that $f'$ acts on $H^+(M')$ trivially (such an involution exists, see e.g. \cite{bry}). Take $\mathfrak{s}_{M'}$ to be the Spin$^c$-structure with $c_1(\mathfrak{s}_{M'}) = 0$. Take $\rho$ to be the unique non-trivial element of $H^1(\mathbb{RP}^3 , \mathbb{Z}_2)$. Then applying Corollary \ref{cor:relation}, we get $SW(M , \mathfrak{s}_M) = SW( K3 , \mathfrak{s}_{M'}) = 1 \; ({\rm mod} \; 2)$ if $f$ acts trivially on $H^+(M)$ and $SW(M , \mathfrak{s}_M) = 0 \; ({\rm mod} \; 2)$ otherwise.
\end{proof}

\begin{proposition}\label{prop:involutionvanish}
Let $M$ be a compact smooth $4$-manifold with $b_1(M) = 0$, $b^+(M) = 3 \; ({\rm mod} \; 4)$ and $b^+(M) \neq 3$. Suppose that $f : M \to M$ is an involutive diffeomorphism of $M$ whose fixed point set contains an isolated point. Suppose $\mathfrak{s}_M$ is a Spin$^c$-structure such that $f(\mathfrak{s}_M) = -\mathfrak{s}_M$ and $c_1(\mathfrak{s}_M)^2 - \sigma(M) = 4(b^+(M)+1)$. Then $SW(M , \mathfrak{s}_M) = 0 \; ({\rm mod} \; 2)$.
\end{proposition}
\begin{proof}
First we construct an involution on $K3$ as follows. Recall the Kummer construction of $K3$. Start the $4$-torus $T^4 = \mathbb{C}^2/(\mathbb{Z}^2 + i\mathbb{Z}^2)$ equipped with the involution $i : T^4 \to T^4$ defined by $i((z_1,z_2)) = -(z_1,z_2)$. There are $16$ fixed points of $i$. Take the blowup $Y = T^4 \#^{16} \overline{\mathbb{CP}}^2$ of $T^4$ at the fixed points. Then $i$ extends to an involution on $Y$ with only non-isolated fixed points ($16$ spheres each with self-intersection $-1$). The quotient $K3 = Y/\langle i \rangle$ can be given the structure of a smooth manifold  and is a $K3$ surface. The $-1$ curves in $Y$ become $-2$ curves in $K3$. Let $f : T^4 \to T^4$ be the involution $f(z_1,z_2) = (z_1 +1/2 , z_2)$. Then $f$ commutes with $i$ and acts freely on the fixed points set of $i$. It follows that $f$ determines an involution on $Y$ which permutes the $16$ $-1$ curves in $Y$ in pairs. Since $f$ commutes with $i$, we have that $f$ descends to an involution $f : K3 \to K3$ which permutes the $16$ $-2$ curves in $K3$. Note that $f$ has $8$ isolated fixed points and by \cite{bry} acts as the identity on $H^+(K3)$. Let $S_1,S_2$ be two of the $-2$ curves in $K3$ which may be chosen so that $S_1, S_2, f(S_1), f(S_2)$ are pairwise disjoint. Let $c = 2[S_1] +2[S_2]-2[f(S_1)]-2[f(S_2)]$. Then $c = c_1(\mathfrak{s}_{K3})$ for a uniquely determined characteristic on $K3$. Moreover $f(c) = -c$, which implies that $f(\mathfrak{s}_{K3}) = -\mathfrak{s}_{K3}$. Also we have $c^2 = -32$, so that $-c_1(\mathfrak{s}_{K3})^2 + \sigma(K3) = 16$.\\

Now define $k \ge 2$ by $16k = 4(b^+(M)+1)$ and set $N = \#^{k} K3$. We equip $N$ with the involution $f_N$ which is the equivariant connected sum of $k$ copies of $f : K3 \to K3$. Also 
we take $\mathfrak{s}_N = \#^k \mathfrak{s}_{K3}$. Then $f_N(\mathfrak{s}_N) = -\mathfrak{s}_N$ and $-c_1(\mathfrak{s}_N)^2 + \sigma(N) = 16k$.\\

Next, suppose that $(M' , f_{M'} , \mathfrak{s}_{M'})$ also satisfies the hypotheses of the proposition. In fact we will take $(M' , f_{M'})$ to be an equivariant connected sum of the form $M' = \#^{k-1}(S^2 \times S^2) \#^k (K3)$, where on each factor of $K3$ we use the involution $f$ and on each factor of $S^2 \times S^2$ we use the involution $\phi \times \phi$, where $\phi : S^2 \to S^2$ is a rotation by $\pi$ about some axis. Note that $\phi \times \phi$ has $4$ isolated fixed points. We also take $\mathfrak{s}_{M'}$ to be the Spin$^c$-structure with characteristic zero. Now we apply Proposition \ref{prop:relation}, where $B = \mathbb{RP}^{7k-1}$ and where $E_M, E_{M'}, E_N$ are the families over $B$ associated to the involutions on $M,M',N$. Also $\rho$ is taken to be the unique non-trivial element of $H^1(\mathbb{RP}^{7k-1} , \mathbb{Z}_2)$. We perform the families connected sum using the isolated fixed points. Proposition \ref{prop:relation} now gives $SW(M , \mathfrak{s}_M) = SW( M' , \mathfrak{s}_{M'}) = 0 \; ({\rm mod} \; 2)$ if $f$ acts trivially on $H^+(M)$ and $SW(M , \mathfrak{s}_M) = 0 \; ({\rm mod} \; 2)$ otherwise. In either case, we get $SW(M , \mathfrak{s}_M) = 0 \; ({\rm mod} \; 2)$.
\end{proof}

\begin{remark}
It is worth comparing Proposition \ref{prop:involutionvanish} with the following similar but much more elementary result: let $M$ be a compact smooth $4$-manifold with $b_1(M)=0$, $b^+(M) = 1 \; ({\rm mod} \; 4)$ and $b^+(M) > 1$. Suppose that $f : M \to M$ is a diffeomorphism of $M$ such that $det(f|_{H^+(M)}) = 1$ and suppose $\mathfrak{s}_M$ is a Spin$^c$-structure such that $f(\mathfrak{s}_M) = -\mathfrak{s}_M$ and $c_1(\mathfrak{s}_M)^2 - \sigma(M) = 4(b^+(M)+1)$. Then $SW(M , \mathfrak{s}_M) = 0$. The proof is simple: first, by charge conjugation symmetry we have: 
\[
SW(M , \mathfrak{s}_M) = (-1)^{\tfrac{b^+(M)+1}{2}}SW(M , -\mathfrak{s}_M) = -SW(M , -\mathfrak{s}_M).
\]
On the other hand since $det(f|_{H^+(M)}) = 1$, we have by diffeomorphism invariance that 
\[
SW(M , \mathfrak{s}_M) = SW(M , f(\mathfrak{s}_M)) = SW(M , -\mathfrak{s}_M).
\]
Hence $SW(M , \mathfrak{s}_M) = 0$.
\end{remark}

Using the blowup formula for Seiberg-Witten invariants, we can strengthen Proposition \ref{prop:involutionvanish} as follows:
\begin{corollary}
Let $M$ be a compact smooth $4$-manifold with $b_1(M) = 0$, $b^+(M) = 3 \; ({\rm mod} \; 4)$ and $b^+(M) \neq 3$. Suppose that $f : M \to M$ is an involutive diffeomorphism of $M$ whose fixed point set contains an isolated point. Suppose $\mathfrak{s}_M$ is a Spin$^c$-structure such that $f(\mathfrak{s}_M) = -\mathfrak{s}_M$ and $c_1(\mathfrak{s}_M)^2 - \sigma(M) = 0 \; ({\rm mod} \; 16)$. Then $SW(M , \mathfrak{s}_M) = 0 \; ({\rm mod} \; 2)$.
\end{corollary}
\begin{proof}
We can assume $c_1(\mathfrak{s}_M)^2 - \sigma(M) \ge 4(b^+(M)+1)$, since otherwise $d(M , \mathfrak{s}_M) < 0$. Let $c_1(\mathfrak{s}_M)^2 - \sigma(M) - 4(b^+(M)+1) = 16k$, where $k \ge 1$. Taking the blowup of $M$ at $k$ pairs of non-fixed points $(m_1, f(m_1)), \dots , (m_k , f(m_k))$ we obtain an involution $f'$ on $M' = M \#^{2k} \overline{\mathbb{CP}}^2$. Let $\mathfrak{s}_{M'} = \mathfrak{s}_M \# \mathfrak{s}$, where $c_1(\mathfrak{s}) = (3,-3,3,-3, \dots , 3,-3)$. Then $(M' , f' , \mathfrak{s}_{M'})$ satisfies the conditions of Proposition \ref{prop:involutionvanish}, so that $SW(M' , \mathfrak{s}_{M'}) = 0 \; ({\rm mod} \; 2)$. But $SW(M , \mathfrak{s}_M) = SW(M' , \mathfrak{s}_{M'})$ by the blowup formula for Seiberg-Witten invariants.
\end{proof}

Using the same blowup argument as above we can similarly strengthen Corollary \ref{cor:involutionnonvanish}:
\begin{corollary}\label{cor:involutionnonvanish2}
Let $M$ be a compact smooth $4$-manifold with $b_1(M) = 0$ and $b^+(M) = 3$. Suppose that $f : M \to M$ is an involutive diffeomorphism of $M$ whose fixed point set contains an isolated point. Suppose suppose $\mathfrak{s}_M$ is a Spin$^c$-structure such that $f(\mathfrak{s}_M) = -\mathfrak{s}_M$, $c_1(\mathfrak{s}_M)^2 - \sigma(M) = 0 \; ({\rm mod} \; 16)$ and $c_1(\mathfrak{s}_M)^2 - \sigma(M) \ge 16$. Then $SW(M , \mathfrak{s}_M) = 1 \; ({\rm mod} \; 2)$ if $f$ acts trivially on $H^+(M)$ and $SW(M , \mathfrak{s}_M) = 0 \; ({\rm mod} \; 2)$ otherwise.
\end{corollary}

\begin{example}
Suppose $M$ is a compact smooth spin $4$-manifold with $b_1(M) = 0$, $b^+(M) = 3$ and $\sigma(M) \le -32$. Suppose that $f : M \to M$ is a smooth involution which has an isolated fixed point. Then $f$ acts non-trivially on $H^+(M)$. To see this, suppose on the contrary that $f$ acts trivially on $H^+(M)$. Then by Corollary \ref{cor:involutionnonvanish2}, we would find that $SW(M , \mathfrak{s}_M) = 1 \; ({\rm mod} \; 2)$ where $\mathfrak{s}_M$ is a spin structure on $M$. But since $d(M,\mathfrak{s}_M) > 0$ and $b^+(M) = 3$, we can deduce that $SW(M , \mathfrak{s}_M) = 0 \; ({\rm mod} \; 2)$ by \cite[Corollary 3.6]{bafu}.
\end{example}

\subsection{Seiberg-Witten invariants of group actions}\label{sec:applgroupaction}

Let $X$ be a smooth compact oriented $4$-manifold and let $G$ be a group acting on $X$ by orientation preserving diffeomorphisms. Suppose that $G$ preserves the isomorphism class of a Spin$^c$-structure $\mathfrak{s}_X$. Suppose that $d(X , \mathfrak{s}_X) \le 0$ and that $b^+(X) > -d(X , \mathfrak{s}_X) + 1$. Recall from Section \ref{sec:SWgroupaction} that in this situation we may define the $G$-equivariant Seiberg-Witten invariant $SW_G^{\mathbb{Z}_2}(X , \mathfrak{s}_X) \in H^{-d(X,\mathfrak{s}_X)}(BG, \mathbb{Z}_2)$. More generally, such an invariant is defined if $G$ only preserves $\mathfrak{s}_X$ up to a sign factor $\rho : G \to \mathbb{Z}_2$ and we again obtain an invariant $SW_G^{\mathbb{Z}_2}(X , \mathfrak{s}_X , \rho) \in H^{-d(X , \mathfrak{s}_X)}(BG , \mathbb{Z}_2)$.

Recall also from Section \ref{sec:SWgroupaction} that $SW_G^{\mathbb{Z}_2}(X , \mathfrak{s}_X , \rho)$ is related to the families Seiberg-Witten invariant as follows. Let $B$ be a smooth compact manifold with $dim(B) < b^+(X)-1$ and let $E \to B$ be a principal $G$-bundle over $B$. Let $E_X = E \times_G X \to B$ be the associated family of $4$-manifolds over $B$. Then $\mathfrak{s}_X$ defines
\[
FSW^{\mathbb{Z}_2}( E_X , \mathfrak{s}_X) = f^*(SW_G^{\mathbb{Z}_2}(X , \mathfrak{s}_X)),
\]
where $f : B \to BG$ is the classifying map of $E \to B$.

Next, we will use the gluing formula to compute $SW_G^{\mathbb{Z}_2}$ under some conditions. Suppose first that $G$ acts smoothly on a smooth compact $4$-manifold $N$. Then $G$ acts on $H^2(N , \mathbb{R})$ preserving the intersection form. We therefore have an assoicated vector bundle $\mathbb{H}^2(N , \mathbb{R}) = EG \times_G H^2(N , \mathbb{R})$ with structure group $O(b^+(N) , b^-(N))$. We can reduce this to the maximal compact subgroup $O(b^+(N)) \times O(b^-(N))$, which amounts to decomposing $\mathbb{H}^2(N , \mathbb{R})$ into the direct sum $\mathbb{H}^+(N) \oplus \mathbb{H}^-(N)$ of a maximal positive definite subbundle and its orthogonal complement. We will abuse notation and write $H^+(N)$ instead of $\mathbb{H}^+(N)$. In particular, we may consider the top Stiefel-Whitney class of $H^+(N)$:
\[
w_{b^+(N)}( H^+(N) ) \in H^{b^+(N)}(BG , \mathbb{Z}_2).
\]

\begin{theorem}[Gluing formula for $SW_G^{\mathbb{Z}_2}$]\label{thm:rubequivariant}
Suppose that $X = M \# N$ is a $G$-equivariant connected sum, where $M,N$ are smooth compact $4$-manifolds equipped with an action of $G$ by diffeomorphisms. Suppose that $G$ preserves a Spin$^c$-structure $\mathfrak{s}_X = \mathfrak{s}_M \# \mathfrak{s}_N$ up to a sign factor given by a homomorphism $\rho : G \to \mathbb{Z}_2$. Suppose in addition that the following assumptions hold:
\begin{itemize}
\item[(i)]{$d(M , \mathfrak{s}_M) = 0$.}
\item[(ii)]{$c_1(\mathfrak{s}_N)^2 = \sigma(N)$.}
\item[(iii)]{$b_1(N) = 0$.}
\item[(iv)]{$b^+(M) > 1$.}
\end{itemize}
Then $SW_G^{\mathbb{Z}_2}(X , \mathfrak{s}_X , \rho) \in H^{b^+(N)}(BG , \mathbb{Z}_2)$ is defined and given by:
\begin{equation}\label{equ:gequivar}
SW_G^{\mathbb{Z}_2}(X , \mathfrak{s}_X , \rho) = SW(M , \mathfrak{s}_M) \cdot w_{b^+(N)}( H^+(N) \otimes \mathbb{R}_\rho)
\end{equation}
where $\mathbb{R}_\rho$ is the real line bundle over $BG$ determined by $\rho \in H^1(BG , \mathbb{Z}_2)$.
\end{theorem}
\begin{proof}
To show Equation \ref{equ:gequivar} it suffices by the universal coefficient theorem to show that
\[
\langle SW_G^{\mathbb{Z}_2}(X , \mathfrak{s}_X , \rho) , \beta) = SW(M , \mathfrak{s}_M) \cdot w_{b^+(N)}( H^+(N) \otimes \mathbb{R}_\rho) , \beta \rangle
\]
for every $\beta \in H_{-d(X , \mathfrak{s}_X)}(G , \mathbb{Z}_2)$. By the solution to the Steenrod problem with $\mathbb{Z}_2$-coefficients, there exists a compact smooth manifold $B$ of dimension $-d(X,\mathfrak{s}_X)$ and a continuous map $f : B \to BG$ such that $f_*[B] = \beta $. Let $E = f^*(EG)$ be the pullback of the universal bundle $EG \to BG$ and let $E_X = E \times_G X$ be the family associated to $E$. Then
\[
\langle SW_G^{\mathbb{Z}_2}(X , \mathfrak{s}_X , \rho) , \beta \rangle = \langle f^*(SW_G^{\mathbb{Z}_2}(X , \mathfrak{s}_X) , \rho) , [B] \rangle = \langle FSW^{\mathbb{Z}_2}(E_X , \mathfrak{s}_X , \rho ) , [B] \rangle.
\]
But since $X = M \# N$ is an equivariant connected sum, it follows that $E_X$ is the families connected sum of $E_M = E \times_G M$ and $E_N = E \times_G N$. Thus we are in the setting of the families gluing formula, which gives:
\[
FSW^{\mathbb{Z}_2}(E_X , \mathfrak{s}_X , \rho) = SW(M , \mathfrak{s}_M) \cdot w_{b^+(N)}(  f^* (H^+(N) \otimes \mathbb{R}_\rho ) ).
\]
Hence we have that
\begin{equation*}
\begin{aligned}
\langle SW_G^{\mathbb{Z}_2}(X , \mathfrak{s}_X , \rho) , \beta \rangle &= SW(M , \mathfrak{s}_M) \langle f^*( w_{b^+(N)}(H^+(N) \otimes \mathbb{R}_\rho ) ) , [B] \rangle \\
&= SW(M , \mathfrak{s}_M) \langle w_{b^+(N)}(H^+(N) \otimes \mathbb{R}_\rho) , \beta \rangle
\end{aligned}
\end{equation*}
as required.
\end{proof}

\subsection{Seiberg-Witten invariants of involutions}\label{sec:involutions}

In this section we take $G = \mathbb{Z}_2 = \langle f \rangle$. The cohomology ring of $B\mathbb{Z}_2$ with $\mathbb{Z}_2$-coefficients is given by
\[
H^*( B\mathbb{Z}_2 , \mathbb{Z}_2) \cong \mathbb{Z}_2[u]
\]
where $u \in H^1( B\mathbb{Z}_2 , \mathbb{Z}_2)$ classifies the unique non-trivial real line bundle $\mathbb{R}_-$ over $B\mathbb{Z}_2$.\\

Let $\rho : G \to \mathbb{Z}_2$ be a homomorphism. Thus either $\rho$ is the trivial homomorphism which we write as $\rho = 0$, or $\rho$ is the identity homomorphism which we write as $\rho = 1$. Note that $\mathbb{R}_\rho = \mathbb{R}$ for $\rho = 0$ and $\mathbb{R} = \mathbb{R}_-$ for $\rho = 1$.\\

Let $N$ be a smooth compact $4$-manifold. An action of $G$ on $N$ by diffeomorphisms is equivalent to giving an involutive diffeomorphism $f_N : N \to N$. Since $G$ is finite, we can always choose a $G$-invariant metric on $N$. Then $G$ acts on $H^+(N)$. We then find that $w_{b^+(N)}( H^+(N))$ is non-zero if and only if $f_N$ acts as $-Id$ on $H^+(N)$ and similarly $w_{b^+(N)}( H^+(N) \otimes \mathbb{R}_-)$ is non-zero if and only if $f_N$ acts as $+Id$ on $H^+(N)$.

\begin{proposition}\label{prop:involution}
Suppose that $X = M \# N$ is a $\mathbb{Z}_2$-equivariant connected sum, where $M,N$ are smooth compact $4$-manifolds equipped with involutions $f_M : M \to M$ and $f_N : N \to N$. Suppose that there are Spin$^c$-structures $\mathfrak{s}_M, \mathfrak{s}_N$ and a $\rho \in \mathbb{Z}_2$ such that:
\[
f_M^*(\mathfrak{s}_M) = (-1)^\rho \mathfrak{s}_M, \quad f_N^*(\mathfrak{s}_N) = (-1)^\rho \mathfrak{s}_N.
\]
Set $f = f_M \# f_N$ and $\mathfrak{s}_X = \mathfrak{s}_M \# \mathfrak{s}_N$. Suppose in addition that the following assumptions hold:
\begin{itemize}
\item[(i)]{$d(M , \mathfrak{s}_M) = 0$.}
\item[(ii)]{$c_1(\mathfrak{s}_N)^2 = \sigma(N)$.}
\item[(iii)]{$b_1(N) = 0$.}
\item[(iv)]{$b^+(M) > 1$.}
\end{itemize}
Then $SW^{\mathbb{Z}_2}_{\mathbb{Z}_2}(X , \mathfrak{s}_X , \rho) \in H^{b^+(N)}( B\mathbb{Z}_2 , \mathbb{Z}_2)$ is defined and given by:
\[
SW^{\mathbb{Z}_2}_{\mathbb{Z}_2}(X , \mathfrak{s}_X , \rho) = \begin{cases} SW(M , \mathfrak{s}_M) u^{b^+(N)},  & \text{if } \rho = 0 \text{ and } f_N|_{H^+(N)} = -Id, \\ SW(M , \mathfrak{s}_M) u^{b^+(N)},  & \text{if } \rho = 1 \text{ and } f_N|_{H^+(N)} = +Id, \\ 0, & \text{otherwise}. \end{cases}
\]
\end{proposition}

\begin{remark}
Let $f_M : M \to M$ be an orientation preserving smooth involution on $M$ which is not the identity map. Let $F$ denote the fixed point set of $f_M$. If $m \in F$ is a fixed point, then $f_M$ acts on $T_m M$ with eigenvalues $+1$ and $-1$. Since the orientation is preserved, the number of $-1$ eigenvalues must be $2$ or $4$. Thus in some basis $f_M$ acts either as $diag(1,1,-1,-1)$ or $diag(-1,-1,-1,-1)$. In the first case $m$ lies on a $2$-dimensional component of $F$ and we say $m$ is a non-isolated fixed point. In the second case $m$ is an isolated point of $F$ and we say $m$ is an isolated fixed point.\\

Now suppose that $f_N : N \to N$ is an orientation preserving smooth involution of $N$. We can form the equivariant connected sum $X = M \# N$ provided $f_M, f_N$ have fixed points of the same type (either both isolated or both non-isolated) and we write $f = f_M \# f_N$ for the resulting involution on $X$. It is important to note that the isomorphism class of the pair $(X , f)$ will in general depend on the choice of fixed points of $M,N$ used to perform the connected sum. However in the examples where we compute $SW^{\mathbb{Z}_2}_{\mathbb{Z}_2}(X , \mathfrak{s}_X , \rho)$, the result will not depend on the choice of fixed point, since the gluing formula is insensitive to this choice.
\end{remark}

To make use of Proposition \ref{prop:involution}, we look for some examples of pairs $(N , f_N)$ for which $w_{b^+(N)}( H^+(N) \otimes \mathbb{R}_\rho) \neq 0$. By the above remark it is also important to keep track of the dimensions of the connected components of the fixed point set of $f_N$. Our examples are listed in Figure \ref{fig:involution}. In this table we use some notation $r,\phi,\iota,e$ as we now explain. We let $r$ denote a reflection $r : S^2 \to S^2$, we let $\phi : S^2 \to S^2$ denote a rotation by $\pi$ around some axis and let $\iota : \overline{\mathbb{CP}}^2 \to \overline{\mathbb{CP}}^2$ be the map $\iota[z_0 , z_1 , z_2] = [z_0 , z_1 , -z_2]$. In the last three rows of the table we view $\mathbb{CP}^2 \#^3 \overline{\mathbb{CP}^2}$ as $(S^2 \times S^2) \#^2 \overline{\mathbb{CP}^2}$ and construct the involution $f_N$ as follows. In row $4$, $f_N$ is constructed as an equivariant connected sum. Note that the fixed point set of $\iota$ has components of dimension $0$ and $2$. We form the connected sum using only non-isolated fixed points. In rows $5$ and $6$, we start with an involution $i$ on $(S^2 \times S^2)$. Let $p$ be a non-fixed point of $i$. Then we attach a copy of $\overline{\mathbb{CP}}^2$ at $x$ and $i(x)$ and denote the resulting involution on the connected sum as $i'$. The last column of the table gives the characteristic of a Spin$^c$-structure $\mathfrak{s}_N$ satisfying $f_N(\mathfrak{s}_N) = (-1)^\rho \mathfrak{s}_N$ and $c_1(\mathfrak{s}_N)^2 = \sigma(N)$. In the last three rows we use the isomorphism  $H^2( \mathbb{CP}^2 \#^3 \overline{\mathbb{CP}}^2 , \mathbb{Z}) \cong H^2(S^2 \times S^2 , \mathbb{Z}) \oplus H^2( \mathbb{CP}^2 , \mathbb{Z}) \oplus H^2( \mathbb{CP}^2 , \mathbb{Z})$ and write $e$ for a generator of $H^2( \mathbb{CP}^2 , \mathbb{Z})$.

\begin{figure}\label{fig:involution}
\begin{equation*}
\renewcommand{\arraystretch}{1.4}
\begin{tabular}{|c|c|c|c|c|c|}
\hline
& $N$ & $f_N$ & dimensions of components of $F$ & $\rho$ & $c_1(\mathfrak{s}_N)$ \\
\hline
$1$& $S^2 \times S^2$ & $r \times r$ & $2$ & $0$ & $0$ \\
$2$ & $S^2 \times S^2$ & $\phi \times \phi$ & $0$ & $1$ & $0$ \\
$3$ & $S^2 \times S^2$ & $\phi \times id$ & $2$ & $1$ & $0$ \\
$4$ & $\mathbb{CP}^2 \#^3 \overline{\mathbb{CP}^2}$ & $(r \times r) \# \iota \# \iota$ & $0,2$ & $0$ & $(0,e,e)$ \\
$5$ & $\mathbb{CP}^2 \#^3 \overline{\mathbb{CP}^2}$ & $(\phi \times \phi)' $ & $0$ & $1$ & $(0,e,-e)$ \\
$6$ & $\mathbb{CP}^2 \#^3 \overline{\mathbb{CP}^2}$ & $(\phi \times id)'$ & $2$ & $1$ & $(0,e,-e)$ \\
\hline
\end{tabular}
\end{equation*}
\caption{Some pairs $(N , f_N)$ with $w_{b^+(N)}( H^+(N) \otimes \mathbb{R}_\rho) \neq 0$.}
\end{figure}

\begin{proposition}
Let $M$ be a smooth compact $4$-manifold with $b^+(M) > 1$ and let $f_M : M \to M$ be a smooth involution. Let $\mathfrak{s}_M$ be a Spin$^c$-structure with $d(M , \mathfrak{s}_M) = 0$. Suppose that one of the following holds:
\begin{itemize}
\item[(i)]{$f_M$ has an isolated fixed point and $f_M(\mathfrak{s}_M) = \mathfrak{s}_M$.}
\item[(ii)]{$f_M$ has a non-isolated fixed point and $f_M(\mathfrak{s}_M) = \mathfrak{s}_M$.}
\item[(iii)]{$f_M$ has an isolated fixed point and $f_M(\mathfrak{s}_M) = -\mathfrak{s}_M$.}
\item[(iv)]{$f_M$ has a non-isolated fixed point and $f_M(\mathfrak{s}_M) = -\mathfrak{s}_M$.}
\end{itemize}
Then for any integer $k \ge 1$:
\begin{itemize}
\item{In cases (ii),(iii) and (iv) the manifolds $X = M\#^k(S^2 \times S^2)$ and $X = M\#^k \mathbb{CP}^2 \#^{k+2} \overline{\mathbb{CP}^2}$ each admit a smooth involution $f_X$ and a Spin$^c$-structure $\mathfrak{s}_X$ such that $FSW^{\mathbb{Z}_2}_{\mathbb{Z}_2}(X , \mathfrak{s}_X , \rho) \in H^k( B\mathbb{Z}_2 , \mathbb{Z}_2)$ is defined, where $\rho = 0$ in case (ii), $\rho = 1$ in case (iii) or (iv) and
\[
FSW^{\mathbb{Z}_2}_{\mathbb{Z}_2}(X , \mathfrak{s}_X , \rho) = SW(M , \mathfrak{s}_M) u^k.
\]
}
\item{In case (i) the manifold $X = M\#^k \left( \mathbb{CP}^2 \#^3 \overline{\mathbb{CP}^2}\right)$ admits a smooth involution $f_X$ and a Spin$^c$-structure $\mathfrak{s}_X$ such that $FSW^{\mathbb{Z}_2}_{\mathbb{Z}_2}(X , \mathfrak{s}_X , \rho) \in H^k( B\mathbb{Z}_2 , \mathbb{Z}_2)$ is defined, where $\rho = 0$ and
\[
FSW^{\mathbb{Z}_2}_{\mathbb{Z}_2}(X , \mathfrak{s}_X , \rho) = SW(M , \mathfrak{s}_M) u^k.
\]
}
\end{itemize}
\end{proposition}
\begin{proof}
We apply Proposition \ref{prop:involution} where $(N , f_N , \mathfrak{s}_N)$ is an equivariant connected sum of $k$ summands, where each summand is one of the entries in Figure \ref{fig:involution}.
\begin{itemize}
\item{In case (i), we use $k$ copies of (4) as our summands.}
\item{In case (ii), we either use $k$ copies of (1) or $k-1$ copies of (1) and one copy of (4).}
\item{In case (iii), we either use $k$ copies of (2) or $k-1$ copies of (2) and one copy of (5).}
\item{In case (iv), we either use $k$ copies of (3) or $k-1$ copies of (3) and one copy of (6).}
\end{itemize}
\end{proof}

\begin{theorem}\label{thm:involutionpsc}
Let $M$ be a smooth compact $4$-manifold with $b^+(M) > 1$ and let $f_M : M \to M$ be a smooth involution. Let $\mathfrak{s}_M$ be a Spin$^c$-structure with $d(M , \mathfrak{s}_M) = 0$. Suppose that one of the following holds:
\begin{itemize}
\item[(i)]{$f_M$ has an isolated fixed point and $f_M(\mathfrak{s}_M) = \mathfrak{s}_M$.}
\item[(ii)]{$f_M$ has a non-isolated fixed point and $f_M(\mathfrak{s}_M) = \mathfrak{s}_M$.}
\item[(iii)]{$f_M$ has an isolated fixed point and $f_M(\mathfrak{s}_M) = -\mathfrak{s}_M$.}
\item[(iv)]{$f_M$ has a non-isolated fixed point and $f_M(\mathfrak{s}_M) = -\mathfrak{s}_M$.}
\end{itemize}
Suppose also that $SW(M , \mathfrak{s}_M)$ is odd. Then for any integer $k \ge 1$:
\begin{itemize}
\item{In cases (ii),(iii) and (iv) the manifolds $X = M\#^k(S^2 \times S^2)$ and $X = M\#^k \mathbb{CP}^2 \#^{k+2} \overline{\mathbb{CP}^2}$ each admit a smooth involution $f_X$ and a Spin$^c$-structure $\mathfrak{s}_X$ such that $FSW^{\mathbb{Z}_2}_{\mathbb{Z}_2}(X , \mathfrak{s}_X , \rho)$ is non-zero. Hence $X$ does not admit an $f_X$-invariant metric of positive scalar curvature.}
\item{In cases (i), (ii),(iii) and (iv) the manifold $X = M\#^k \left( \mathbb{CP}^2 \#^3 \overline{\mathbb{CP}^2} \right)$ admits a smooth involution $f_X$ and a Spin$^c$-structure $\mathfrak{s}_X$ such that $FSW^{\mathbb{Z}_2}_{\mathbb{Z}_2}(X , \mathfrak{s}_X , \rho)$ is non-zero. Hence $X$ does not admit an $f_X$-invariant metric of positive scalar curvature.}
\end{itemize}
\end{theorem}

\begin{corollary}
For every $n \ge 4$ and $m \ge n+18$, the $4$-manifold $X = \#^{n} \mathbb{CP}^2 \#^{m} \overline{\mathbb{CP}}^2$ admits a metric of positive scalar curvature and a smooth involution $f : X \to X$ such that there does not exist an $f$-invariant metric of positive scalar curvature on $X$.
\end{corollary}
\begin{remark}
In the course of the proof we construct an explicit example of such an involution $f$. Many more examples could be produced by varying the choice of the pair $(M , f_M)$ used below in the proof. A similar result to this was obtained by LeBrun in \cite{lebr}. However, the proof in \cite{lebr} uses a covering space argument which can only be used to construct examples where the involution $f$ acts freely (in particular, the Euler characteristic must be even in such examples).  
\end{remark}
\begin{proof}
That $X$ has a metric of positive scalar curvature is clear. Let $M_0$ be a $K3$ surface and let $M = M_0 \#^{m-n-18} \overline{\mathbb{CP}}^2$. Then $M\#^{n-3} \mathbb{CP}^2 \#^{n-1} \overline{\mathbb{CP}^2}$ is diffeomorphic to $X$. We will find an involution $f_M$ on $M$ which anti-preserves a Spin$^c$-structure $\mathfrak{s}_M$ with $d(M , \mathfrak{s}_M) = 0$, $SW(M , \mathfrak{s}_M) = 1 \; ({\rm mod} \; 2)$ and such that $f_M$ has non-isolated fixed points.\\

Let us view $M_0$ as the quartic hypersurface in $\mathbb{CP}^3$ defined by $z_1^4 - z_2^4 - z_3^4 - z_4^4 = 0$. Clearly complex conjugation $[z_1,z_2,z_3,z_4] \mapsto [\overline{z_1} , \overline{z_2} , \overline{z_3} , \overline{z_4} ]$ defines an orientation preserving involution on $M_0$ with non-isolated fixed points. Since $M_0$ is spin we can take $\mathfrak{s}_{M_0}$ to be the unique spin structure, then $\mathfrak{s}_{M_0}$ is anti-preserved by $f_{M_0}$. Next, we blow $M_0$ up at $m-n-18$ non-isolated fixed points. Then $f_{M_0}$ extends to an involution $f_M$ on the blowup $M = M_0 \#^{m-n-18} \overline{\mathbb{CP}}^2$. The blowup can be viewed as an equivariant connected sum in the following way. Each copy of $\overline{\mathbb{CP}}^2$ is equipped with the involution given by complex conjugation. This acts as $-1$ on $H^2( \overline{\mathbb{CP}}^2 , \mathbb{Z})$. Let $\mathfrak{s}_{\overline{\mathbb{CP}}^2}$ be the Spin$^c$-structure whose characteristic is a generator of $H^2(\overline{\mathbb{CP}}^2 , \mathbb{Z})$. Then we take our Spin$^c$-structure $\mathfrak{s}_M$ on $M$ to be $\mathfrak{s}_{M_0}$ summed with a copy of $\mathfrak{s}_{\overline{\mathbb{CP}}^2}$ for each copy of $\overline{\mathbb{CP}}^2$. Then $\mathfrak{s}_{M}$ is clearly anti-preserved by $f_M$. By the blowup formula for Seiberg-Witten invariants we have $SW(M ,\mathfrak{s}_M) = SW(M_0 , \mathfrak{s}_{M_0}) = 1 \; ({\rm mod} \; 2)$. To complete the proof, we apply Theorem \ref{thm:involutionpsc} (iv) with $k=n-3$.
\end{proof}


\bibliographystyle{amsplain}

\end{document}